\documentclass{article}
\usepackage[top = 1.5cm, left = 1.5cm, right = 1.5cm]{geometry}
\usepackage{fullpage}

\usepackage{amsmath}
\usepackage[utf8]{inputenc} 
\usepackage[T1]{fontenc}    
\usepackage{hyperref}       
\usepackage{url}            
\usepackage{booktabs}       
\usepackage{amsfonts}       
\usepackage{nicefrac}       
\usepackage{microtype}      
\usepackage{natbib}
\setcitestyle{square,numbers,sort&compress}
\usepackage{graphicx}
\usepackage{amssymb,color,amsthm}
\usepackage{subfigure}
\usepackage{bbm}
\usepackage{etoolbox}
\usepackage[dvipsnames]{xcolor}
\usepackage{algorithm}
\usepackage{algorithmicx}
\usepackage{algpseudocode}
\usepackage{mathrsfs}
\usepackage{flushend}
\usepackage{wrapfig}
\usepackage{lipsum}
\def\remark{\addtocounter{remark}{1}\def\@currentlabel{\theremark}%
\emph{Remark~\theremark}. } \makeatother
\newcounter{remark}

\def\b1{\mathbf 1}

\def\be{\mathbf e}
\def\bb{\mathbf b}
\def\bd{\mathbf d}
\newcommand{\st}{{\rm s.t.}}
\def\bx{\mathbf x}

\def\bq{\mathbf q}
\def\bw{\mathbf w}

\def\by{\mathbf y}
\def\bu{\mathbf u}

\def\bv{\mathbf v}

\def\bz{\mathbf z}

\def\bA{\mathbf A}

\def\bH{\mathbf H}
\def\bI{\mathbf I}

\def\bM{\mathbf M}

\def\bP{\mathbf P}

\def\bW{\mathbf W}
\def\bX{\mathbf X}
\def\bY{\mathbf Y}
\def\cA{\mathcal A}
\def\bh{\mathbf h}

\def\bmu{\boldsymbol \mu}

\newcommand{\T}{\scriptscriptstyle T}

\def\minimize{\mathop{\text{minimize}}}

\def\sS{\mathscr{S}}

\def\sF{\mathscr{F}}
\def\sR{\mathscr{R}}
\def\sT{\mathscr{T}}
\def\hf{\widehat{f}}
\def\hq{\widehat{q}}

\def\leref#1{Lemma~\ref{#1}}
\def\prref#1{Proposition~\ref{#1}}
\def\thref#1{Theorem~\ref{#1}}
\def\coref#1{Corollary~\ref{#1}}
\def\figref#1{Figure~\ref{#1}}

\def\algref#1{Algorithm~\ref{#1}}
\def\asref#1{Assumption~\ref{#1}}
\def\conref#1{Condition~\ref{#1}}
\def\bydef{\triangleq}
\newtheorem{proposition}{Proposition}
\newtheorem{lemma}{Lemma}
\newtheorem{theorem}{Theorem}
\newtheorem{corollary}{Corollary}
\newtheorem{definition}{Definition}
\newtheorem{assumption}{Assumption}
\newtheorem{condition}{Condition}
\newcommand{\overbar}[1]{%
  \mkern 1.5mu\overline{\mkern-1.5mu#1\mkern-1.5mu}\mkern 1.5mu}

\usepackage[suppress]{color-edits}
\addauthor{mh}{blue}
\addauthor{sl}{Orange}

\usepackage{lipsum}

\newcommand\blfootnote[1]{%
  \begingroup
  \renewcommand\thefootnote{}\footnote{#1}%
  \addtocounter{footnote}{-1}%
  \endgroup
}

\title{SNAP: Finding Approximate Second-Order Stationary Solutions Efficiently for Non-convex Linearly Constrained Problems 
}

\author{
\and  \\ \and 
Songtao Lu \thanks{Department of Electrical and Computer Engineering, University of Minnesota -- Twin Cities}\;
\\
\texttt{\small lus@umn.edu}
\and
Meisam~Razaviyayn \thanks{Department of Industrial \& Systems Engineering, University of Southern California}\;
\\
\texttt{\small razaviya@usc.edu}
\and
Bo Yang \footnotemark[1]
\\
\texttt{\small yang4173@umn.edu}
 \and \\ \and 
Kejun~Huang \thanks{Department of Computer \& Information Science \& Engineering, University of Florida}\;
\\
\texttt{\small kejun.huang@ufl.edu}
\and
Mingyi Hong \footnotemark[1]
\\
\texttt{\small mhong@umn.edu}
\blfootnote{S. Lu and M. Hong are supported in part by a NSF grant CMMI-1727757, and an AFOSR grant 15RT0767.}
}

\begin{document}
\maketitle
\begin{abstract}
This paper proposes low-complexity algorithms for finding approximate second-order stationary points (SOSPs) of problems with  smooth non-convex objective and linear constraints. While finding (approximate) SOSPs is computationally intractable, we first show that  \textit{generic} instances of the problem can be solved efficiently. More specifically, for a {\it generic} problem instance, certain  {\it strict complementarity} (SC) condition holds for all  Karush–Kuhn–Tucker (KKT) solutions (with probability one). The SC  condition is then used to establish an equivalence relationship between two different notions of SOSPs, one of which is computationally easy to verify. Based on this particular notion of SOSP, we design an algorithm named the Successive Negative-curvature grAdient Projection (SNAP), which successively performs either conventional gradient projection or some negative curvature based projection steps to find SOSPs. SNAP and its first-order extension SNAP$^+$, require $\mathcal{O}(1/\epsilon^{2.5})$ iterations to compute an $(\epsilon, \sqrt{\epsilon})$-SOSP, and their per-iteration  computational complexities are  {\it polynomial} in the number of constraints and problem dimension. To our knowledge, this is the first time that first-order algorithms with polynomial per-iteration complexity and global sublinear rate have been designed to find SOSPs of the important class of non-convex problems with linear constraints.
\end{abstract}

\section{Introduction}
\label{sec:intro}
We consider the following class of non-convex linearly constrained optimization problems
\begin{equation}\label{eq.pro}
\minimize_{\bx}\quad f(\bx), \quad \bx\in \mathcal{X}\bydef\{\mathbf{x}\mid \mathbf{A} \bx\le \mathbf{b}\}
\end{equation}
where  $f:\mathbb{R}^{d}\rightarrow\mathbb{R}$ is twice differentiable (possibly non-convex); $\bA\in\mathbb{R}^{m\times d}$, and $\bb\in\mathbb{R}^{m}$ are some given matrix and vector. Such a class of problems finds many applications in machine learning and data science. For example, in the  nonnegative matrix factorization (NMF) problem,
a given data matrix $\bM\in\mathbb{R}^{n\times m}$ is to be factorized into two nonnegative matrices $\bW\in\mathbb{R}^{n\times k}$ and $\bH\in\mathbb{R}^{m\times k}$ such that $\|\bW\bH^{\T}-\bM\|^2_F$ is minimized \cite{lee1999learning}. It is also of interest to consider the symmetric case $\min_{\bX\ge 0}\|\bW\bW^{\T}-\bM\|^2_F$ where $\bM\in\mathbb{R}^{n\times n}$. Further, for non-convex problems with $\ell_1$ regularizers (such as sparse PCA),
we need to solve $\min\; g(\bx) +\|\bx\|_1$, which can be equivalently written as
\begin{align}
    \min\; g(\bx) +\mathbf{1}^{\T}\by,\quad  \mbox{s.t.}\; -\by \le \bx\le \by,
\end{align}
which is of the form in \eqref{eq.pro}.

Recently, algorithms that escape strict saddle points (stationary points that have negative curvatures) for unconstrained non-convex problems attracted significant research efforts. This is partly due to the recent discoveries that, for certain non-convex unconstrained problems, their optimization landscape is nice, in the sense that the stationary points are either global minimum or strict saddle points (e.g., shallow neural network training \cite{kawaguchi2016deep,soja18}),  or all saddle points are \emph{strict} saddle points (e.g., tensor decomposition \cite{ge2015escaping}, phase retrieval \cite{jusun17}, low-rank matrix factorization \cite{rong17}, etc.). Therefore, escaping (strict) saddle points guarantees convergence to either local or even global optima.

A natural question then arises: what if the considered problem has some simple constraints or non-smooth regularizers? After all, there are many machine learning and data sciences problems of this type. It would seem to be straightforward to extend the previous ``saddle-point escaping'' algorithms to these setting, just like we can extend  algorithms that can reach unconstrained first-order stationary solutions to constrained problems. Unfortunately, this is not the case. As will be seen shortly, even {\it checking} the second-order stationary solution for linearly constrained problems could be daunting. The main task of this paper is then to identify situations in which finding second-order stationary solution for problem \eqref{eq.pro} is easy, and design efficient algorithms for this task.

\noindent{\bf Related work}. For unconstrained smooth problems, there has been a line of work that develops the algorithm by using both the gradient direction and negative curvature  so that second-order stationary points (SOSPs) can be attained within a certain number of iterations \cite{royer2018complexity,carmon2018accelerated,agarwal2017finding}. For example, a Hessian-free algorithm \cite{carmon2018accelerated} is guaranteed to provably converge to SOSPs within a certain number of iterations. By exploiting the eigen-decomposition, the convergence rate of some variant of the Newton method to SOSPs has been shown in \cite{paternain2019newton,Royer2019}, where the algorithm is assumed to be able to access the Hessian matrix around strict saddle points \cite{jlee16jordan}. Another way of finding negative curvature is to occasionally add noise to the iterates. A perturbed version of gradient descent (GD) was firstly proposed in \cite{jin2017jordan}, which shows that the convergence rate of perturbed GD to SOSPs is provably faster than the ordinary gradient descent algorithm with random initializations. In a follow-up work of the perturbation technique \cite{xu2017first}, the authors proposed NEgative-curvature-Originated-from-Noise (NEON), and illustrated that the perturbed gradient descent is essentially implementing the power method around the saddle point so that a negative curvature of the Hessian matrix is extracted without performing the eigenvalue decomposition. Other recent works include generalizations of NEON such as NEON$^+$ ~\cite{xu2017first} 
NEON2  \cite{allen2018neon2} 
and  perturbed alternating gradient descent proposed in \cite{lu2018sublinear} for block structured non-convex problems. In practice, there may be some constraints, such as equality and inequality constraints. For the equality constraint, negative curvature method has been proposed \cite{Goldfarb2017,hong172ndorder} so that SOSP can be obtained asymptotically as the algorithm proceeds.

Despite these recent exciting developments, the aforementioned methods are unable to incorporate even the simpliest linear {\it inequality} constraints. In practical machine learning problems, however, inequality constraints are ubiquitous due to physical constraints. Examples include neural networks training with the nonnegative constraint \cite{chzu15}, NMF \cite{lee1999learning}, nonnegative tensor factorization \cite{sidiropoulos2017tensor}, \mhdelete{resource allocation in wireless networks \cite{luwa18}, image restoration, } non-convex quadratic programming with box constraints (QPB) \cite{burer2009nonconvex}, to name just a few. Existing work either directly rely on~\emph{second-order} descent directions of the objective function  \cite{andreani2010second,royer2018complexity}, or use this information together with other methods  such as the trust region method \cite{conn2000trust}, cubic regularized Newton's method \cite{cartis2018second,}, primal-dual algorithm \cite{di2005convergence}, etc. These algorithms are generally unfavorable for large-scale problems due to the scalability issues when computing the second-order information. However, to the best of our knowledge, there has been no first-order methods that can provably compute SOSPs for linearly constrained problem \eqref{eq.pro}.

An even more challenging issue is that finding SOSP for general linearly constrained non-convex problems is NP-hard. Indeed, it has been shown in \cite{mei18} that even obtaining the approximate SOSPs is hard in terms of both the total number of constraints and the inverse of the desired second-order optimality error. So existing methods for finding SOSPs with global convergence rate all require some exponential complexity (exponential in the total number of constraints); see \cite{mei18,aras18}.

\noindent{\bf Contributions of this work}. We first  introduce two notions of (approximate) SOSPs for problem \eqref{eq.pro}, one based on identifying the active constraints at a given solution (referred to as SOSP1), and one based on the feasible directions orthogonal to the gradient (referred to as SOSP2). In particular, we show that, these two conditions become {\it equivalent} when certain (provably mild) strict complementarity (SC) conditions are satisfied. Such equivalence conditions enable us to design
an algorithm by exploiting the active sets of the solution path, which is computationally much simpler compared with existing methods based on checking feasible directions orthogonal to gradient. Then we propose a Successive Negative-curvature grAdient Projection (SNAP) algorithm, which can find second-order solutions with high probability. The algorithm updates the iterates by successively using either gradient projection step, or certain negative-curvature projection step (with appropriate active constraints based line-search procedures). Further, we extend SNAP by proposing a first-order algorithm (abbreviated as SNAP$^+$) which utilizes gradient steps to extract negative curvatures.
Numerical simulations demonstrate that the proposed algorithm efficiently converges to SOSPs.

The main contributions of this work are summarized as follows:

1) We study problem \eqref{eq.pro} and analyze the equivalence of two different notions of (approximate) SOSPs under the assumption of strict complementarity. This part of work provides new insights of solution structures, and will be useful in subsequent algorithm design.

2)  We propose the SNAP algorithm, which computes some approximate $(\epsilon_G,\epsilon_H)$-SOSP with  $\mathcal{O}(\max\{1/\epsilon^2_G,1/\epsilon^3_H\})$ iterations, and with polynomial computational complexity in dimensions $(d,m)$ as long as projection, gradient, and Hessian can be computed efficiently.

3) We extend SNAP to SNAP$^+$, an  algorithm that only uses the gradient of the objective function, while being able to compute $(\epsilon,\sqrt{\epsilon})$-SOSPs for problem \eqref{eq.pro} within  $\mathcal{O}(1/\epsilon^{2.5})$ iterations. Each iteration of the proposed algorithm only requires simple vector operations and  projections to the feasible set which can be done in polynomial iterations under reasonable oracles~\cite{nemirovski1995information}. This makes the proposed algorithm amenable for large scale optimization problems. To the best of our knowledge, this is the first first-order method that is capable of achieving the above stated properties.

\noindent{\bf Notation.} Bold lower case characters, e.g., $\bx$ represents vectors and bold capital ones, e.g., $\bX$ denotes a matrix, $\bx_i$ or $(\bA\bx)_i$ denotes the $i$th entry of vector $\bx$ or $\bA\bx$ where $\bA\in\mathbb{R}^{m\times d}$ and $\bx\in\mathbb{R}^d$. $\bA^{\dag}$ denotes the pseudo inverse of matrix $\bA$, and $\|\bA\|$ denotes the spectral norm of $\bA$.
\section{Preliminaries}

\mhdelete{\begin{definition}
A differentiable function $f(\cdot)$ is $L_1$-gradient Lipschitz
if
$
\|\nabla f(\bx)-\nabla f(\by)\|\le L_1\|\bx-\by\|,\quad\forall \bx,\by\in\mathcal{X}
$
where $\mathcal{X}$ denotes the feasible set.
\end{definition}
\begin{definition}
 A twice-differentiable function $f(\cdot)$ is $L_2$-Hessian Lipschitz  if
$
\|\nabla^2f(\bx)-\nabla^2f(\by)\|\le L_2\|\bx-\by\|,\quad\forall \bx,\by\in\mathcal{X}.
$
\end{definition}
}

We make the following assumption on the objective function of ~\eqref{eq.pro}. 
\begin{assumption}\label{as1}
$f(\bx)$ in \eqref{eq.pro} is $L_1$-gradient Lipschitz and $L_2$-Hessian Lipschitz, i.e.,
\begin{align}
    \|\nabla f(\bx)-\nabla f(\by)\|\le L_1\|\bx-\by\|,\; \|\nabla^2f(\bx)-\nabla^2f(\by)\|\le L_2\|\bx-\by\|,\quad\forall \bx,\by\in\mathcal{X}.
\end{align}
\end{assumption}
Let $\mathcal{A}(\bx)=\{j \mid \bA_j\bx=\bb_j, \forall j\in [m]\}$ denote the active set at a given point $\bx$, where $\bA_j$ denotes the $j$th row of matrix $\bA$ and $\bb_j$ denotes the $j$th entry of $\bb$. Define $\overbar{\mathcal{A}(\bx)}$ to be the complement of the  set $\mathcal{A}(\bx)$, i.e.,
\begin{equation}
\overbar{\mathcal{A}(\bx)} \cup \mathcal{A}(\bx) = [m], \;\quad \overbar{\mathcal{A}(\bx)} \cap \mathcal{A}(\bx) = \emptyset.\label{eq.defofm}
\end{equation}
Concatenating the coefficients of the {\it active}  constraints at $\bx$, we define the matrix $\bA'(\bx)$ as
\begin{align}\label{eq:A:prime}
\bA'(\bx)\triangleq \left[\begin{array}{lll} \ldots & \bA_j & \ldots\end{array}\right]^{\T}\in\mathbb{R}^{|\cA(\bx)|\times d},\quad \forall j\in\mathcal{A}(\bx).
\end{align}
In other words, $\bA'(\bx)$ is a submatrix of $\bA$ containing the rows of $\bA$ corresponding to the active set. Similarly, we can define $\bb'(\bx)\in\mathbb{R}^{|\cA(\bx)|\times 1}$ by concatenating the entries of $\bb$ corresponding to the active set of constraints.
At a given point $\mathbf{x}$, we define the projection onto the space spanned by the inactive constraints as
\begin{equation}\label{eq:P}
\pi_{\mathcal{A}}(\bx)\bydef\bP(\bx)\bx,\quad \textrm{where}\quad \bP(\bx)\bydef\left(\bI-(\bA'(\bx))^{\T}\left(\bA'(\bx)(\bA'(\bx))^{\T}\right)^{\dag}\bA'(\bx)\right)\in\mathbb{R}^{d\times d}.
\end{equation}
Here, $\bP(\bx)$ represents the projection matrix to the null space of $\bA'(\bx)$.  Define $\bP_{\perp}(\bx)$ as the projector to the column space of $\bA'(\bx)$.
Similarly, let us define
\vspace{-4px}
\begin{align}\label{eq.compq}
q_{\pi}(\bx)\bydef \pi_{\mathcal{A}}(\nabla f(\bx)) = \bP(\bx)\nabla f(\bx).
\end{align}

To measure the first-order optimality of a given  point, we first define the  {\it proximal gradient}
\begin{equation}\label{eq:projection}
g_\pi(\bx)\bydef1/\alpha(\pi_{\mathcal{X}}\left(\bx-\alpha\nabla f(\bx)\right)-\bx),\;\; \mbox{with}\;\;\pi_{\mathcal{X}}(\bv) \triangleq \arg\min_{\bw\in \mathcal{X}}\|\bw -\bv\|^2,
\end{equation}
where $\alpha>0$ is a fixed given constant and $\pi_{\mathcal{X}}$ denotes the projection operator onto the feasible set. Then $\| g_\pi(\bx)\|$ can be used to define the {\it first-order optimality gap} for a given point $\bx\in\mathcal{X}$ \mhdelete{the size of the proximal gradient, expressed as}.

To define the second-order optimality gap, let us start by stating the popular  {\it exact} second-order necessary conditions for local minimum points of constrained optimization  \cite[Proposition 3.3.1]{bertsekas99}.
\begin{proposition}
\cite[Proposition 3.3.1]{bertsekas99} If $\bx^*\in\mathcal{X}$ is a local minimum  of \eqref{eq.pro}, then
\begin{equation}
\|g_{\pi}(\bx^*)\| =0, \quad
\by^{\T}\nabla^2 f(\bx^*)\by\ge 0 ,\quad \forall~ \by~\mbox{satisfying}~~\bA'(\bx^*)\by=0,
\end{equation}
where $\bA'(\bx^*)$, as defined in \eqref{eq:A:prime}, is a matrix collecting all active constraints.
\end{proposition}

This proposition leads to the following form of {\it exact} SOSP.

\begin{definition} [Exact SOSP1]
The point $\bx^*\in \mathcal{X}$ is a second-order stationary point of~\eqref{eq.pro} if
\begin{subequations}\label{eq.cond1:exact}
\begin{align}
&\|g_{\pi}(\bx^*)\| =0,  & \quad\textrm{(first-order condition)}\label{eq.cond11:exact}
\\
&\by^{\T}\nabla^2 f(\bx^*)\by\ge 0 ,\quad \forall~ \by~\mbox{satisfying}~~\bA'(\bx^*)\by=0, & \quad\quad \textrm{(second-order condition)}\quad \label{eq.cond12:exact}
\end{align}
\end{subequations}
where $\bA'(\bx^*)$, defined  in \eqref{eq:A:prime},  is a matrix that collects the active constraints.
\end{definition}
Similarly, we define the following  {\it approximate} SOSP condition for problem~\eqref{eq.pro} as:
\begin{definition}[$(\epsilon_G,\epsilon_H)$-SOSP1]
A point $\bx^*\in\mathcal{X}$ is an $(\epsilon_G,\epsilon_H)$-SOSP point of problem \eqref{eq.pro} if
\begin{subequations}\label{eq.cond1}
\begin{align}
&\|g_{\pi}(\bx^*)\|\le\epsilon_G, & \textrm{(approx. first-order condition)}\label{eq.cond11}
\\
&\by^{\T}\nabla^2 f(\bx^*)\by\ge-\epsilon_H,\quad \forall~ \by~\mbox{satisfying}~~\bA'(\bx^*)\by=0, &  \textrm{(approx. second-order condition)}\label{eq.cond12}
\end{align}
\end{subequations}
where $\epsilon_G,\epsilon_H>0$ are some given small constants.
\end{definition}
By utilizing the definition of the null space of the active set in \eqref{eq:P},  we can rewrite condition \eqref{eq.cond12} as
\begin{equation}\label{eq:Hessian:small}
\lambda_{\min}(\bH_{\bP}(\bx^*))\ge -\epsilon_H, \quad \mbox{with}\quad \bH_{\bP}(\bx^*):={\bP(\bx^*)\nabla^2 f(\bx^*)\bP(\bx^*)},
\end{equation}
where $\lambda_{\min}(\cdot)$ is the operator that returns the smallest eigenvalue of a matrix. The above definition of second-order solutions leads to the following definition of first-order stationary solutions.
\begin{definition}[$\epsilon_G$-FOSP1]
If a point $\bx^*\in\mathcal{X}$  satisfies the condition $\|g_{\pi}(\bx^*)\|\le\epsilon_G$, we call it an $\epsilon_G$-first-order stationary point  \textit{of the first kind}, abbreviated as $\epsilon_G$-FOSP1.
\vspace{-0.2cm}
\end{definition}
Note that the conditions in \eqref{eq.cond1:exact} and  \eqref{eq.cond1} are necessary conditions for $\bx^*$ being a local minimum solutions. There are, of course, many other necessary conditions of this kind. Therefore, to distinguish from various solution concepts, we will refer to the solutions satisfying the above conditions as SOSP \textit{of the first kind} and $(\epsilon_G,\epsilon_H)$-SOSP \textit{of the first kind}, abbreviated as SOSP1 and $(\epsilon_G,\epsilon_H)$-SOSP1, respectively. Below we present another popular second-order solution concept, which appears in optimization literature; see \cite{aras18,mei18}, and the references therein.
\begin{definition}[$({\epsilon}_G,{\epsilon}_H)$-SOSP2]
A point $\bx^*\in \mathcal{X}$ is an $({\epsilon}_G,{\epsilon}_H)$-second-order stationary point of the second kind of problem \eqref{eq.pro} if the following conditions are satisfied:
\begin{subequations}\label{eq.cond2}
\begin{align}
&\nabla f(\bx^*)^{\T}(\bx-\bx^*)\ge-{\epsilon}_G,\quad\forall \bx\in\mathcal{X},\quad\st\quad\|\bx-\bx^*\|\le1\label{eq.cond21}
\\
&(\bx-\bx^*)^{\T}\nabla^2 f(\bx^*)(\bx-\bx^*)\ge-{\epsilon}_H,\quad \forall \bx\in\mathcal{X} \quad\st \quad\nabla f(\bx^*)^{\T}(\bx-\bx^*)=0.\label{eq.cond22}
\end{align}
\end{subequations}
We refer to the above conditions as  $({\epsilon}_G,{\epsilon}_H)$-second order stationary solution of the second kind, abbreviated as $({\epsilon}_G,{\epsilon}_H)$-SOSP2.
\end{definition}

\begin{definition}[$\epsilon_G$-FOSP2]
{If a solution $\bx^*\in\mathcal{X}$ only satisfies \eqref{eq.cond21}, we call it an $\epsilon_G$-first-order stationary point {of the second kind (FOSP2)}.}
\end{definition}

The classical result~\cite{murty1987some} shows that checking $(\epsilon_G,\epsilon_H)$-SOSP2 for \eqref{eq.pro} is NP-hard in the problem dimension and in $\log(1/\epsilon_H)$ even for the class of quadratic functions. This hardness result has recently been strengthened by showing NP-hardness in terms of problem dimension and in $1/\epsilon_H$  \cite{mei18}. On the other hand, checking $(\epsilon_G,\epsilon_H)$-SOSP1 condition only requires projection onto linear subspaces and finding minimum eigenvalues. A natural question then arises:
How do these different kinds of {\it approximate} and {\it exact} second-order solution concepts relate to each other? In what follows, we provide a concrete answer to this  question. This answer relies on a critical concept called {\it strict complementarity}, which will be introduced below.

\begin{definition}
A give primal-dual pair $({\bx^*},{\bmu^*})$ for the linearly constrained problem \eqref{eq.pro} satisfies the Karush–Kuhn–Tucker (KKT) condition with strict complementarity  if
\begin{subequations}\label{eq:complementarity}
\begin{align}
&\nabla f({\bx^*})+\sum^m_{j=1}\bmu^*_j\bA_j=0, \label{eq.condk1}
\\
& \mbox{for each $j\in[m]$, either}~\bmu^*_j>0,\;\bA_j{\bx^*}=\bb_j \quad \textrm{or}\quad \bmu^*_j=0,\; \bA_j\bx^*<\bb_j\label{eq.condk2}.
\end{align}
\end{subequations}
\end{definition}

{Note that the SC condition has been assumed and used in convergence analysis of many algorithms, e.g., trust region algorithms for non-convex optimization with bound constraints  in \cite{conn89,lin1999newton,lescrenier1991convergence}.}

 The results below extend a recent result in \cite[Proposition 2.3]{jiaw19}, which shows that SC is satisfied for box constrained non-convex problems (with high probability). See Appendix \ref{sec:slack1} --  \ref{sec:slack2} for proof. 

\begin{proposition}\label{pr:slack1}
Suppose $f(\bx)=g(\bx)+\bq^{\T}\bx$ in problem~\eqref{eq.pro} where $g(\bx)$ is differentiable. Let $\bx^*$ be a KKT point of problem \eqref{eq.pro}.
If vector $\bq$ is generated from a continuous measure, and if the set $\{\bA_j\mid j\in\mathcal{A}(\bx^*)\}$ are linearly independent, then the SC condition holds with probability one.
\end{proposition}
\begin{corollary}\label{co:slack2}
Suppose $f(\bx)=g(\bx)+\bq^{\T}\bx$ and $g(\bx)$ is differentiable. If the data vector $(\bq,\bb)$ is generated from a continuous measure, then the SC  condition holds for \eqref{eq.pro} with probability one.
\end{corollary}

This result shows that for a certain generic choice of objective functions, the SC condition is satisfied. As we will see in the next section,  this SC condition leads to the equivalence of SOSP1 and SOSP2 conditions. Hence, instead of working with the computationally intractable SOSP2 condition, we can use a tractable SOSP1 condition for developing algorithms. In other words, by adding a small random linear perturbation to the objective function, which does not practically change the landscape of the optimization problem, we can avoid the computational intractability of SOSP2.

\section{Almost Sure Equivalence of SOSP1 and SOSP2}
To understand the relation between SOSP1 and SOSP2, let us consider the following example.
{\noindent{\it Example 1.} } Consider the following box constrained quadratic problem:
\begin{equation}
\minimize_{x_1,x_2}\quad -x^2_1-x^2_2,\quad\textrm{s.t.}\quad0\le x_1\le 1,\quad 0\le x_2\le 1.
\end{equation}
Clearly the point $\bx^* = (0,0)$ is an SOSP1. This is because the gradient of the objective is zero, both inequality constraints are active at this point and  $\dim(\by)=0$ in~\eqref{eq.cond1}. However, the point $\bx^*=(0,0)$ is {\it not} an SOSP2 according to the definition in~\eqref{eq.cond2}.  This is because the condition~$\nabla f(\bx^*)^{\T}(\bx-\bx^*)=0$ is  true for all feasible $\bx $, but $(\bx-\bx^*)^{\T}\nabla^2 f(\bx^*)(\bx-\bx^*) = -2$ for $\bx = (1,1)$. \hfill $\blacksquare$

The above example suggests that  condition \eqref{eq.cond2} is {\it stronger} than \eqref{eq.cond1}, even when  $\epsilon_H =0, \epsilon_G=0$.
Indeed, one can show that any point $\bx^*$ satisfying \eqref{eq.cond2} also satisfies \eqref{eq.cond1}. More importantly, these conditions become {\it equivalent} when the SC condition \eqref{eq:complementarity} holds true. These results are presented in~\ref{sec:eq00} -- \ref{sec:continuity} and is summarized in Proposition~\ref{pr:00} below.
\begin{proposition}\label{pr:00}
Suppose that every KKT solution $(\bx^*, \bmu^*)$ of problem~\eqref{eq.pro} satisfies the SC condition \eqref{eq:complementarity}. Then  $(0,0)$-SOSP1  in  \eqref{eq.cond1},  and the $(0,0)$-SOSP2 in \eqref{eq.cond2} are equivalent in the sense that for any $\bx^*\in \mathcal{X}$, if  it is  a $(0,0)$-SOSP1, then it is also a $(0,0)$-SOSP2 solution, and vice versa.
\end{proposition}

{In view of {\it Example 1}, the above equivalence result is somewhat surprising.
However, by applying Proposition \ref{pr:slack1}, one can slightly perturb the problem in Example 1 by adding to its objective a random linear term in the form of $\bq^T \bx$ (with $\|\bq\|$ being very small) to satisfy the SC condition. One can check that after this perturbation, the two conditions become the same.}

{\noindent} Next we proceed by analyzing the {\it approximate} second-order conditions of SOSP1 and SOSP2. 

\begin{corollary}\label{co:0eh}
The second-order conditions \eqref{eq.cond12} and \eqref{eq.cond22} are equivalent in the following sense: suppose the SC condition \eqref{eq:complementarity} holds, then any $(0,{\epsilon}_H)$-SOSP2 
must satisfy $(0,\epsilon_H)$-SOSP1, and vice versa.
\end{corollary}
Although at this point we have not shown the equivalence of $(\epsilon_G,\epsilon_H)$-SOSP1 and $(\epsilon_G,\epsilon_H)$-SOSP2 (a result that remains very challenging), we provide an alternative result showing that $(\epsilon_G,\epsilon_H)$-SOSP1 is a valid approximation of $(0,0)$-SOSP1, which in turn is equivalent to $(0,0)$-SOSP2 by Proposition \ref{pr:00}.
In particular, we show that
$(\epsilon_G,\epsilon_H)$-SOSP1 becomes $(0,0)$-SOSP1 as $(\epsilon_G,\epsilon_H)\to 0$.
\begin{proposition}\label{pr:continuity}
 Let $\{\bx^{(r)}\}_{r=1}^{\infty}$ be a sequence generated by an algorithm. Assume for each $r$, the point $\bx^{(r)}$ be an $(\epsilon_G^{(r)},\epsilon_H^{(r)})$-SOSP1. Assume further that $\{(\epsilon_G^{(r)},\epsilon_H^{(r)})\}$ converges to the point $(0,0)$. Then, any limit point of the sequence $\{\bx^{(r)}\}_{r=1}^{\infty}$ is an exact SOSP1.
\end{proposition}

While in general SOSP2 is stronger than SOSP1, using SOSP1 has the following advantages:\\

\noindent{\bf 1)} For a given $\bx$, checking whether SOSP1 holds  is computationally tractable, since it only requires finding the active constraints, computing its null space, and performing an eigenvalue decomposition. On the other hand, as proved in \cite{mei18}, checking SOSP2 is NP-hard even for quadratic $f(\cdot)$. \\
\noindent{\bf 2)} Intuitively, it is relatively easy to design an algorithm based on active constraints: When a sequence of iterates approaches an FOSP, the corresponding active set remains the same (see \cite[Proposition 1.37]{bertsekas2014constrained}, \cite[Proposition 3]{gafni1984two}). Therefore locally the inequality constrained problem is reduced to an equality constrained problem, whose second-order conditions are much easier to satisfy; see 
\cite{facchinei1998accurate}\cite{facchinei1998convergence}. \\
\noindent{{\bf 3)} As we have shown, the SC condition is satisfied with probability one for problems with random data, implying that finding SOSP1 is already good enough for these problems.} \\
{Clearly, our proposed solution concept  SOSP1 represents an interesting tradeoff between the quality of the solutions and computational complexity of the resulting algorithms. In what follows, we will design efficient algorithms that can compute such a solution concept.}

\section{SNAP for Computing SOSP1}

\subsection{Algorithm Description}\label{sec:freespace}

Our proposed algorithm successively performs two main steps: a conventional projected gradient descent (PGD)  step and a {negative curvature descent} (NCD) step. Assuming that the feasible set $\mathcal{X}$ is easy to project (e.g., the non-negativity constraints for the NMF problem), the PGD   finds an approximate first-order solution efficiently, while the negative curvature descent step explores curvature around a first-order stationary solution to move the iterates forward.

To provide a detailed description of the algorithm, we will first introduce the notion of the {\it feasible directions} using the directions $\by$ in \eqref{eq.cond12:exact} and \eqref{eq.cond12}. In particular,  for a given point $\bx\in \mathcal{X}$, we define the feasible direction subspace as
$\mathcal{F}(\bx) = \mathsf{Null}(\bA'(\bx))$, where $\mathsf{Null}(\bA'(\bx))$ denotes the null space of matrix $\bA'(\bx)$.
Such directions  are useful for extracting negative curvature directions. We will refer to the subspace $\mathcal{F}(\bx)$ as  {\it free space} and  we refer to its orthogonal complement  as {\it active space}.

\begin{algorithm}[th]
\caption{Negative-curvature grAdient Projection algorithm (SNAP)}
\label{alg:p1}\footnotesize
\begin{algorithmic}[1]\small
\State {\bfseries Input:} $\bx^{(1)},\epsilon_G,\epsilon_H,L_1,L_2,\alpha_\pi=1/L_1,\delta,\bA,\bb, r_{\textsf{th}}, \textsf{flag}=\Diamond, \textsf{flag}_{\alpha}=\Diamond, r_{\textsf{last}}=0$
\For {$r=1,\ldots$}
\If {$\|g_\pi(\bx^{(r)})\|\le {\epsilon_G}$ and   $(\textsf{flag}_{\alpha}=\Diamond$ or $(\textsf{flag}_{\alpha}=\emptyset$  and $r-r_{\textsf{last}}\ge r_{\textsf{th}}))$}
\State $[\textsf{flag}, \bv(\bx^{(r)}), -{\epsilon'_H(\delta)}]= \textsf{ \it Negative-Eigen-Pair}(\bx^{(r)},f,\delta)$
\If {$\textsf{flag}=\Diamond$}
\State Compute $q_{\pi}(\bx^{(r)})$ by \eqref{eq.compq}
\State Choose $\bv(\bx^{(r)})$ such that $q_{\pi}(\bx^{(r)})^{\T}\bv(\bx^{(r)})\le0$
\If {$\frac{L_1\epsilon'_H(\delta)}{L_2}q_{\pi}(\bx^{(r)})^{\T}\bv(\bx^{(r)})-\frac{63L_1\epsilon_H'^3(\delta)}{128L^2_2}\ge-\|q_{\pi}(\bx^{(r)})\|^2$}
\State { $\bd^{(r)}= -q_{\pi}(\bx^{(r)})$} \Comment{Choose gradient direction}
\Else
\State $\bd^{(r)}=\bv(\bx^{(r)})$  \Comment{Choose negative curvature direction}
\EndIf
\State Update ($\bx^{(r+1)}$, $\textsf{flag}_{\alpha}$) by  \algref{alg:p3}\Comment{Perform line search}
\If {$\textsf{flag}_{\alpha}=\emptyset$}
\State  $r_{\textsf{last}}\leftarrow r$
\EndIf
\Else \State Output $\bx^{(r)}$
\EndIf
\Else \State Update $\bx^{(r+1)}$ by \eqref{eq.pgd}\Comment{{Perform PGD}}
\EndIf
\EndFor
\end{algorithmic}
\end{algorithm}

The input of the algorithms are some constants related to the problem data, the initial solution $\bx^{(1)}$, and parameters $\epsilon_G,\epsilon_H$. Further $\alpha_{\pi}>0$ is the step-size, and $\delta>0$ is the accuracy of the curvature finding algorithm, both will be determined later.

It is important to note that as long as the computation of gradient, Hessian, and projection can be done in a polynomial number of floating point operations, the computational complexity of SNAP becomes polynomial. For most practical problems, it is reasonable to assume that gradient and Hessian and can be computed efficiently. In addition, projection to linear inequality constraints is well-studied and can be done polynomially under reasonable oracles \cite{nemirovski1995information}.

\noindent{\bf {The PGD step (line {21}).}} The conventional PGD, given below,  is implemented in line 21 of \algref{alg:p1}, with a constant step-size $\alpha_\pi>0$:
\begin{equation}\label{eq.pgd}
\bx^{(r+1)}=\pi_{\mathcal{X}}(\bx^{(r)}-\alpha_\pi\nabla f(\bx^{(r)})).
\end{equation}
The PGD  guarantees that the objective value decreases so that the algorithm can achieve some approximate FOSPs, i.e., $\|g_{\pi}(\bx^{(r)})\|\le\epsilon_G$ efficiently (assuming that the projection can be done relatively easily). The procedure stops whenever the FOSP1 gap $\|g_\pi(\bx^{(r)})\|\le \epsilon_G$.

\noindent{\bf Negative curvature descent (NCD line 4-19).}
After PGD has been completed, we know that $\|g_\pi(\bx^{(r)})\|$ is already small. Suppose that the $(\epsilon_G,\epsilon_H)$-SOSP1 solution has not been found yet.  Then our next step is to design an update direction to increase $\lambda_{\min}(\bH_{\bP}(\bx^{(r)}))$.
Towards this end,  a NCD step will be performed (\algref{alg:p1}, line 4--19), which constructs update directions that can exploit curvature information about the Hessian matrix, while ensuring that the iterates stay in the feasible region.
The NCD further contains the following sub-procedures.\\
\noindent{\bf (1) Extracting negative curvature.}
 Assuming that \eqref{eq:Hessian:small} {does not hold}. First, a procedure \textsf{\it Negative-Eigen-Pair} is called, which exploits some second-order information about the function at $\bx^{(r)}$, and returns an approximate eigen-pair $\{\bv(\bx^{(r)}), -\epsilon'_H(\delta)\}$ of the Hessian $\nabla^2 f(\bx^{(r)})$. Such an approximate eigen-pair should satisfy the following requirements (with probability at least $1-\delta$):
\begin{enumerate}
 	\item $\bv(\bx^{(r)})\in\mathcal{F}(\bx^{(r)})$ and $\|\bv(\bx^{(r)})\|=1$;
 	\item $\bv(\bx^{(r)})^T \nabla^2 f(\bx^{(r)})\bv(\bx^{(r)}) \le -\epsilon'_H(\delta)$ for some  $\epsilon'_H$  where $\exists \gamma>0$ such that $\gamma \epsilon'_H(\delta)>\epsilon_H$.
\end{enumerate}

If $\{\bv(\bx^{(r)}), -\epsilon'_H(\delta)\}$ satisfies all the above conditions,  \textsf{ \it Negative-Eigen-Pair} returns $\Diamond$, otherwise, it returns $\emptyset$. As long as \eqref{eq:Hessian:small} holds, many existing algorithms can achieve the two conditions stated above in a finite number of iterations {(e.g., the power or Lanczos method)}. However, these methods typically require to evaluate the Hessian matrix or Hessian-vector product. Subsequently, we will design a new procedure that only utilizes gradient information for such purposes.\\

\noindent{\bf (2) Selection of update direction.} {First, note that whichever choice of the directions we make, we will have $\bd^{(r)}\in \mathcal{F}(\bx^{(r)})$.} Second, by the $L_1$-Lipschitz continuity, we have
\begin{equation}\label{eq.des1}
f(\bx^{(r)}+\alpha \bd^{(r)})\le  f(\bx^{(r)})+{\alpha} q_{\pi}(\bx^{(r)})^{\T}\bd^{(r)}+\frac{\alpha^2L_1}{2}\|\bd^{(r)}\|^2,
\end{equation}
since $\nabla f(\bx^{(r)})=\bP(\bx^{(r)})\nabla f(\bx^{(r)})+\bP_\perp(\bx^{(r)})\nabla f(\bx^{(r)})$ and $q_{\pi}(\bx^{(r)})=\bP(\bx^{(r)})\nabla f(\bx^{(r)})$. Similarly, by $L_2$-Lipschitz continuity, we have
\begin{equation}\label{eq.des2}
f(\bx^{(r)}+\alpha \bd^{(r)})\le  f(\bx^{(r)})+\alpha q_{\pi}(\bx^{(r)})^{\T}\bd^{(r)}+\frac{\alpha^2}{2}(\bd^{(r)})^{\T}\nabla^2 f(\bx^{(r)})\bd^{(r)}+\frac{\alpha^3}{6}L_2 \|\bd^{(r)}\|^3.
\end{equation}
Therefore, it can be observed that the descent of the objective value is determined by the choice of $\bd^{(r)}$.

The actual update direction we use is chosen between the direction $\bv(\bx^{(r)})\in \mathcal{F}(\bx^{(r)})$ found in the previous step, and a direction {$q_\pi(\bx^{(r)})\in \mathcal{F}(\bx^{(r)})$} computed by directly projecting $\nabla f(\bx^{(r)})$ to the subspace of feasible directions. The selection criteria, given in {line 8} of Algorithm 1, is motivated by the following descent properties (note $\|\bv(\bx^{(r)})\|=1$).
\begin{lemma}\label{le.selection}
If {$\bx^{(r)}$ is updated by \algref{alg:p3} (the line search)  and $\alpha^{(r)}_{\max}$ in Algorithm 2 (linear search) is not chosen}, the minimum descent of the objective value by choosing $\bd^{(r)}=-q_{\pi}(\bx^{(r)})$ is no less than the one by selecting $\bd^{(r)}=\bv(\bx^{(r)})$, and vice versa.
\end{lemma}
\noindent{\bf(3) Backtracking line search (\algref{alg:p3}).} The third sub-procedure uses line search to determine the step-size so that the new and feasible iterate $\bx^{(r+1)}$ can be generated. The key in this step is to make sure that, either the new iterate achieves some kind of ``sufficient'' descent, or  it touches the boundary of the feasible set. We use ${\small\textsf{flag}_{\alpha}}$ \normalsize to denote whether the updated iterate touches a new boundary or not. Since the direction is already fixed to be in the {\it free space}, after performing the line search,  the dimension of the feasible directions will be non-increasing, i.e.,  $\dim(\mathcal{F}(\bx^{(r+1)}))\le\dim(\mathcal{F}(\bx^{(r)}))$. 

To understand the algorithm, let us
first define the set of {\it inactive constraints} as
\begin{equation}
\overbar{\bA'}\bydef\left[\begin{array}{lll} \ldots & \bA_i & \ldots\end{array}\right]^{\T}\in\mathbb{R}^{|\overbar{\cA}(\bx)|\times d},\quad \forall i\in\overbar{\mathcal{A}}(\bx).\label{eq:barA}
\end{equation}

The details of the line search algorithm is shown in \algref{alg:p3}.
\begin{algorithm}[t]
\caption{Line search algorithm}
\label{alg:p3}
\begin{algorithmic}[1]
\State {\bfseries Input:} $\bx^{(r)},\bd^{(r)},\epsilon'_H(\delta),\lambda,\bA,\bb$
\If {$\exists i, (\overbar{\bb}'-\overbar{\bA}'(\bx^{(r)})\bx^{(r)})_i/(\overbar{\bA}'(\bx^{(r)})\bd^{(r)})_i>0$}
\State Compute $\alpha^{(r)}_{\max}$ by 
\begin{equation}\label{eq.alphamax}
\alpha^{(r)}_{\max}\bydef \max\{\alpha\ge0\mid { \bx^{(r)}+\alpha \bd^{(r)}}\in\mathcal{X}\}
\end{equation}
\Else
\State Set $\alpha^{(r)}_{\max}=1/L_1$ 
\EndIf
\State Update $\bx^{(r+1)}$ by: $\bx^{(r+1)}=\bx^{(r)}+\alpha^{(r)}_{\max}\bd^{(r)}$
\If {$f(\bx^{(r)}+\alpha^{(r)}_{\max}\bd^{(r)})<f(\bx^{(r)})$}
\State return ($\bx^{(r+1)}$, $\textsf{flag}=\Diamond$) 
\Else
\State $\alpha\leftarrow\alpha^{(r)}_{\max}$
\If {$\bd^{(r)}=-{q_{\pi}(\bx^{(r)})}$}
\State
\begin{equation}\label{eq.defofrho}
{\rho(\alpha)=-\alpha \|q_{\pi}(\bx^{(r)})\|^2}
\end{equation}
\Else
\State
\begin{equation}\label{eq:alpha}
\rho(\alpha)=-\frac{\alpha^2\epsilon'_H(\delta)}{4}
\end{equation}
\EndIf
\While
{\begin{equation}\label{eq.back}
f(\bx^{(r)}+\alpha \bd^{(r)}) > f(\bx^{(r)})+{\frac{1}{2}}\rho(\alpha)
\end{equation}}
\begin{equation}\label{eq.shink}
\alpha\leftarrow \frac{1}{2}\alpha
\end{equation}
\State Compute $\rho(\alpha)$
\EndWhile
\State $\alpha^{(r)}\leftarrow \alpha$
\State   $\bx^{(r+1)}=\bx^{(r)}+\alpha^{(r)}\bd^{(r)}$
\State return ($\bx^{(r+1)}$, $\textsf{flag}=\emptyset$)
\EndIf
\end{algorithmic}
\end{algorithm}
In particular, in this line search procedure, we will first decide a maximum stepsize $\alpha^{(r)}_{\max}>0$. Recall that $\overbar{\bA}'(\bx)$ defined in \eqref{eq:barA} represents the set of constraints that are inactive at point $\bx$.

\begin{lemma}\label{le.alpha}
If there exits an index $i\in \overbar{\cA}(\bx^{(r)})$ so that the following holds 
\begin{equation}
\alpha^{(r)}_i\triangleq \frac{(\overbar{\bb}'-\overbar{\bA}'(\bx^{(r)})\bx^{(r)})_i}{(\overbar{\bA}'(\bx^{(r)})\bd^{(r)})_i}>0,\exists i,\label{eq.closalpha}
\end{equation}
then, we have
\begin{equation}\alpha^{(r)}_{\max}=\left\{\min\{\alpha_i^{(r)}>0\}\mid (\overbar{\bA}'(\bx^{(r)})\bx^{(r)}+\alpha_i^{(r)}\overbar{\bA}'(\bx^{(r)})\bd^{(r)})_i=\overbar{\bb}_i, \forall i\right\}.\label{eq.closealpha}
\end{equation}
\end{lemma}

On the other hand, if  the condition \eqref{eq.closalpha} does not hold, it means that along the current direction the problem is effectively {\it unconstrained}.  Therefore, the line search algorithm reduces to the classic unconstrained update. Then by setting $\alpha^{(r)}_{\max}=1/L_1$, SNAP will give a sufficient decrease in this case; see \leref{le.2}.

After choosing $\alpha^{(r)}_{\max}$, we check if  the following holds
\begin{equation}\label{eq:stop}
f(\bx^{(r)}+\alpha^{(r)}_{\max}\bd^{(r)}) <f(\bx^{(r)}).
\end{equation}
If so, then the algorithm {either touches the boundary without increasing the objective, or it has already achieved sufficient descent}.

If \eqref{eq:stop} does not hold, then the algorithm will call the backtracking line search  by successively shrinking the step-size starting at $\alpha\leftarrow\alpha^{(r)}_{\max}$. In particular, if $f(\bx^{(r)}+\alpha \bd^{(r)})>f(\bx^{(r)})+\lambda\rho(\alpha)$ (where $\rho(\alpha)$ is some pre-determined negative quantity, see \eqref{eq:alpha}),  we will implement $\alpha\leftarrow \frac{1}{2}\alpha$ until a sufficient descent is satisfied ({note, such a sufficient descent can be eventually achieved}, see \leref{le.2} and \leref{le.3}).

\noindent{\bf(4) ``Flags'' in the algorithm.} We note that after each NCD step if $\textsf{flag}_{\alpha}=\emptyset$  (i.e., some sufficient descent is achieved), we perform $r_{\textsf{th}}$ iterations of PGD. This design tries to improve the practical efficiency of the algorithm by striking the balance between objective reduction and computational complexity. In practice $r_{\textsf{th}}$ can be chosen as any constant number. When $r_{\textsf{th}}=0$, SNAP becomes simpler and has the same order of convergence rate
as the case where $r_{\textsf{th}}>0$. See Appendix \ref{sec:ssnp}.

\subsection{Theoretical Guarantees}
The convergence analysis of the proposed algorithm is provided in this section. See Appendix \ref{sec:snap:proof}. 


\begin{theorem}\label{th.1}
(Convergence rate of SNAP) Suppose the objective function satisfies assumption~\ref{as1}. There exists a sufficient small $\delta'$ so that the sequence  $\{\bx^{(r)}\}$ generated by \algref{alg:p1} satisfies optimality condition \eqref{eq.cond1}  in the following number of iterations with probability at least $1-\delta'$:
\begin{equation}\label{eq.convrate}
\widetilde{\mathcal{O}}\left(\max\left\{\frac{L_1\min\{d,m\}}{\epsilon^2_G},\frac{L^2_2\max\{r_{\textsf{th}},\min\{d,m\}\}}{\epsilon^3_H}\right\}(f(\bx^{(1)})-f^{\star})\right)
\end{equation}
where the randomness comes from the oracle  \textsf{\it Negative-Eigen-Pair}, $f^{\star}\bydef \min_{\bx\in\mathcal{X}} f(\bx)$ denotes the global minimum value, and $\widetilde{\mathcal{O}}$ hides the number of iterations run by an oracle \textsf{\it Negative-Eigen-Pair}.
\end{theorem}
\remark The convergence rate of SNAP has the same order in $\epsilon_G,\epsilon_H,L_1,L_2$, compared with those proposed in \cite{aras18} and \cite{mei18} (which compute $(\epsilon_G, \epsilon_H)$-SOSP2s). However, it is important to note that the per-iteration complexity of SNAP is polynomial in both problem dimension and in number of constraints, while algorithms proposed in \cite{aras18} and \cite{mei18} have  exponential per-iteration complexity.

\remark In particulay, SNAP needs $\widetilde{\mathcal{O}}(\min\{d,m\}/\epsilon^2)$ number of iterations to achieve an $(\epsilon,\sqrt{\epsilon})$-SOSP1 by just substituting $\epsilon_G=\epsilon$, $\epsilon_H=\sqrt{L_2\epsilon}$ and $r_{\textsf{th}}\sim\mathcal{O}(L_1/\sqrt{L_2\epsilon})$.

\section{First-order Successive Negative-curvature Gradient Projection (SNAP$^+$)}
{In this section, we propose a {\bf f}irst-order algorithm for SNAP, i.e., SNAP$^+$, featuring a subspace perturbed gradient descent (SP-GD) procedure that can extract the negative curvature in a subspace. Our work is motivated by recent works, which show that occasionally adding random noise to the iterates of GD can help escape from saddle points efficiently \cite{jin2017jordan,ge2015escaping}. The benefit of the proposed SNAP$^{+}$ is that its complexity can be improved significantly since the procedure of finding the negative eigenpair is implemented by a first-order method.}

In particular, the key idea of these perturbation schemes is to use the difference of the gradient successively \cite{xu2017first}, given below, to approximate the Hessian-vector product
\begin{equation}\label{eq.spgd}
\bz^{(\tau+1)}=\bz^{(\tau)}-\beta(q_{\pi}(\bx^{(r)}+\bz^{(\tau)})-q_{\pi}(\bx^{(r)})), \quad \textrm{for }\tau = 1,\ldots,T.
\end{equation}
Here $T$ is some properly selected constant, $\beta\le 1/L_1$ is the step-size and the algorithm is initialized from a random vector $\bz^{(1)}\in\mathcal{F}(\bx^{(r)})$ drawn from a uniform distribution in the interval $[0,\sR]$, where $\sR$ is some constant. This process can be viewed as performing power iteration around the strict saddle point. The details of the algorithm is presented in \algref{alg:p5}, and its convergence is as the following.
\begin{theorem}\label{th:spgd1}
SP-GD is called with the step-size $\beta\le 1/L_1$,
\begin{equation}
T\ge \frac{\widehat{c}\log(\frac{dL_1}{\epsilon_H\delta})}{\beta\epsilon_H}+1,\quad \sF=\frac{\epsilon_H^3}{L^2_2\widehat{c}^5\log^3(\frac{dL_1}{\epsilon_H\delta})},\quad\sR=\frac{\epsilon_H^2}{L_1L_2\widehat{c}^4\log^2(\frac{dL_1}{\epsilon_H\delta})}
\end{equation}
where $\widehat{c}\ge 51$. Then, for any $0<\delta<1$, $\epsilon_H\le L_1$,  SP-GD returns $\Diamond$ and a vector $\bz$ such that
\begin{equation}
\frac{\bz^{\T}\nabla^2 f(\bx)\bz}{\|\bz\|^2}\le -\frac{\epsilon_H}{8\widehat{c}\log(\frac{dL_1}{\epsilon_H\delta})}
\end{equation}
with probability $1-\delta$. Otherwise SP-GD returns $\emptyset$ and vector 0, indicating that $\lambda_{\min}(\bH_{\bP}(\bx))\ge-\epsilon_H$ with probability $1-\delta$.
\end{theorem}

\begin{algorithm}[t]
\caption{Subspace Perturbed Gradient Descent (SP-GD)}
\label{alg:p5}
\begin{algorithmic}[1]
\State {\bfseries Input:} $\bx^{(r)},T,q_{\pi},\sF,\sR, \beta=1/L_1,d,\delta,\widehat{c},\epsilon_H$
\State Generate vector $\bz$ randomly from the sphere of an Euclidean ball of radius $\sR$ in $\mathcal{F}(\bx^{(r)})$.
\For {$\tau=1,\ldots,T$}
\State
\begin{equation}
\bz^{(\tau+1)}=\bz^{(\tau)}-\beta(q_{\pi}(\bx^{(r)}+\bz^{(\tau)})-q_{\pi}(\bx^{(r)}))
\end{equation}
\EndFor
\If {$f(\bx^{(r)}+\bz^{(T)})-f(\bx^{(r)})-q_{\pi}(\bx^{(r)})^{\T}\bz^{(T)}\le-1.5\sF$}
\State return $[\Diamond,\bz^{(T)}/\|\bz^{(T)}\|,-\frac{\epsilon_H}{4\widehat{c}\log(\frac{dL_1}{\epsilon_H\delta})}]$
\Else
\State return $[\emptyset,0,0]$
\EndIf
\end{algorithmic}
\end{algorithm}

{Essentially, when SP-GD stops, it produces a direction $\bz$ that satisfies the requirements of the outputs for the {\it Negative-Eigen-Pair} oracle in Algorithm 1. It follows that the rate claimed in Theorem \ref{th.1} still holds for SNAP$^+$. 
The proof can be found in Appendix \ref{sec:th2}}

\begin{corollary}\label{co:th2}
(Convergence rate of SNAP$^+$) Suppose Assumption~\ref{as1} is satisfied and  SP-GD with step-size less than $1/L_1$ is used to find the negative eigen-pair. Then, there exists a sufficiently small $\delta'$ such that the sequence  $\{\bx^{(r)}\}$ generated by \algref{alg:p1} finds an $(\epsilon,\sqrt{\epsilon})$-SOSP1 in the following number of iterations with probability at least $1-\delta'$ (where $\widehat{\mathcal{O}}$ hides the polynomial in terms of $d,1/\epsilon$)
{\small\begin{equation}\label{eq.convrate2}
\widehat{\mathcal{O}}\left(\max\left\{\frac{L^2_1}{L_2^{1/2}\epsilon^{2.5}},\frac{L_1}{\epsilon^2}\right\}\min\{d,m\}(f(\bx^{(1)})-f^{\star})\right).
\end{equation}}
\end{corollary}
{Note that the total rate includes the number of iteration required by SP-GD, so it is $\mathcal{O}({1/\epsilon^{0.5}})$ slower than the rate of SNAP.}

\section{Connection with Existing Work and Future Work}

\begin{table*}[b]
\caption{Convergence rates of algorithms to SOSPs, where ESP denotes the escape saddle point algorithm proposed in  \cite{aras18}, SO-LC-Trace denotes the second-order-linear constrained-TRACE algorithm proposed in \cite{nouiehed2019trust}, H-V denotes Hessian-vector, P-I denotes per-iteration,  $\exp(\cdot)$ stands for exponential, and $\widetilde{\mathcal{O}}(\cdot)$ hides per-iteration complexity.}
\label{tab.com}
\vskip 0.15in
\begin{center}
\begin{small}
\begin{sc}
\begin{tabular}{lllll}
\toprule
Algorithm & Complexity P-I & Iterations  & $(\epsilon_G,\epsilon_H)$-SOSP & Oracle\\
\midrule
ESP \cite{aras18}  & $\mathcal{O}(\exp(m))$  & $\widetilde{\mathcal{O}}(\max\{\epsilon^{-2}_G,\epsilon^{-3}_H\})$ & $(\epsilon_G,\epsilon_H)$-SOSP2 & Hessian \\
SO-LC-TRACE \cite{nouiehed2019trust} & $\mathcal{O}(\exp(m))$ & $\widetilde{\mathcal{O}}(\max\{\epsilon^{-1.5}_G,\epsilon^{-3}_H\})$& $(\epsilon_G,\epsilon_H)$-SOSP2  &H-V product \\
{\bf SNAP} (This work) & $\mathcal{O}(\textrm{poly}(d,m))$  & $\widetilde{\mathcal{O}}(\max\{\epsilon^{-2}_G,\epsilon^{-3}_H\})$  & $(\epsilon_G,\epsilon_H)$-SOSP1 & H-V product \\
{\bf SNAP$^+$} (This work) & $\mathcal{O}(\textrm{poly}(d,m))$  & $\widehat{\mathcal{O}}(1/\epsilon^{2.5})$  & $(\epsilon,\epsilon^{\frac{1}{2}})$-SOSP1 & Gradient\\
\bottomrule
\end{tabular}
\end{sc}
\end{small}
\end{center}
\vskip -0.1in
\end{table*}

In Table \ref{tab.com}, we provide comparisons of key features of a few existing algorithms for finding SOSPs for problem linearly constrained non-convex problems.

{1), SNAP and SNAP$^+$ has a polynomial complexity per-iteration, while the existing works rely on an exponential complexity of some subroutine for solving the inner loop optimization problems.

2), SNAP has the same convergence rate as the existing work in \cite{aras18}. If the acceleration technique is adopted, SNAP can also have a faster convergence.

3), To the best of our knowledge, SNAP$^+$ is the first first-order algorithm that has the provable convergence rate to SOSPs, where the rate is only in the order of polynomial in terms of problem dimension and the total number of constraints.}

\remark \leref{le:equiv}, \prref{pr:00}, \coref{co:0eh} have stated clearly the equivalence between SOSP1 and SOSP2. We leave the equivalence between $(\epsilon_G,\epsilon_H)$-SOSP1 and $(\widetilde{\epsilon}_G,\epsilon_H)$-SOSP2 as the future work.

\section{Numerical Results}
In this section, we showcase the numerical advantages of SNAP and SNAP$^+$, compared with PGD and PGD with line search (PGD-LS) for multiple machine learning problems, such as NMF, training nonnegative neural networks, penalized NMF, etc. 
\subsection{NMF Problems}
First, we consider the following NMF problem, which is
\begin{subequations}
\begin{align}
\min_{\bW\in\mathbb{R}^{n\times k},\bH\in\mathbb{R}^{m\times k}} &\|\bW\bH^{\T}-\bM\|^2
\\
\textrm{s.t.}\quad&\bW\ge0,\bH\ge0.
\end{align}
\end{subequations}
The starting point for all the algorithms is $\bX^{(1)}=c\pi_{\mathcal{A}}([\bW^{(1)};\bH^{(1)}])$, where $\bW^{(1)}$ and $\bH^{(1)}$ are randomly generated. Constant $c$ controls the distance between the initialization point and the origin. The cases of $c=1$ and $c=10^{-10}$ correspond to large and small initialization, respectively. The rationale for considering a small initialization is that for NMF problems, it can be easily checked that $(\bW={\bf 0}, \bH={\bf 0})$ is a saddle point. By initializing around this point, we aim at examining whether indeed the proposed SNAP and SNAP$^+$ are able to escape from this region.

\begin{figure}[ht]
\centering
\subfigure[Loss value versus iteration]{
\label{fig:lvi1}
\includegraphics[width=.4\linewidth]{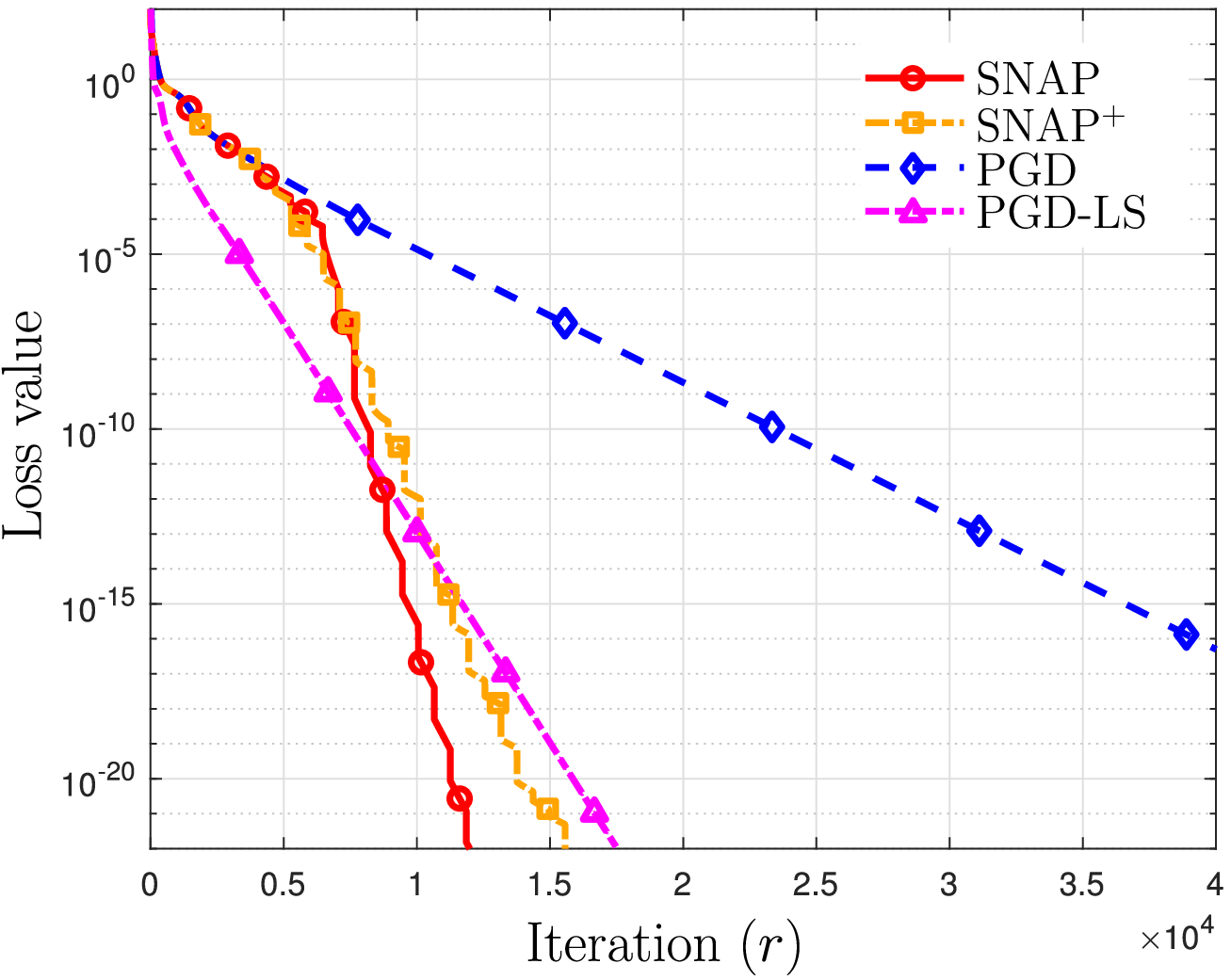}}
\hspace{0.4in}
\subfigure[{Loss value versus computational time}]{
\label{fig:lvt1}
\includegraphics[width=.4\linewidth]{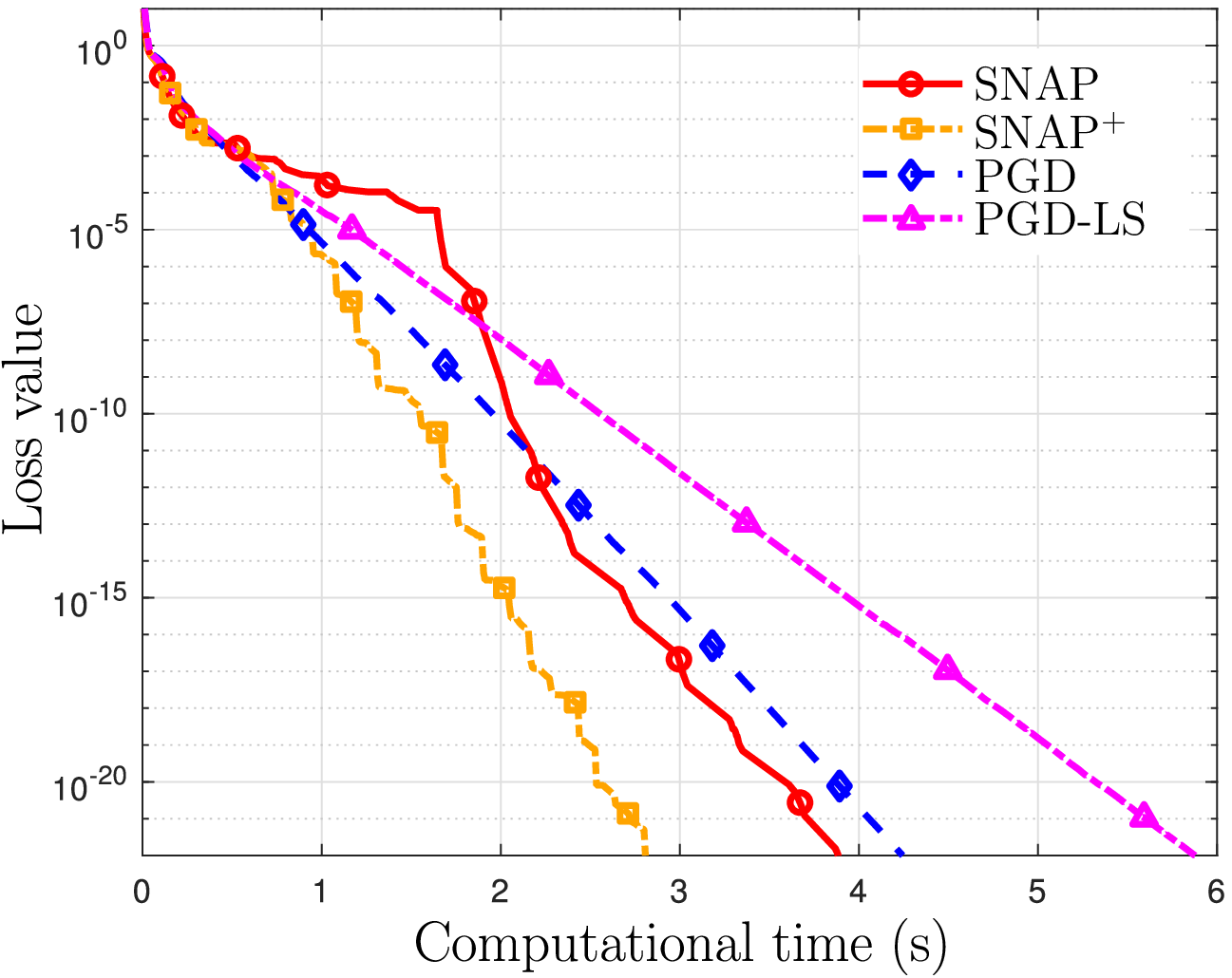}}
\caption{The convergence behaviors of SNAP, SNAP$^+$, PGD, PGD-LS for NMF, where $c=1$.}
\label{fig:nmf1}
\end{figure}
\begin{figure}[ht]
\centering
\subfigure[Loss value versus iteration]{
\label{fig:lvi2}
\includegraphics[width=.4\linewidth]{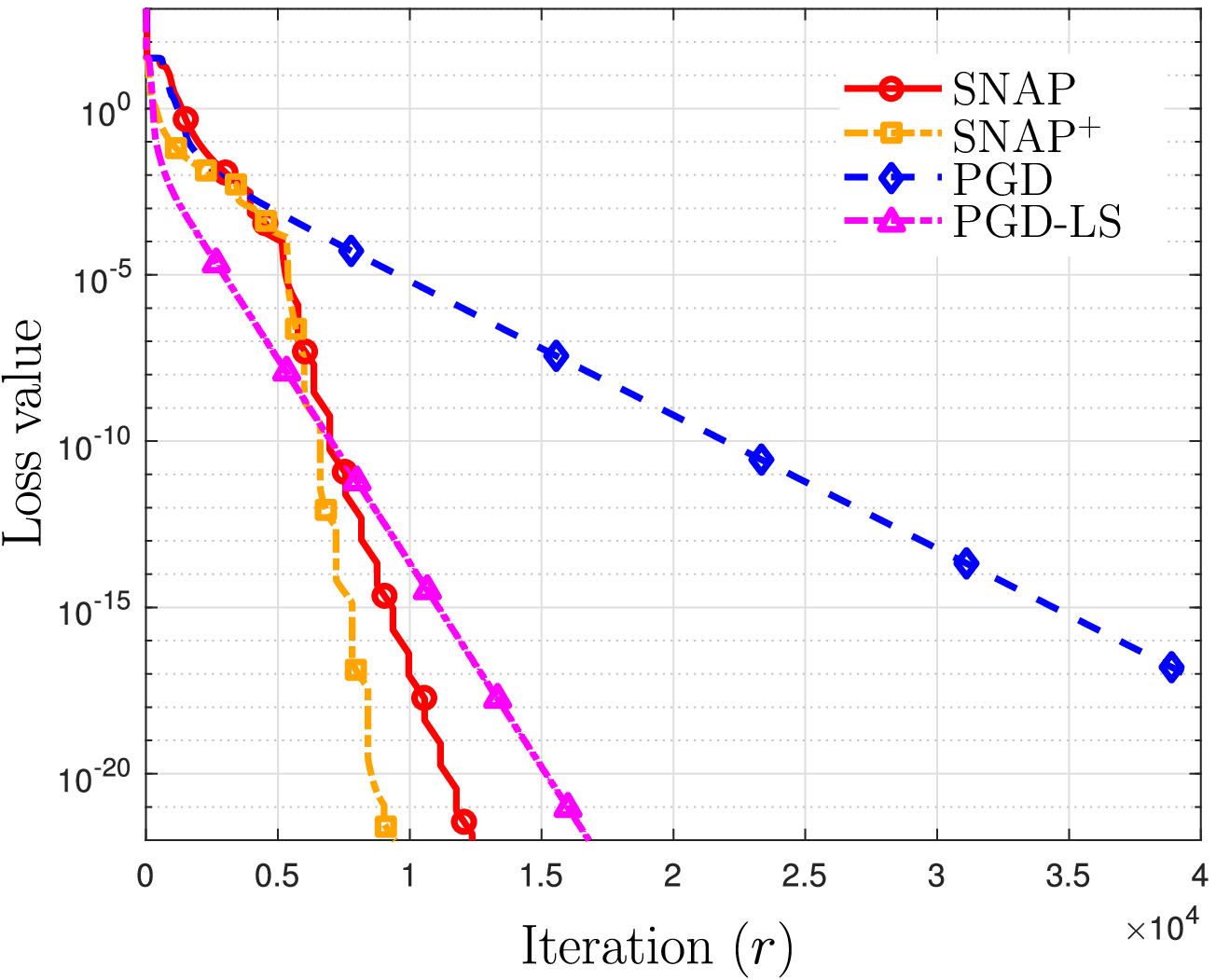}}
\hspace{0.4in}
\subfigure[{Loss value versus computational time}]{
\label{fig:lvt2}
\includegraphics[width=.4\linewidth]{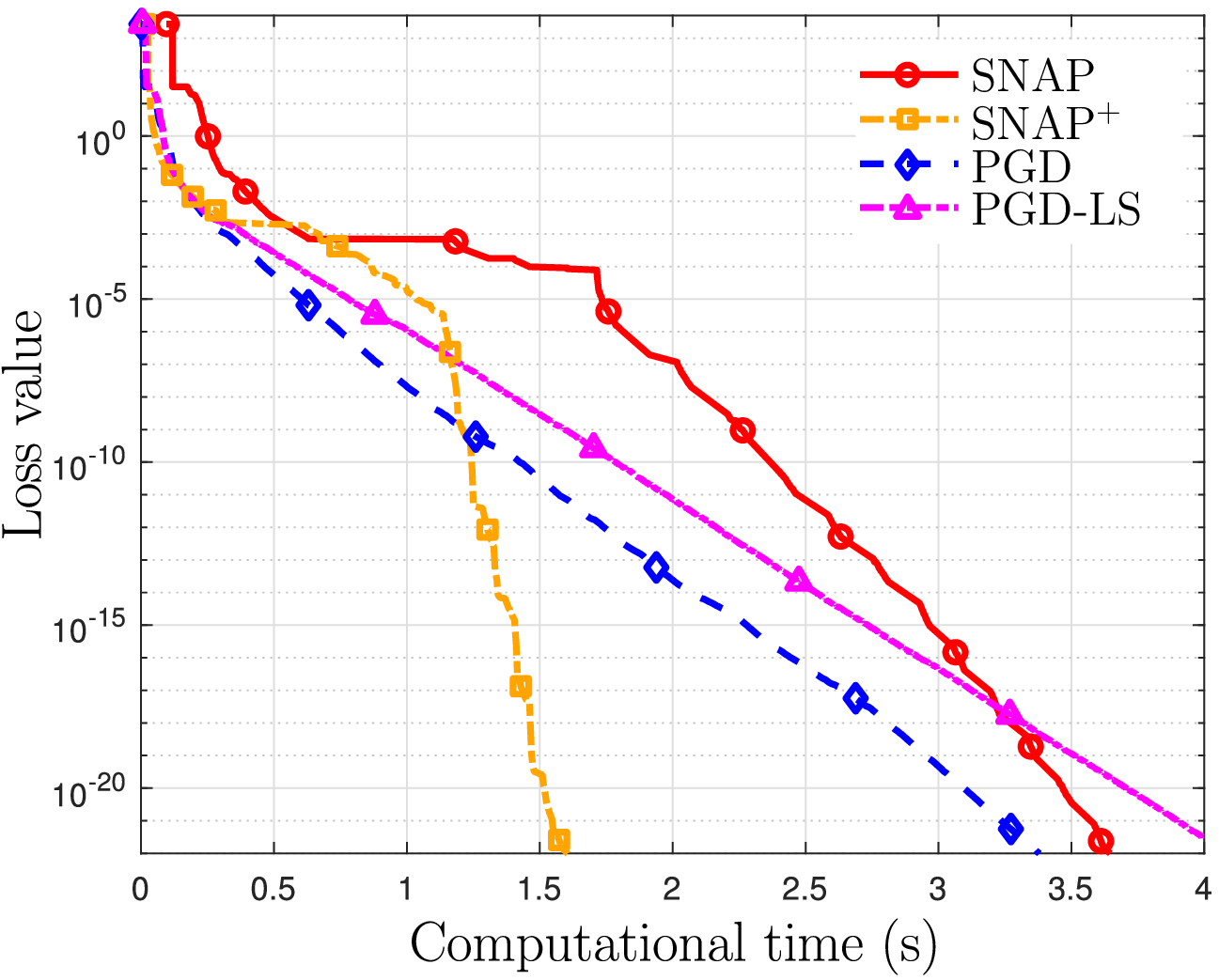}}
 \caption{The convergence behaviors of SNAP, SNAP$^+$, PGD, PGD-LS for NMF, where $c=1\times 10^{-5}$.}
\label{fig:nmf2}
\end{figure}
\begin{figure}[ht]
\centering
\subfigure[Loss value versus iteration]{
\label{fig:lvi3}
\includegraphics[width=.4\linewidth]{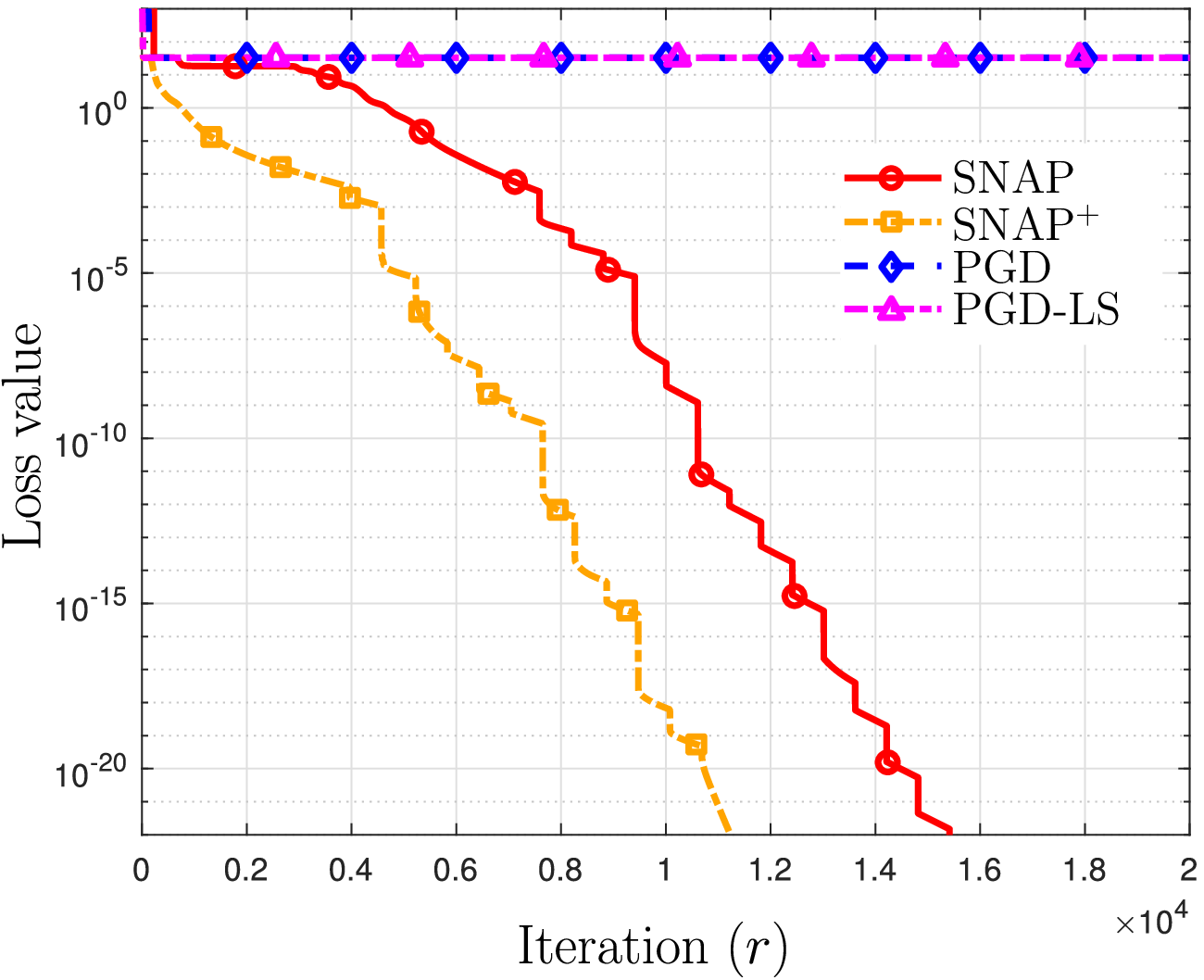}}
\hspace{0.4in}
\subfigure[{Loss value versus computational time}]{
\label{fig:lvt3}
\includegraphics[width=.4\linewidth]{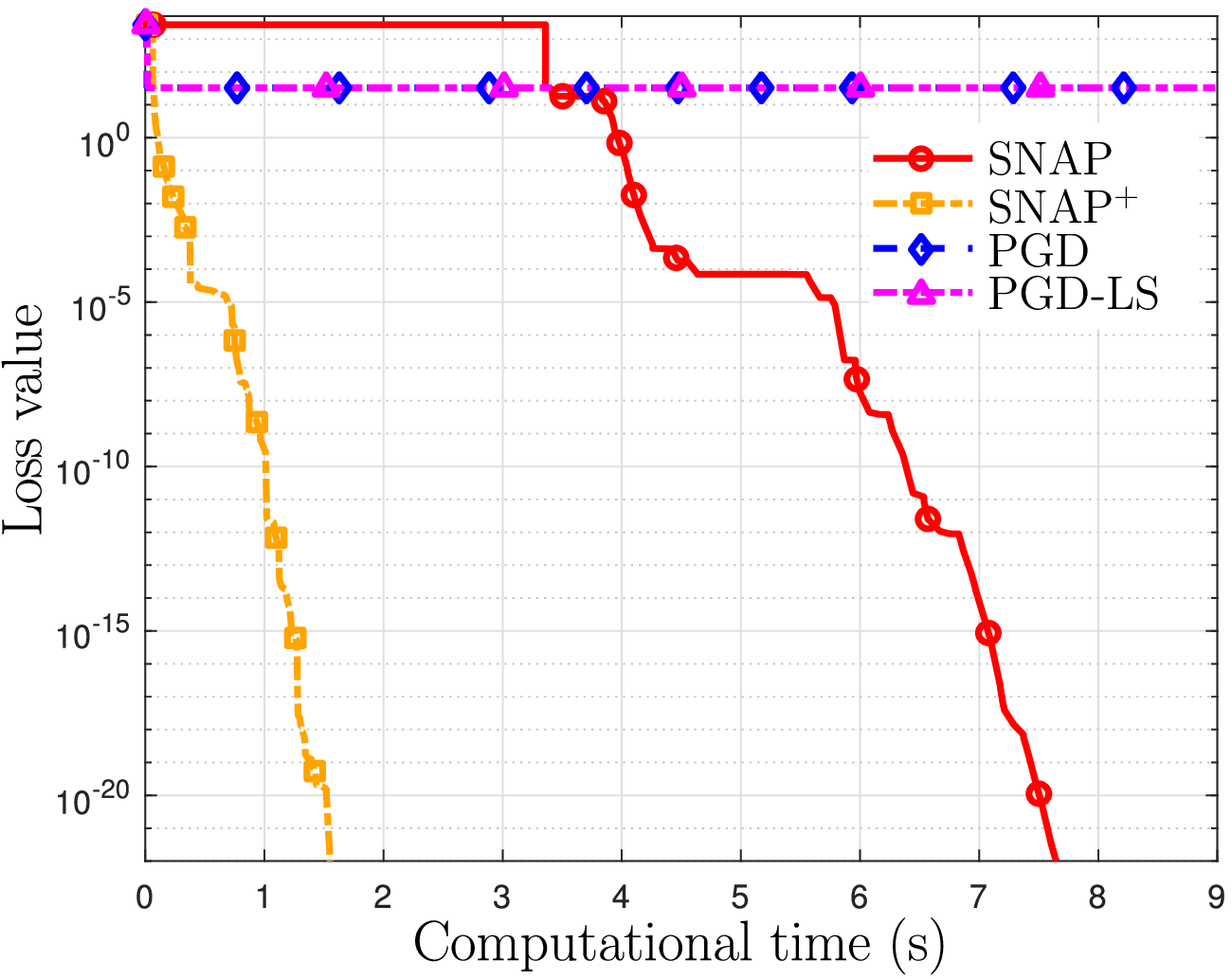}}
\caption{The convergence behaviors of SNAP, SNAP$^+$, PGD, PGD-LS for NMF, where $c=1\times 10^{-10}$.}
\label{fig:nmf3}
\end{figure}

\subsubsection{Synthetic Dataset}\label{sec:snmf}
We compare the proposed SNAP, SNAP$^+$ with PGD and PGD-LS on the synthetic dataset for MNF problem and show the advantages of exploiting negative curvatures. The data matrices are randomly generated, where $m=20$, $n=50$, $k=10$, $\bM=\bW\bH^{\T}$, and $[\bW;\bH]\in\mathbb{R}^{(n+m)\times k}$ follows the uniform distribution in the interval $[0,1]$. Further, we randomly set 5\% entries of $\bM$ as 0. The starting point for all the algorithms is 
$\bX^{(1)}=c\pi_{\mathcal{A}}([\bW^{(1)};\bH^{(1)}])$, where $\bW^{(1)}$ and $\bH^{(1)}$ are randomly generated and follow Gaussian distribution $\mathcal{CN}(0,1)$. Clearly, the origin point is a strict saddle point. We use three different constants $c$ to initialize sequence $\bX^{(r)}$ and the results are shown in \figref{fig:nmf1}--\figref{fig:nmf3}, where step-size  $\alpha_{\pi}$ for PGD, SNAP, and SNAP$^+$ is $0.01$, $\beta=0.01$, $\epsilon_G=1\times 10^{-3}$, $T=100$, $r_{\textsf{th}}=600$, $\sR=1\times 10^{-4}$ and $\sF=100$. Note that the stopping criteria are removed in the simulation, otherwise PGD and PGD-LS will not give any output if the initialing point is close to origin.

\begin{figure}[ht]
\centering
\subfigure[Loss value versus iteration]{
\label{fig:lvi4}
\includegraphics[width=.4\linewidth]{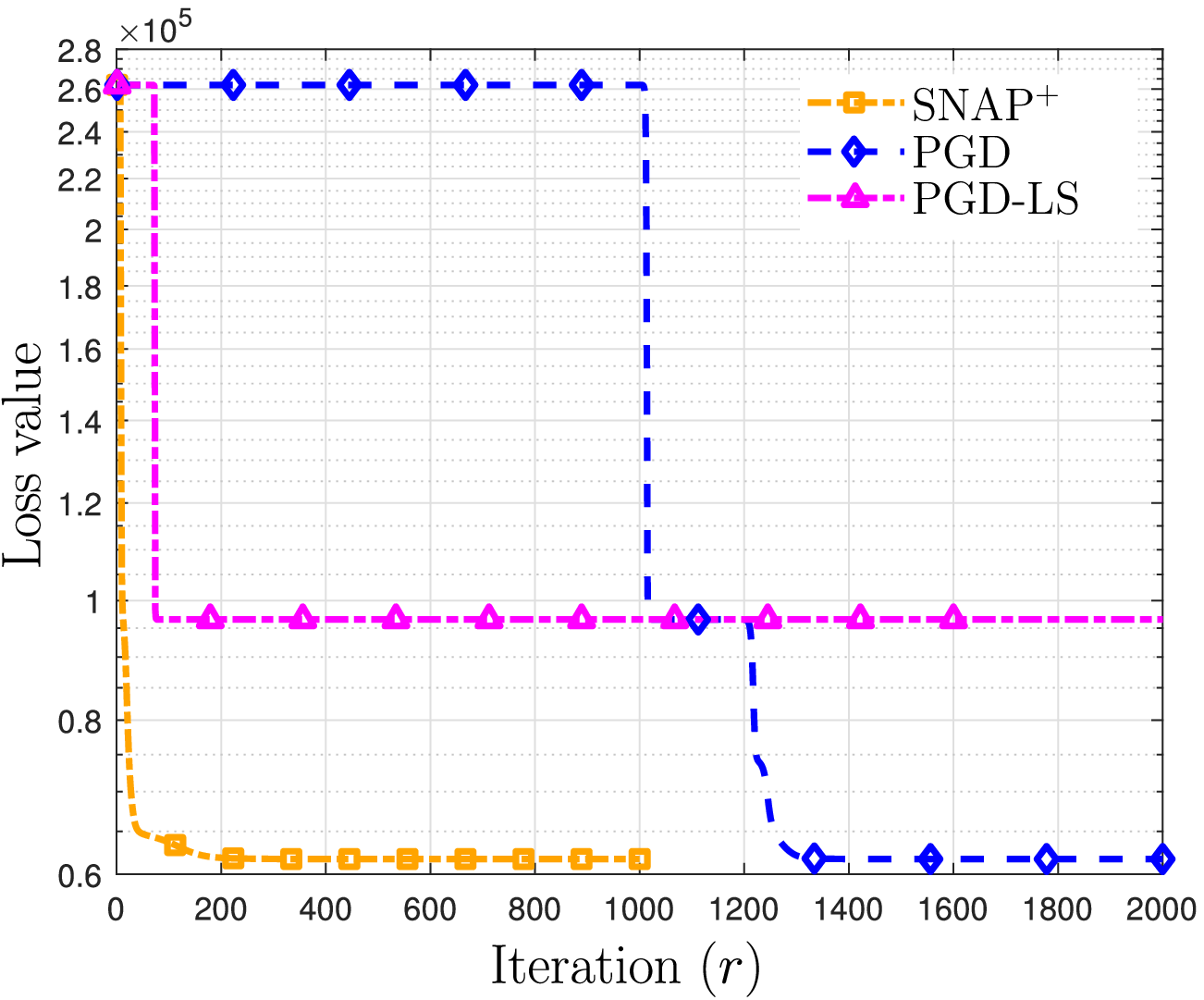}}
\hspace{0.4in}
\subfigure[{Loss value versus computational time}]{
\label{fig:lvt4}
\includegraphics[width=.4\linewidth]{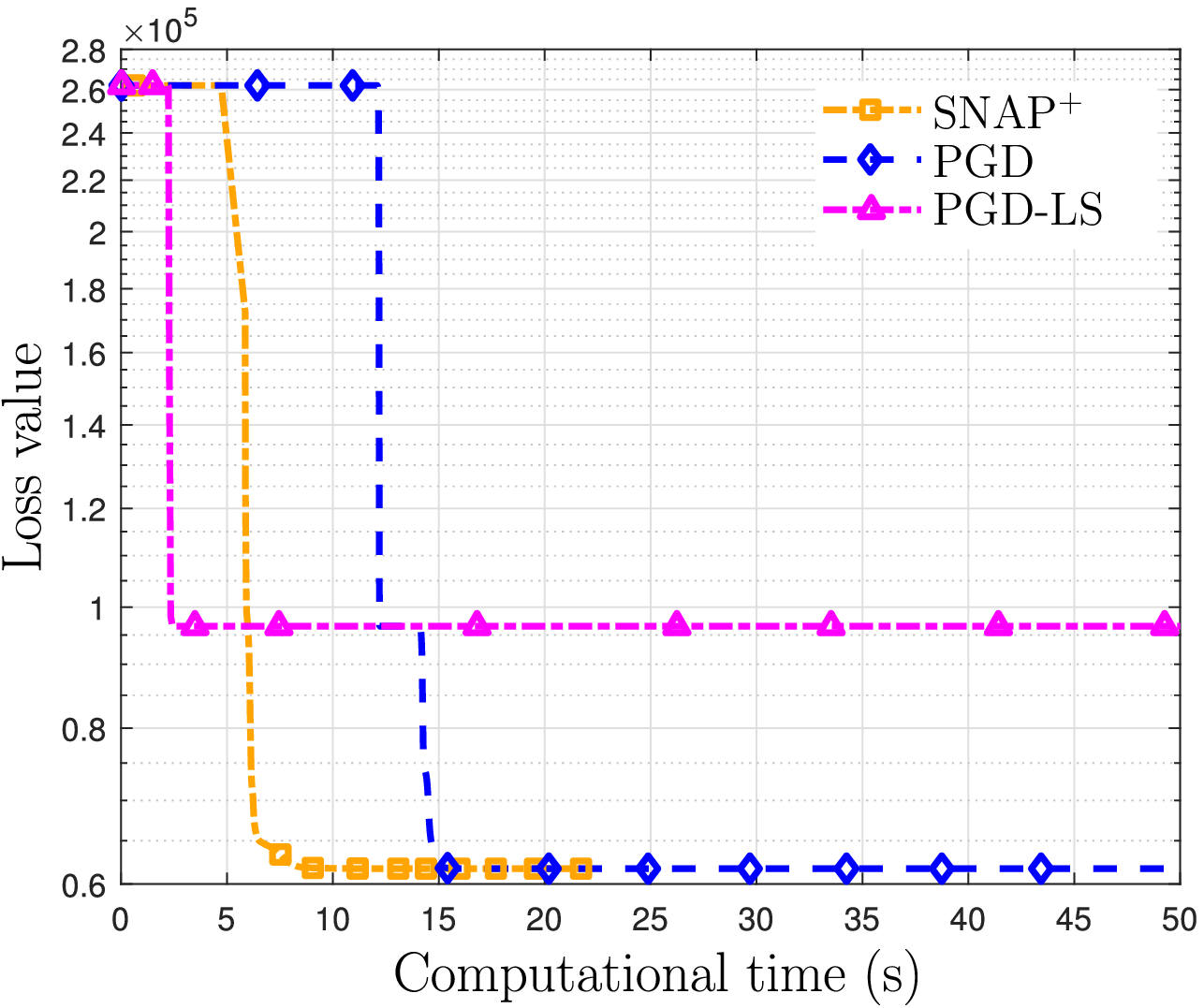}}
\caption{The convergence behaviors of SNAP$^+$, PGD, PGD-LS for NMF, where $c=1\times 10^{-10}$.}
\label{fig:nmf4}
\end{figure}

It can be observed that when $c$ is large, all algorithms can converge to the global optimal point of this NMF problem, whereas when $c$ is small as shown in \figref{fig:nmf3} PGD and PGD-LS only converge to a point that has a very large loss value compared with the ones achieved by SNAP and SNAP$^+$. These results show that when the iterates are near the strict saddle points, by exploring the negative curvature, SNAP and SNAP$^+$ are able to escape from the saddle points quickly and converge to the global optimal solutions. Comparing SNAP and SNAP$^+$, we can see that the computational time of SNAP$^+$ is less than SNAP. The reason is simple, which is the computational complexity of calculation of Hessian and eigen-decomposition is too high so that SNAP takes more time to converge. By accessing the gradient and loss value of the objective function, SNAP$^+$ is only required to compute one eigen-vector whose eigenvalue is the smallest of Hessian around the strict saddle point. The line search algorithm is one of the most effective ways of computing step-sizes. From \figref{fig:nmf1} and \figref{fig:nmf2}, it can be observed that PGD-LS converges faster than PGD in terms of iterations but costs more computational time. SNAP and SNAP$^+$ are using line search occasionally rather than each step, so the computational time is not as high as PGD-LS. In \figref{fig:lvt3}, it can be observed that SNAP and SNAP$^+$ obtain loss values that are many orders of magnitude smaller than those obtained by PGD and PGD-LS, confirming that the proposed methods are able to escape from saddle points, while PGD and PGD-LS get trapped.

\subsubsection{Real Dataset}
We also compare the convergence behaviours of the algorithms on USPS handwritten digits dataset \cite{hull1994database}, where images are $16\times 16$ grayscale pixels. In \figref{fig:nmf4}, we use the  $m=3250$, $n=256$ , $k=5$. Since the problem size is large, performing eigenvalue decomposition is prohibitive, so we only compare SNAP$^+$, PGD, and PGD-LS, where $\alpha_\pi=5\times 10^{-3}$ and $\beta=5\times 10^{-3}$.

\subsection{Nonnegative Two Layer Non-linear Neural Networks}

In this section, we consider a nonnegative two layer non-linear neural network, which is
\begin{align}
\notag
\min_{\bW\in\mathbb{R}^{k\times d},\bH\in\mathbb{R}^{m\times k}} &\|\bW\sigma(\bH^{\T}\bX)-\bY\|^2
\\
\textrm{s.t.}\quad&\bW\ge0,\bH\ge0
\end{align}
where $\sigma(\cdot)$ denotes  the activation function. The formulation has a wide applications in regression and learning problems.
\\
\noindent
In the numerical simulation, the activation function is chosen as sigmoid. Data matrix $\bX\in\mathbb{R}^{m\times n}$ is randomly generated which follows uniform distribution in the interval $[0,1]$, where $n=100$ denotes the number of samples and $m=50$ denotes the number of features. Weight matrices $\bW\in\mathbb{R}^{k\times d}$ and $\bH\in\mathbb{R}^{m\times d}$ are also randomly generated, where $k=10$ denotes dimension of the output, $d=15$ is the dimension of the  hidden layer. Then, data matrix $\bY\in\mathbb{R}^{k\times n}$ is generated by $\bY=\bW\sigma(\bH^{\T}\bX)$. The step-size $\alpha_{\pi}$ for PGD, SNAP$^+$ is 0.001, $\beta=0.001$, $r_{\textsf{th}}=50$, $T=50$, $\sR=1\times 10^{-4}$, $\sF=50$, and $\epsilon_G=1\times 10^{-2}$. From \figref{fig:nn1}, it can be observed that SNAP$^+$ can find the stationary  points faster than PGD and PGD-LS.
\begin{figure}[ht]
\centering
\subfigure[Loss value versus iteration]{
\label{fig:lvi9}
\includegraphics[width=.4\linewidth]{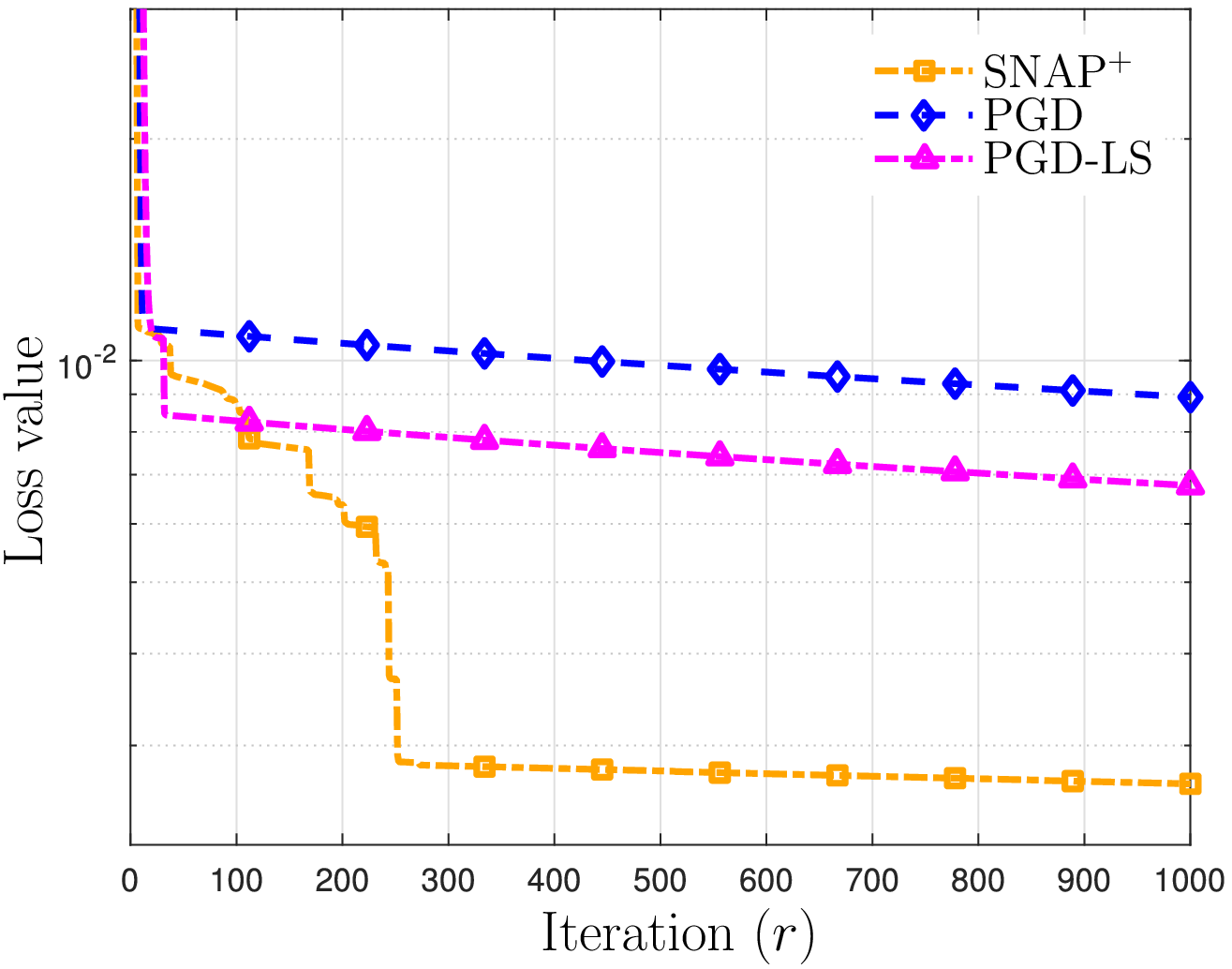}}
\hspace{0.4in}
\subfigure[{Loss value versus computational time}]{
\label{fig:lvt9}
\includegraphics[width=.4\linewidth]{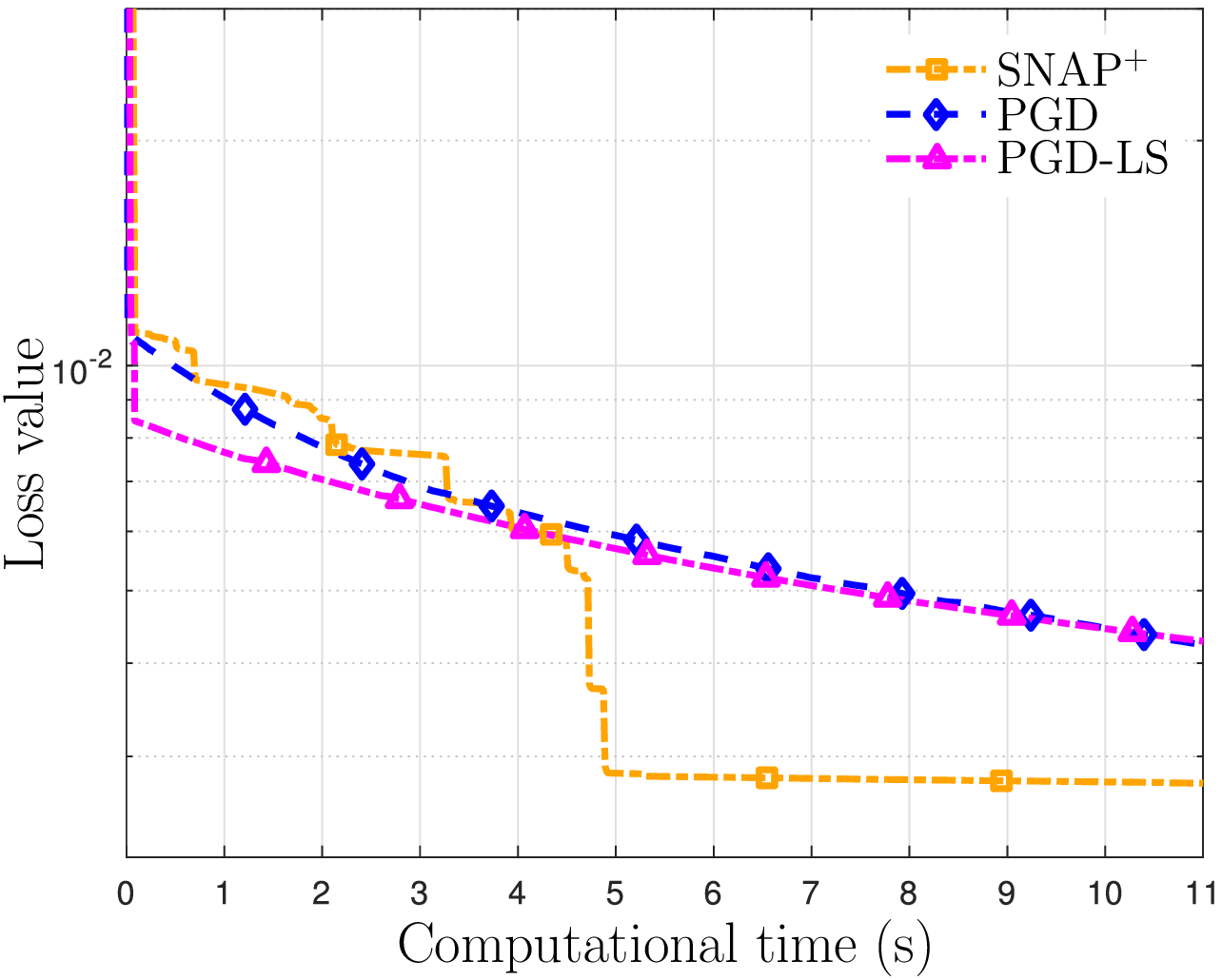}}
\caption{The convergence behaviors of SNAP$^+$, PGD, PGD-LS for NNN, where $c=1$.}
\label{fig:nn1}
\end{figure}
\subsection{Symmetric Matrix Factorization over Simplex}
In application of topic modelling, the simplex constraint turns out to be essential in modeling (marginal) probability mass functions. In this section, we also consider symmetric matrix factorization over a simplex constraint as the following,
\begin{align}
\notag
\min_{\bX\in\mathbb{R}^{n\times k}}\quad&\|\bM-\bH\bH^{\T}\|
\\\notag
\textrm{s.t.}\quad&\bH\ge0,\quad\bH^{\T}\b1=\b1.
\end{align}
In the numerical experiments, the data is generated similar as the NMF case, where $n=100$, $k=5$ and each column of $\bH$ is normalized. We set $\alpha_{\pi}=1\times 10^{-2}$, $T=100$, $r_{\textsf{th}}=100$, $\sR=1\times 10^{-4}$ and $\sF=100$. From \figref{fig:pnmfs}, it is interesting to see that three algorithms converge to different objective values. It turns out there would be multiple stationary points around the origin, where SNAP$^+$ finds the lowest one.
\begin{figure}[ht]
\centering
\subfigure[Loss value versus iteration]{
\label{fig:lvi7}
\includegraphics[width=.4\linewidth]{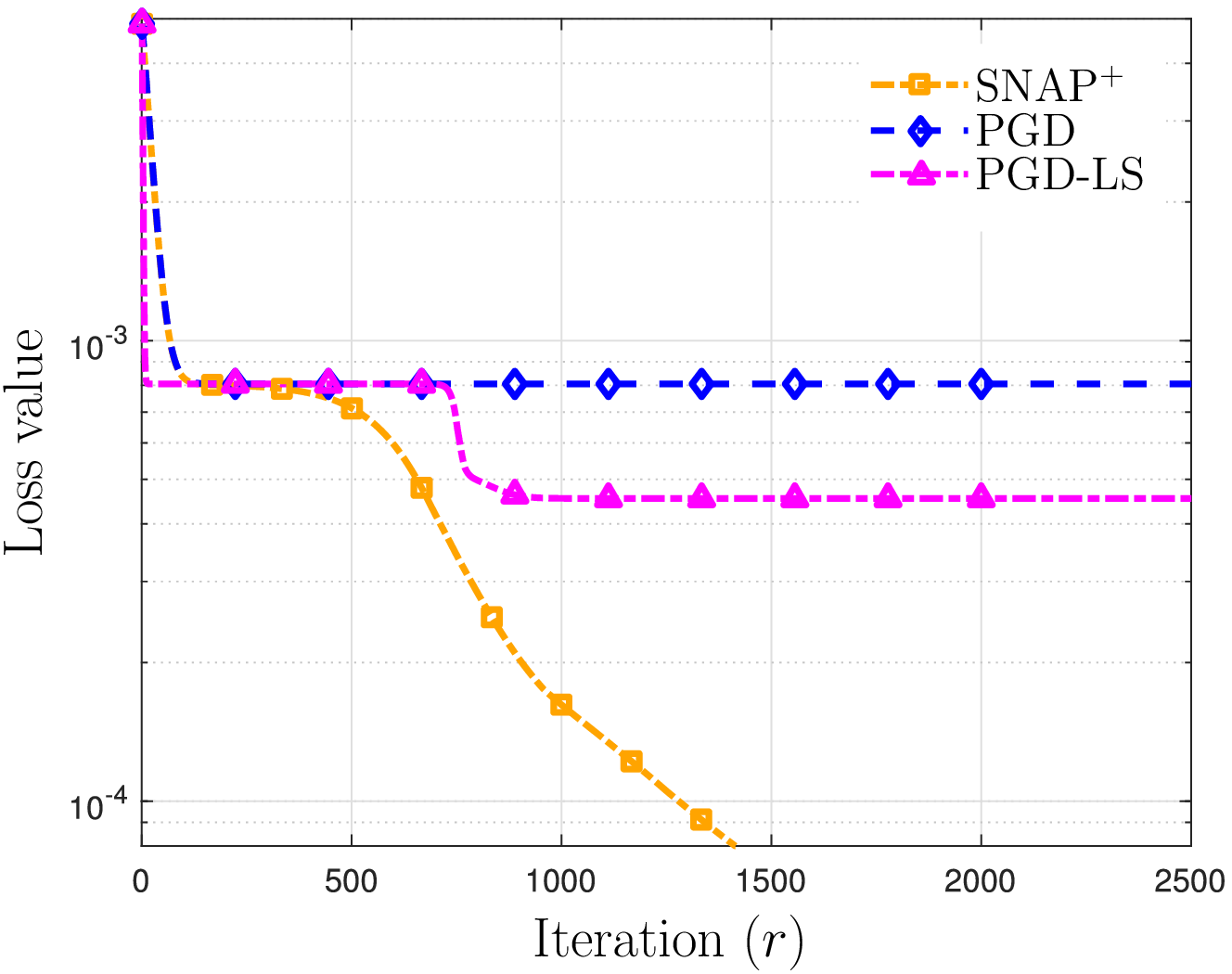}}
\hspace{0.4in}
\subfigure[{Loss value versus computational time}]{
\label{fig:lvt7}
\includegraphics[width=.4\linewidth]{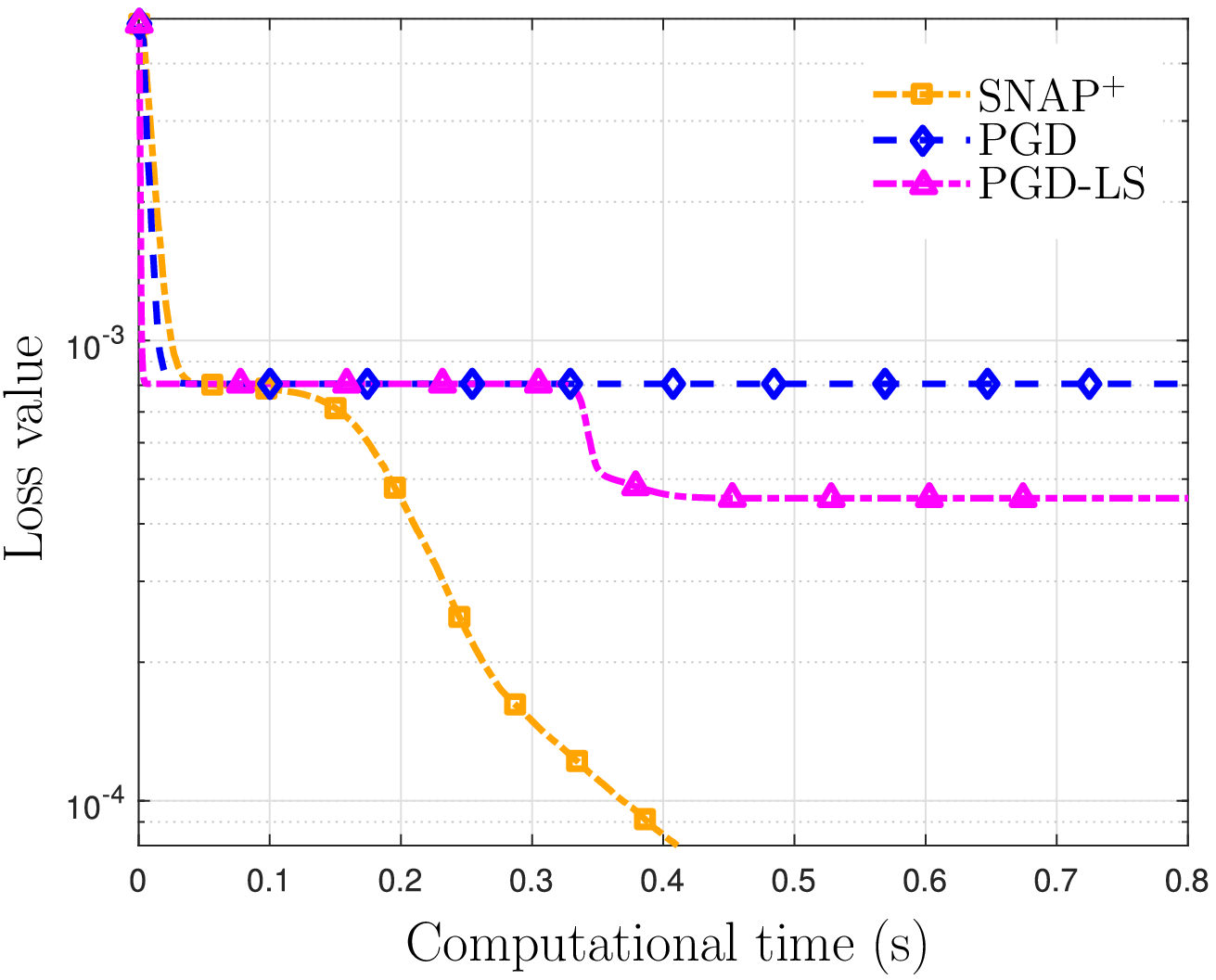}}
\caption{The convergence behaviors of SNAP$^+$, PGD, PGD-LS for matrix factorization under simplex constraints, where $c=1\times 10^{-10}$.}
\label{fig:pnmfs}
\end{figure}

\begin{figure}[ht]
\centering
\subfigure[Loss value versus iteration]{
\label{fig:lvi6}
\includegraphics[width=.4\linewidth]{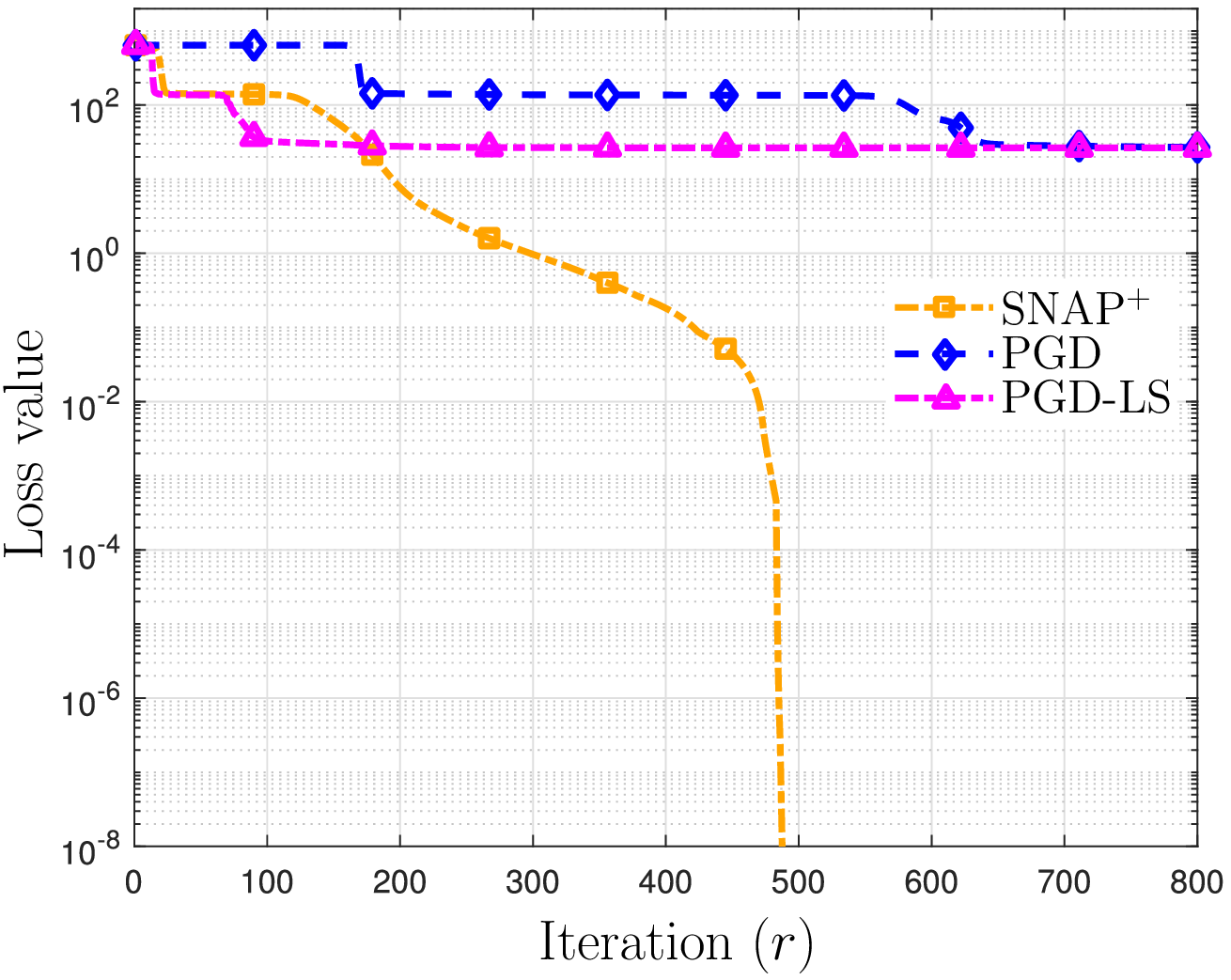}}
\hspace{0.4in}
\subfigure[{Loss value versus computational time}]{
\label{fig:lvt6}
\includegraphics[width=.4\linewidth]{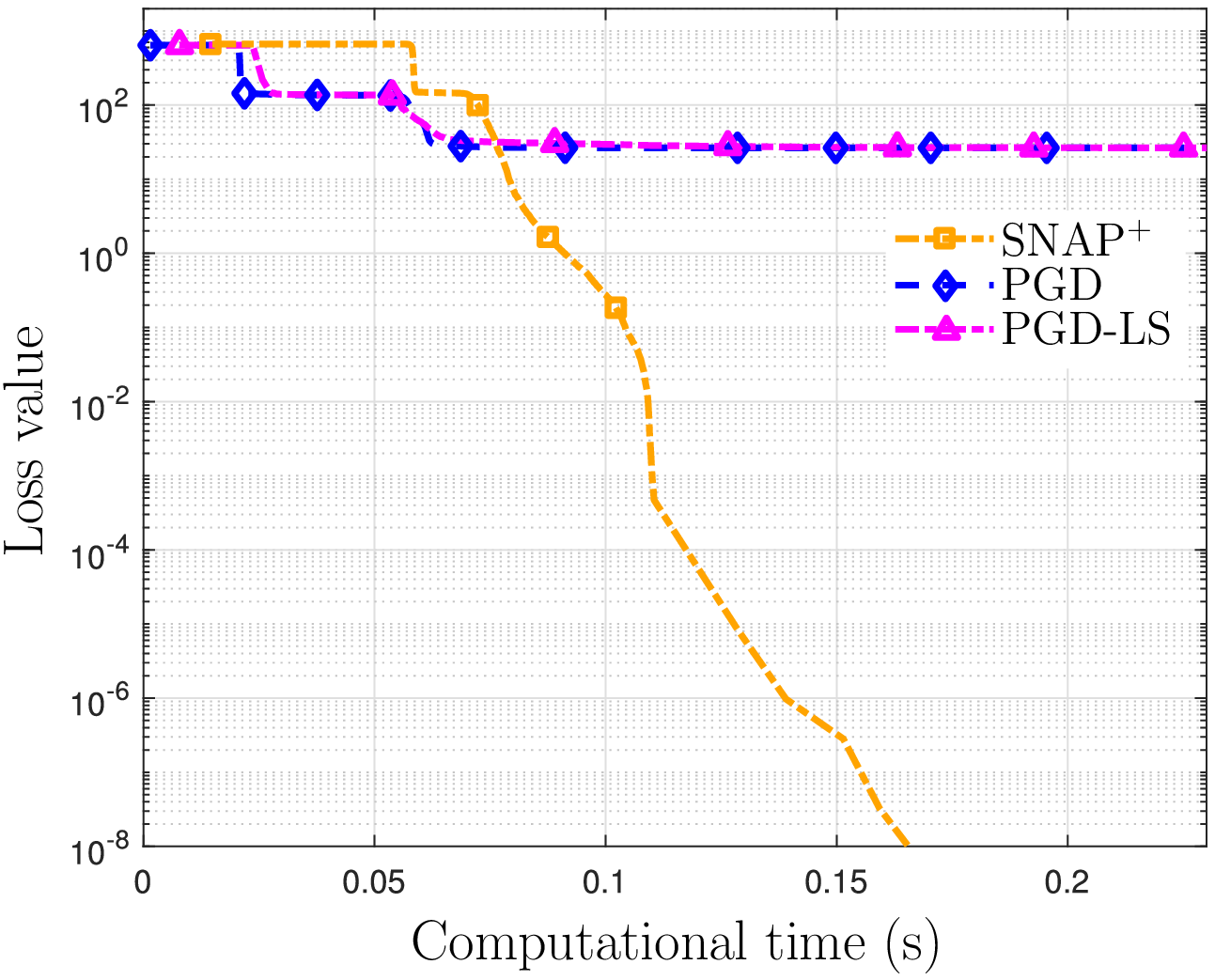}}
\caption{The convergence behaviors of SNAP$^+$, PGD, PGD-LS for penalized NMF, where $c=1\times 10^{-10}$.}
\label{fig:pnmf}
\end{figure}
\begin{figure}[ht]
\centering
\subfigure[Loss value versus iteration]{
\label{fig:lvi8}
\includegraphics[width=.4\linewidth]{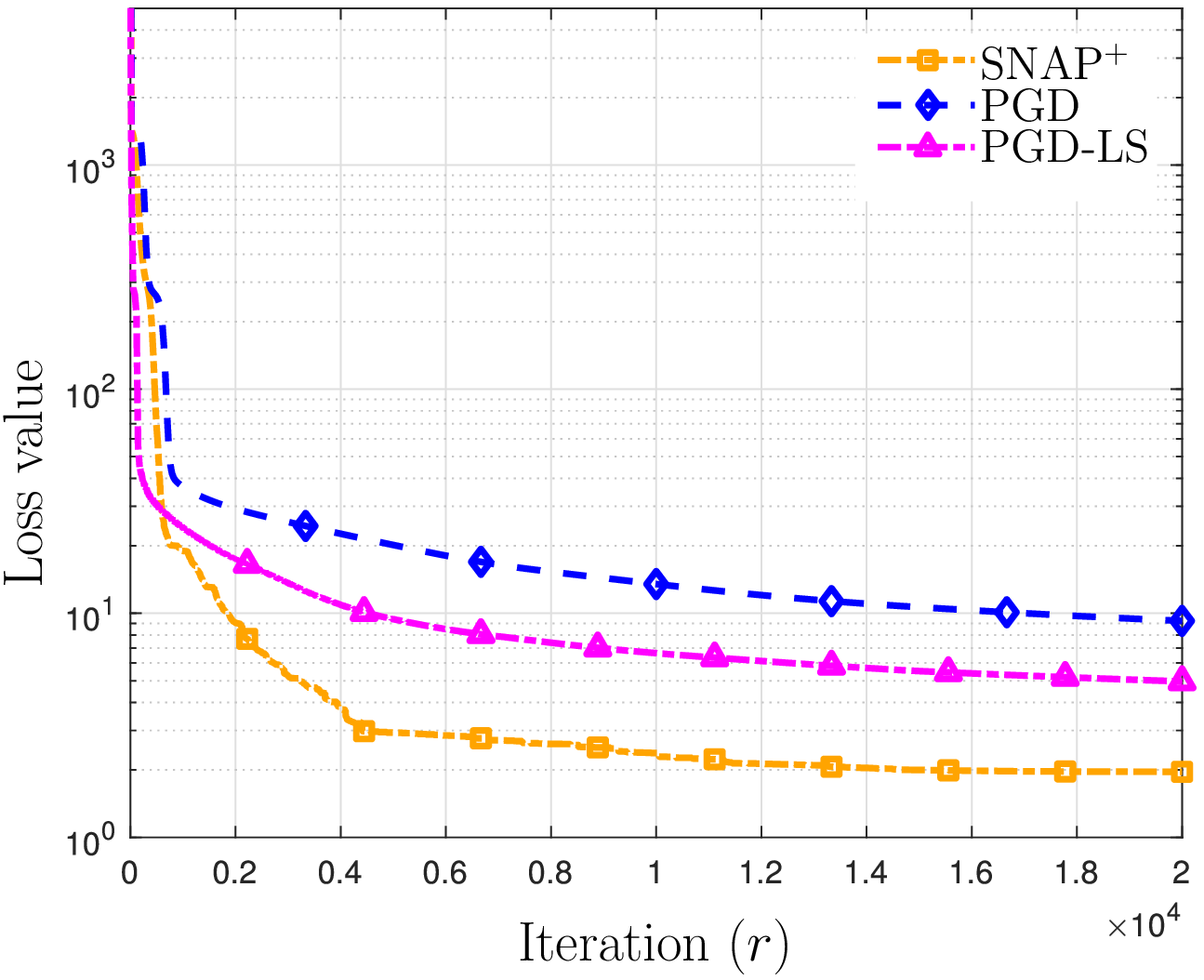}}
\hspace{0.4in}
\subfigure[{Loss value versus computational time}]{
\label{fig:lvt8}
\includegraphics[width=.4\linewidth]{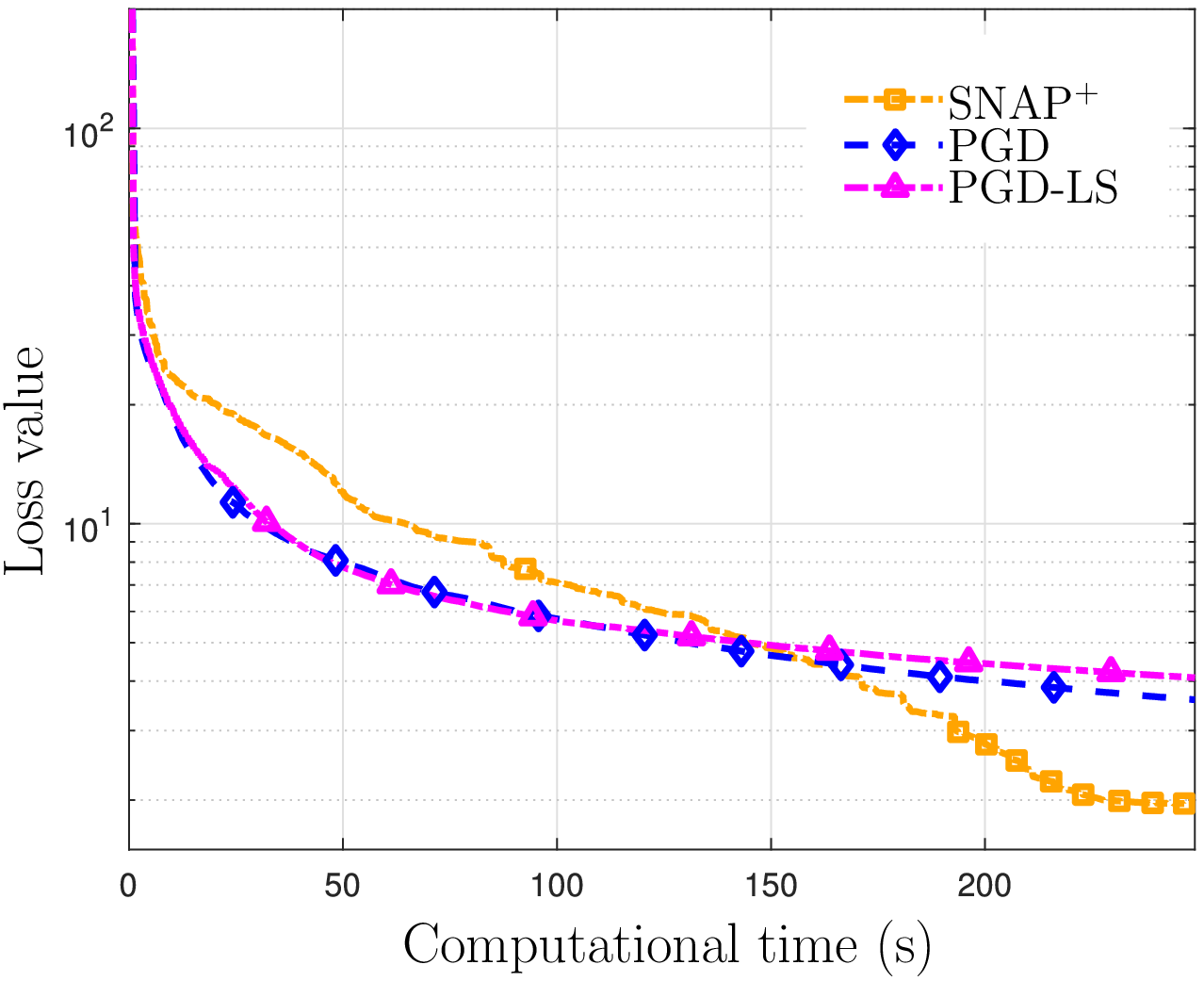}}
\caption{The convergence behaviors of SNAP$^+$, PGD, PGD-LS for penalized NMF, where $c=1$.}
\label{fig:pnmf3}
\vspace{-0.4cm}
\end{figure}

\subsection{Penalized NMF}
We also consider a penalized version of NMF, i.e.,
\begin{align}
\notag
\min_{\bW\in\mathbb{R}^{n\times k},\bH\in\mathbb{R}^{m\times k}} &\|\bW\bH^{\T}-\bM\|^2+\frac{\rho}{2}\sum_i^m\left((\b1^{\T}\bh_i)^2-\|\bh_i\|^2\right)
\\\notag
\textrm{s.t.}\quad&\bW\ge0,\bH\ge0
\end{align}
where $\bh_i$ denotes the columns of $\bH$.
It has been shown in \cite{chang19} that this variant of NMF could provide improved clustering accuracy, compared with the classic NMF.
Here, we only utilize this formulation to evaluate the performance of the algorithms. In the numerical experiments, we have the similar experimental step-up as the NMF case in \ref{sec:snmf}. The problem size is $m=100$, $n=40$ and $k=5$. We select $\rho=0.1$,  $\alpha_{\pi}=\beta=1\times 10^{-3}$, $T=100$, $r_{\textsf{th}}=20$, $\sR=1\times 10^{-4}$ and $\sF=100$. It can be observed from \figref{fig:pnmf},  SNAP$^+$ converges to the global minimum points of this penalized NMF problem while other ones converge to some points that have relatively large objective values.

\sledit{Further, we also implement the algorithms on a relatively larger problem, where $m=2000$, $n=50$, $k=5$.  In this case, we compare the algorithms by a large initialization, i.e., $c=1$. It can be observed from \figref{fig:pnmf3} that SNAP$^+$ converges faster with respect to the number of iterations.}

\clearpage
\newpage
\normalsize
\bibliographystyle{ieeetr}
\bibliography{refs}
\clearpage
\newpage
\normalsize
\appendix
\section{Proofs Related to Stationary Points}
\subsection{Proof of \prref{pr:slack1}}\label{sec:slack1}
\begin{proof}
When $f(\bx)=g(\bx)+\bq^{\T}\bx$, the KKT conditions of problem \eqref{eq.pro} are given by
\begin{subequations}
\begin{align}
    \nabla g(\bx^*)+\sum_{j\in\mathcal{A}(\bx^*)}\bmu^*_j\bA_j&=-\bq,
    \\
    \bA_j\bx^*&=\bb_j,\quad\bmu_j\ge0\quad\forall j\in\mathcal{A}(\bx^*),
    \\
    \bA_j\bx^*&<\bb_j,\quad\forall j,
    \\
    \bmu^*_j&=0,\quad\forall j\notin\mathcal{A}(\bx^*).
\end{align}
\end{subequations}
Since we assumed that $\bx^*$ has at least one active constraint, we have $|\mathcal{A}(\bx^*)|\ge 1$.

We prove the claim by contradiction. Assume that the  strict complementarity condition dose not hold at $\bx^*$. Without loss of generality, assume that $\bA_1\bx^*=\bb_1$ and $\bmu_1=0$. Consider the Lipschitz continuous map $\Phi$ defined below
\begin{equation}\label{eq.defphi}
\Phi(\bx^*,\bmu^*)=\nabla g(\bx^*)+\sum_{j\in\mathcal{A}(\bx^*)}\bmu^*_j\bA_j.
\end{equation}
This is a map that map a set $\mathcal{T}$ to the entire space of $\mathbb{R}^d$ (because $\bq$ is generated from a continuous measure in $\mathbb{R}^d$), where the set $\mathcal{T}$ is given below:
\begin{equation}
\mathcal{T}=\{(\bx^*,\bmu^*)|\bA_j\bx^*=\bb_j,j\in\mathcal{A}(\bx^*),\bA_j\bx^*<\bb_j,
\bmu^*_j=0,j\notin\mathcal{A}(\bx^*),\bmu^*\succeq0,\bA_1\bx^*=\bb_1,\bmu^*_1=0\}.
\end{equation}
In the following, we will quantify the dimension of $\mathcal{T}$. By assumption, all the $\bA_j$'s with $j\in\mathcal{A}(\bx^*)$ are linearly independent, that is, $\bA'(\bx^*)$ is a full row rank matrix. It follows that  $\bx^*$ is in the range space of matrix $\bA'(\bx^*)$. Since $\bA'(\bx^*)\bx^*=\bb'$, we know that the dimension of the active space of $\bx^*$ is the rank of $\bA'(\bx^*)$, meaning that the dimension of the free space of $\bx^*$ is the rank of $\textsf{Null}(\bA'(\bx^*))$, i.e., $(d-|\mathcal{A}(\bx^*)|)$. \footnote{For the notations of the free and active space, please see section \ref{sec:freespace}} Note that there are $|\mathcal{A}(\bx^*)|$ active constraints and $\bmu_1=0$, so the dimension of the free space of vector $\bmu$ is $(|\mathcal{A}(\bx^*)|-1)$. Therefore, $\Phi$ maps from a $(d-1)$-dimensional subspace to a $d$-dimensional space, implying that the image of the mapping is zero-measure in $\mathbb{R}^d$. However, $\bq$ is generated from a continuous measure, which results in a contradiction of the assumption that the strict complementarity condition does not hold. \end{proof}

\subsection{Proof of \coref{co:slack2}}\label{sec:slack2}
\begin{proof}
We apply the same proof technique in section \ref{sec:slack1} to show the claim of \coref{co:slack2}.  Let $\mathcal{S}(\bx^*)\bydef\{j\mid \bA_j,\forall j\in\mathcal{A}(\bx^*)\; \textrm{are linearly independent}\}$ and let $\overbar{\mathcal{S}}(\bx^*)$ denote the complement of set $\mathcal{S}(\bx^*)$. Clearly, $\mathcal{S}(\bx^*)$ is a subset of $\mathcal{A}(\bx^*)$. First,
we define the matrix $\bA''(\bx^*)$ as
\begin{align}\label{eq:A:sprime}
{\bA}''(\bx^*)\triangleq \left[\begin{array}{l} \vdots \\ \bA_j. \\ \vdots\end{array}\right]\in\mathbb{R}^{|\mathcal{S}(\bx^*)|\times d},\quad \forall j\in\mathcal{S}(\bx^*).
\end{align}
Obviously, $\bA''(\bx^*)$ is a full row rank matrix, where the rank of $\bA''(\bx^*)$ is the size of $\mathcal{S}(\bx^*)$, i.e., $|\mathcal{S}(\bx^*)|$. In the following, we will show that the number of simultaneously active constraints is at most $|\mathcal{S}(\bx^*)|$. We prove the claim by contradiction. Consider $i\in\overbar{\mathcal{S}}(\bx^*)$. Since $i\notin\mathcal{S}(\bx^*)$, $\bA_i$ can be linearly represented by $\bA_j$s $j\in\mathcal{A}(\bx^*)$, i.e., \begin{equation}\label{eq.linde}
 \bA_i=\sum_j\alpha_j\bA_j,j\in\mathcal{A}(\bx^*),
\end{equation} where there exists at least one $\alpha_j$ which is not zero. Since $i,j\in\mathcal{A}(\bx^*)$, we have $\bA_j\bx^*=\bb_j$. Combining \eqref{eq.linde}, we have $\sum_j\alpha_j\bb_j=\bb_i$. Since $\bb_i$ is generated from a continuous measure, $\sum_j\alpha_j\bb_j=\bb_i$ will not hold with high probability. We have a contradiction. Therefore, we can conclude that the dimension of the free space of $\bx^*$ is at least $d-|\mathcal{S}(\bx^*)|$.

Next, we use the same argument  as the proof of \coref{co:slack2} to quantify the dimension of $\bmu$. Since there are $|\mathcal{S}(\bx^*)|$ active constraints and $\bmu_1=0$, the dimension of $\bmu$ is at most $|\mathcal{S}(\bx^*)|-1$. Thus,  the dimension of $\mathcal{T}$ is $d-1$, meaning that $\Phi$ defined in \eqref{eq.defphi} maps from a $d-1$ dimension subset to a $d$-dimensional space. Therefore, the image is zero-measure in $\mathbb{R}^d$. However, $\bq$ is generated from a continuous measure, which again results in a contradiction of the assumption that the strict compementarity condition does not hold.
\end{proof}

\subsection{Equivalence of First-Order Conditions.}\label{sec:eq00}
\begin{lemma}\label{le:equiv}
The first-order conditions \eqref{eq.cond11} and \eqref{eq.cond21} are equivalent in the following sense: For any tuple $(\bx^*,\widetilde{\epsilon}_G)$ that satisfies \eqref{eq.cond21}, this pair also satisfies \eqref{eq.cond11}. Alternatively, for any tuple $(\bx^*,{\epsilon}_G)$ that satisfies \eqref{eq.cond11}, then  $(\bx^*,{\widetilde{\epsilon}}_G)$ satisfies \eqref{eq.cond21}, where ${{\widetilde{\epsilon}}_G\triangleq \frac{1}{2}\epsilon_G\left(\|\nabla f(\bx^*)\|+1+\alpha\epsilon_G\right).}$
\vspace{-5px}
\end{lemma}
\begin{proof}
	First we show that  if $\bx^*$ and $\widetilde{\epsilon}_G$ together satisfy \eqref{eq.cond21}, then they will also satisfy \eqref{eq.cond11}.
	
	Let us define $\widetilde{\bx}\triangleq \pi_{\mathcal{X}}(\bx^*-\alpha\nabla f(\bx^*))$, that is, from \eqref{eq:projection} we have
	\begin{equation}
	\widetilde{\bx}=\arg\min_{\by\in\mathcal{X}}\|\bx^*-\alpha\nabla f(\bx^*)-\by\|^2.\label{eq.xopt}
	\end{equation}
	From the optimality condition of \eqref{eq.xopt}, we know
	\begin{equation}
	\langle\widetilde{\bx}-(\bx^*-\alpha\nabla f(\bx^*)),\by-\widetilde{\bx}\rangle\ge0,\; \forall \by\in\mathcal{X}.\label{eq.optcondx}
	\end{equation}
	Substituting $\by=\bx^*$ into \eqref{eq.optcondx}, we have
	\begin{equation*}
	\langle\widetilde{\bx}-(\bx^*-\nabla f(\bx^*)),\bx^*-\widetilde{\bx}\rangle\ge0,
	\end{equation*}
	which implies $\alpha\langle \nabla f(\bx^*),\widetilde{\bx}-\bx^*\rangle\le-\|\widetilde{\bx}-\bx^*\|^2$.
	Therefore, we have
	\begin{equation*}
	\|\widetilde{\bx}-\bx^*\|\le-\alpha\left\langle \nabla f(\bx^*),\frac{\widetilde{\bx}-\bx^*}{ \|\widetilde{\bx}-\bx^*\|}\right\rangle\mathop{\le}\limits^{\eqref{eq.cond21}}\widetilde{\epsilon}_G,
	\end{equation*}
	meaning that $\|g_{\pi}(\bx^*)\|\le\widetilde{\epsilon}_G$.
	
	Second, we show that if $\bx^*\in\mathcal{X}$ and ${\epsilon}_G>0$ together satisfy  \eqref{eq.cond11}, then there exists $\widetilde{\epsilon}_G>0$, together with $\bx^*$ satisfy \eqref{eq.cond21}. Further, if $\epsilon_G\to  0$ then $\widetilde{\epsilon}_G\to 0$.
	
	Again let us define $\widetilde{\bx}\triangleq \pi_{\mathcal{X}}(\bx^*-\alpha\nabla f(\bx^*))$. Consider an arbitrary point $\by\in\mathcal{X}$ and $\widetilde{\bx}+\theta(\by-{\widetilde{\bx}})\in\mathcal{X}$ where $\theta\in(0,1)$. We have, for all $\by\in \mathcal{X}$, the following holds:
	\begin{align}
	\notag
	&\|\bx^*-\alpha\nabla f(\bx^*)-\widetilde{\bx}\|^2
	\\\notag
	\le&\|\bx^*-\alpha\nabla f(\bx^*)-({\widetilde{\bx}}+\theta(\by-{\widetilde{\bx}}))\|^2
	\\\notag
	=&\|\bx^*-\alpha\nabla f(\bx^*)-\widetilde{\bx}\|^2-2\theta\langle\bx^*-\alpha\nabla f(\bx^*)-\widetilde{\bx},\by-\widetilde{\bx}\rangle+\theta^2\|\by-\widetilde{\bx}\|^2,
	\end{align}
	which is equivalent to 
	\begin{equation}\label{eq.equip}
	\langle \bx^*-\alpha\nabla f(\bx^*)-\widetilde{\bx},\by-\widetilde{\bx}\rangle\le\frac{\theta}{2}\|\by-\widetilde{\bx}\|^2,\;\;\forall \by\in\mathcal{X}.
	\end{equation}
	The right-hand-side (RHS) of \eqref{eq.equip} can be made arbitrarily small by $\theta$ for a given $\by$, so LHS of \eqref{eq.equip} cannot be strictly positive. Therefore, we have
	\begin{equation}
	\left\langle \bx^*-\alpha\nabla f(\bx^*)-\widetilde{\bx},\by-\widetilde{\bx}\right\rangle\le 0,
	\end{equation}
	which is
	\begin{align}
	\notag
	\langle\nabla f(\bx^*),\by-{\bx^*}\rangle\ge&\frac{1}{\alpha}\langle\bx^*-\widetilde{\bx},\by-\widetilde{\bx}\rangle+\frac{1}{\alpha}\langle\nabla f(\bx^*),\widetilde{\bx}-\bx^*\rangle
	\\
	\mathop{\ge}\limits^{(a)}&-\frac{1}{2}\epsilon_G\left(\|\nabla f(\bx^*)\|+\|\by-\widetilde{\bx}\|\right)
	\\
	\ge&-\frac{1}{2}\epsilon_G\left(\|\nabla f(\bx^*)\|+\|\by-\bx^*\|+\|\bx^*-\widetilde{\bx}\|\right)
	\\
	\mathop{\ge}\limits^{(b)}&-\frac{1}{2}\epsilon_G\left(\|\nabla f(\bx^*)\|+1+\alpha\epsilon_G\right)\triangleq -\widetilde{\epsilon}_G
	\end{align}
	where in $(a)$ we use $(1/\alpha)\|\bx^*-\widetilde{\bx}\|\le\epsilon_G$ and Cauchy–Schwartz inequality, in $(b)$ we know $\|\by-\bx^*\|\le 1$ from condition \eqref{eq.cond21}.
\end{proof}

\subsection{Proof of \prref{pr:00}}\label{sec:pr00}
\begin{proof}

The equivalence between the first-order conditions \eqref{eq.cond1} and \eqref{eq.cond2} is obvious, see the proof of Lemma \ref{le:equiv}. Below, we focus on the equivalence of the second-order conditions. We need to show that at a given solution $\bx^*$, if for every $\by$ that satisfies the following condition arising in \eqref{eq.cond12},
\begin{equation}
\bA_j\by=0,\quad \forall j\in\mathcal{A}(\bx^*) \quad\textrm{and}\quad\by\neq 0\label{eq.condk4},
\end{equation}
we have $\by^T \nabla^2 f(\bx^*)\by\ge {0}$, then
 we must have
 $$(\bx-\bx^*)^{\T}\nabla^2f(\bx^*)(\bx-\bx^*)\ge {0}, \; \forall\bx\in\mathcal{X},\langle\nabla f(\bx^*),\bx-\bx^*\rangle=0.$$

 Conversely, if, for every $\bx\in \mathcal{X}$ that satisfies  $\nabla f(\bx^*)^{\T}(\bx-\bx^*)=0$, we have $(\bx-\bx^*)^{\T}\nabla^2f(\bx^*)(\bx-\bx^*)\ge{0}$. Then the following must hold
 $$\forall\by\neq0,\quad\bA_j\by=0,\quad \forall j\in\mathcal{A}(\bx^*), \by\nabla^2f(\bx^*)\by\ge {0}.$$
If the above two directions both hold, then the conditions \eqref{eq.cond12} and \eqref{eq.cond22} will imply each other.

\noindent{\bf Part I.}  First, assume that \eqref{eq.cond12} holds. {For a given $\bx$ satisfying $\langle\nabla f(\bx^*),\bx-\bx^*\rangle=0$. By applying \eqref{eq.condk1},  we have $\sum^m_{j=1}\bmu^*_j\langle\bA_j,\bx-\bx^*\rangle=0$. Further by \eqref{eq.defofm}}, we can decompose the previous sum into the following according to whether the constraints are active or not:
\begin{equation}
\sum_{j\in\mathcal{A}(\bx^*)}\bmu^*_j\langle\bA_j,\bx-\bx^*\rangle+\sum_{j\notin\mathcal{A}(\bx^*)}\bmu^*_j\langle\bA_j,\bx-\bx^*\rangle=0.\label{eq.cond11tp10}
\end{equation}
Combining \eqref{eq.cond11tp10} with the complementarity conditions in \eqref{eq.condk2}, we have
\begin{equation}
\sum_{j\in\mathcal{A}(\bx^*)}\bmu^*_j\langle\bA_j,\bx-\bx^*\rangle=0\label{eq.frcond}.
\end{equation}
Also note that for each $\bx\in \mathcal{X}$, and each active constraint $j\in \mathcal{A}(\bx^*)$, we have
$$\bA_j\bx\le\bb_j=\bA_j\bx^*, \quad \forall~j\in \cA(\bx^*).$$
It follows that
$$\langle\bA_j,\bx-\bx^*\rangle  \le 0, \; \forall~j\in \cA(\bx^*),\quad\textrm{and}\quad\forall\bx\in\mathcal{X}.$$
Due to the assumed strict complementarity condition, we have  $\bmu^*_j>0,\; j\in\mathcal{A}(\bx^*)$.

Combining the above two facts, we can conclude that each term in the summation in \eqref{eq.frcond} is nonpositive. However, the requirement that the sum of them equals to zero implies that :
\begin{equation}
\langle\bA_j,\bx-\bx^*\rangle=0,\;  \forall~j\in\mathcal{A}(\bx^*).
\end{equation}
From \eqref{eq.cond12}, we know that $\forall \by, \bA'(\bx^*)\by=0$, we have { $\by^{\T}\nabla^2f(\bx^*)\by\ge 0$}, so we have \eqref{eq.cond22}.

{\noindent{\bf Part II.} Second, let us suppose that $\bx^*$ satisfies the exact first-order stationary solution, and for each {\it feasible} $\bx\in \mathcal{X}$ that satisfies the following
\begin{align*}
\langle \nabla f(\bx^*), \bx-\bx^*\rangle =0,
\end{align*}
 we have
\begin{align}\label{eq:use}
(\bx-\bx^*)^{\T} \nabla^2 f(\bx^*) (\bx-\bx^*)\ge 0.
\end{align}
Suppose at the KKT point $\bx^*, \bmu^*$ the strict complementarity condition is satisfied. Further we assume that for the inactive set, the following holds:
\begin{align*}
\bA_i \bx^* + \epsilon_i = \bb_i, \; \mbox{for some}~\epsilon_i>0, \; \; \forall~i\in \bar{\cA}(\bx^*).
\end{align*}
Let $\epsilon = \min_{i}\{\epsilon_i\}>0$.

We take the inner product between $\bx-\bx^*$ and left-hand-side (LHS) of \eqref{eq.condk1} and can obtain 
\begin{equation}\label{eq.fbyi}
\langle\nabla f(\bx^*),\bx-\bx^*\rangle=-\sum^m_{j=1}\bmu^*_j\langle\bA_j,\bx-\bx^*\rangle\mathop{=}\limits^{\eqref{eq.condk2}}-\sum_{j\in\mathcal{A}(\bx^*)}\bmu^*_j\langle\bA_j,\bx-\bx^*\rangle.
\end{equation}
Since strict complementarity condition is satisfied, we have $\bmu^*_j>0,\forall~j\in \cA(\bx^*)$. Then we have
\begin{align*}
\bA_j (\bx-\bx^*) = 0, \forall~j\in \cA(\bx^*).
\end{align*}
Then let us consider any $\by$ that satisfies
\begin{align}\label{eq:temp2}
\bA_j \by= 0, \quad \forall~j\in \cA(\bx^*).
\end{align}
First, we argue that, if the following holds
\begin{align}\label{eq:condition}
\bA_j \by = 0, \; \forall~j\in {\cA}(\bx^*), \quad \bA_i \by \le \frac{\epsilon}{2}, \; \forall~i\in \bar{\cA}(\bx^*),
\end{align}
then there must exist $\bx\in \mathcal{X}$ such that $\by = \bx - \bx^*$. 

By setting $\by = \bz - \bx^*$, for any $\bz$, we obtain
\begin{align*}
\bA_i (\bz - \bx^*) \le \epsilon/2,
\end{align*}
which implies
\begin{align*}
\bA_i \bz  \le \bA_i \bx^* +\frac{\epsilon}{2} < \bb_i
\end{align*}
where the last inequality is due to the definition of $\epsilon$.  Further, for the active set, it is clear that
\begin{align*}
\bA_j \bz = \bA_j \bx^* =\bb_j,  \; \forall~j\in {\cA}(\bx^*).
\end{align*}
Therefore $\bz$ is feasible.

So, suppose that for a given $\by$ satisfying \eqref{eq:temp2}, we cannot find any $\bx\in \mathcal{X}$ such that $\by=\bx-\bx^*$, then it must be the case that there exists a subset $\mathcal{Q} \in \bar{\cA}(\bx^*)$ such that
\begin{align*}
\bA_q \by =\theta_q >\frac{\epsilon}{2}, \; \forall~q\in \mathcal{Q}.
\end{align*}
Let us define $\theta_{\max} := \max_q\{\theta_{q}\}$, and
\begin{align}
\tilde{\by} = \frac{1}{\theta_{\max}}\frac{\epsilon}{2}\by
\end{align}
note that $\frac{1}{\theta_{\max}}\frac{\epsilon}{2}<1$.
Then for this new $\tilde{\by}$,  the following holds
\begin{align}
&\bA_j \tilde{\by}  =0, \; \forall~j\in \cA(\bx^*)\nonumber\\
&\bA_q \tilde{\by}  \le \epsilon/2, \; \forall~q\in \mathcal{Q}\nonumber\\
&\bA_j \tilde{\by}  = \frac{1}{\theta_{\max}} \frac{\epsilon}{2}\bA_j \by, \; \forall~j\in \bar{\cA}(\bx^*), j\notin \mathcal{Q}\nonumber.
\end{align}
Note that for all $j\in \bar{\cA}(\bx^*), j\notin \mathcal{Q}$, $\bA_j \by \le \epsilon/2$. We have the following two cases for those indices.

{\bf Case 1.} First, if $\bA_j \by \le 0$ then it is clear that
$$\bA_j \tilde{\by}  = \frac{1}{\theta_{\max}} \frac{\epsilon}{2}\bA_j \by \le 0.$$

{\bf Case 2.} Second, if $0\le \bA_j \by \le \epsilon/2$ then it is clear that
$$\bA_j \tilde{\by}  = \frac{1}{\theta_{\max}} \frac{\epsilon}{2}\bA_j \by \stackrel{(a)} \le  \bA_j \by \le \epsilon/2 .$$
where $(a)$ uses the fact that $\frac{1}{\theta_{\max}}\frac{\epsilon}{2}<1$, and $0\le \bA_j\by$.

Overall we have $\bA_i\tilde{\by}\le \epsilon/2, \forall~i\in\bar{\cA}(\bx^*)$, and $\bA_j\tilde{\by}=0, \; \forall~i\in\cA(\bx^*)$, i.e., condition \eqref{eq:condition} holds for $\tilde{\by}$.

Therefore, there must exist $\bx \in \mathcal{X}$ such that
$\tilde{\by} =\bx -\bx^*$. We conclude that for any $\by$ satisfying $\bA_i \by =0, \; \forall~i\in\cA(\bx^*)$, there exists a constant $\theta$ and $\bx\in \mathcal{X}$ such that $\theta \by = \bx-\bx^*$. By \eqref{eq:use}, we obtain
\begin{align}
\theta^2 \by^{\T} \nabla^2 f(\bx^*) \by \ge 0, \;\;\;\; \mbox{or equivalently,}\;\;\;\; \by^{\T} \nabla^2 f(\bx^*) \by \ge 0.
\end{align}
This direction is proved.}
\end{proof}

\subsection{Proof of Corollary 2.}
\begin{proof}
It can be easily checked that \coref{co:0eh} is true from the proof of \prref{pr:00} by letting $\by=\bx-\bx^*$ and considering $\|\by\|\le1$.
\end{proof}

\subsection{Proof of \prref{pr:continuity}}\label{sec:continuity}
\begin{proof}
Let $\bx^*$ be a limit point of the sequence $\{\bx^{(r)}\}$. By restricting to a subsequence if necessary, let us assume that $\lim_{r\rightarrow \infty} \bx^{(r)} = \bx^*$.
First notice that the function $g_{\pi} (\cdot)$ is continuous based on its definition. Therefore, $g_{\pi}(\bx^*) = \lim_{r\rightarrow \infty} g_{\pi}(\bx^r) = 0$. Therefore, \eqref{eq.cond11:exact} is satisfied at the point $\bx^*$.

In order to show \eqref{eq.cond12:exact}, let us define $\mathcal{Y}^{(r)} \triangleq \{\by \; | \; \bA'(\bx^{(r)}) \by = 0\}$ and $\mathcal{Y}^* \triangleq \{\by \; | \; \bA'(\bx^*) \by = 0\}$. We first prove that there exists an index $r'$ such that  $\mathcal{Y}^* \subseteq \mathcal{Y}^{(r)}, \;\forall r\geq r'$. To show that, first consider an inactive index  $j \in \bar{\mathcal{A}} (\bx^*)$. Clearly, $\bA_j \bx^* \neq \bb$ and therefore, there exists an index $r_j'$ such that $\bA_j \bx^{(r)} \neq \bb, \;\forall r \geq r_j'$. Thus, $j \in \bar{\mathcal{A}} (\bx^{(r)}), \;\forall r\geq r_j'$. By repeating this argument for all indices $j$ and setting $r' = \max_j \{r_j\}$, we have $ \bar{\mathcal{A}} (\bx^*) \subseteq  \bar{\mathcal{A}} (\bx^{(r)}),\;\forall r \geq r'$. Therefore, $\mathcal{A}(\bx^{(r)}) \subseteq \mathcal{A}(\bx^*),\;\forall r \geq r'$, which immediately implies that
\begin{equation}
    \mathcal{Y}^* \subseteq \mathcal{Y}^{(r)}, \;\forall r\geq r'. \label{eq:AssymptYr}
\end{equation}
Furthermore, using the definition of Exact SOSP1, we have
\begin{equation}
    \begin{split}
        -\epsilon_H^{(r)} \leq \min_{\by} \quad  &\by^{\T} \nabla^2 f(\bx^{(r)}) \by\nonumber\\
        \st \quad & \by \in   \mathcal{Y}^{(r)}. \nonumber
    \end{split}
\end{equation}
By letting $r \rightarrow \infty$ and using \eqref{eq:AssymptYr}, we obtain
\begin{equation}
    \begin{split}
        0 \leq \min_{\by} \quad  &\by^{\T} \nabla^2 f(\bx^*) \by\nonumber\\
        \st \quad & \by \in   \mathcal{Y}^*. \nonumber
    \end{split}
\end{equation}
Therefore, \eqref{eq.cond12:exact} is satisfied at the point $\bx^*$.
\end{proof}

\section{Proofs Related to the Implementation of SNAP}
\subsection{Proof of \leref{le.selection}}
\begin{proof}
{{Suppose $-q_{\pi}(\bx^{(r)})$ is chosen, and if $\alpha_{\max}$ in Algorithm is not chosen (i.e., if line 8 of Algorithm \ref{alg:p3} does not hold true). Then according to \eqref{eq.des1} and \eqref{eq.desofq} in the proof of \leref{le.2} (which can be found in Appendix \ref{sec:snap:proof}),  the minimum descent of the objective is $d_q\bydef-\frac{3}{8L_1}\|q_{\pi}(\bx^{(r)})\|^2$. This is because the lower bound of $\alpha$ is $1/(2L_1)$ when $-q_{\pi}(\bx^{(r)})$ is chosen; see the proof of \leref{le.2} for details.}

Further, if $\bv(\bx^{(r)})$ is chosen, the minimum amount of the descent is given  by \eqref{eq.des2} and \eqref{eq.derho}} 
\begin{align}
d_{\bv} \bydef\alpha ({q_{\pi}(\bx^{(r)})^{\T}\bv(\bx^{(r)})}-\left(1-\frac{\alpha L_2}{3\epsilon'_H(\delta)}\right)\frac{\alpha^2\epsilon'_H(\delta)}{2}.
\end{align}	
Since the lower bound of $\alpha$ is $3\epsilon'_H(\delta)/(8L_2)$  if $\bd^{(r)}$ is chosen by $\bv(\bx^{(r)})$, the minimum descent is
\begin{equation}
    d_{\bv}=\frac{3\epsilon'_H(\delta)}{8L_2} {q_{\pi}(\bx^{(r)})^{\T}\bv(\bx^{(r)})}-\frac{189}{1024}\frac{\epsilon'^3_H(\delta)}{L^2_2}.
\end{equation}

{If $d_q<d_{\bv}$, or equivalently  \begin{equation}
    \frac{L_1\epsilon'_H(\delta)}{2L_2}q_{\pi}(\bx^{(r)})^{\T}\bv(\bx^{(r)})-\frac{63L_1\epsilon_H'^3(\delta)}{128L^2_2}\ge-\|q_{\pi}(\bx^{(r)})\|^2
\end{equation} 
where $q_{\pi}(\bx^{(r)})^{\T}\bv(\bx^{(r)})\le0$, it implies that choosing $q_{\pi}(\bx^{(r)})$ may provide more descent of the objective value. }

\end{proof}

\subsection{Proof of \leref{le.alpha}}

\begin{proof}
Since $\bx^{(r)}$ is within the feasible set, then based on the definition of inactive set we have $\overbar{\bA}'(\bx^{(r)})\bx^{(r)}<\overbar{\bb}'$. The largest step-size along the direction $\bd^{(r)}$ is determined by the largest distance in which the boundary of the feasible solution will be touched, see \eqref{eq.alphamax}. According to the update rule of the iterate, we need
\begin{equation}
    \overbar{\bA}'(\bx^{(r)})(\bx^{(r)}+\alpha\bd^{(r)})\le\overbar{\bb}',
\end{equation}
which is equivalent to the component-wise form, i.e.,
\begin{equation}\label{eq.feasofal}
    \alpha\left(\overbar{\bA}'(\bx^{(r)})\bd^{(r)}\right)_i\le\left(\overbar{\bb}'-\overbar{\bA}'(\bx^{(r)})\bx^{(r)}\right)_i,\forall i.
\end{equation}
Due to the feasibility of $\bx^{(r)}$, and the definition of inactive set, we have $(\overbar{\bb}'-\overbar{\bA}'(\bx^{(r)})\bx^{(r)})_i>0$. Then there are two cases as follows:

\begin{enumerate}
    \item $(\overbar{\bA}'(\bx^{(r)})\bd^{(r)})_i\le0$: any $\alpha>0$ can satisfy \eqref{eq.feasofal}.
    \item $(\overbar{\bA}'(\bx^{(r)})\bd^{(r)})_i>0$: we need $\alpha\le(\overbar{\bb}'-\overbar{\bA}'(\bx^{(r)})\bx^{(r)})_i/(\overbar{\bA}'(\bx^{(r)})\bd^{(r)})_i$.  That is, going along the current direction far enough will eventually reach the boundary of the feasible set.
\end{enumerate}
Then it follows if there exists a {\it finite} step-size ${\alpha}$ so that 
\begin{align}\label{eq:reach:boundary}
(\overbar{\bA}'(\bx^{(r)})(\bx^{(r)}+{\alpha} \mathbf{d}^{(r)}))_i = \overbar{\bb}'_i,
\end{align}
 we can easily compute $\alpha^{(r)}_{\max}$ in the closed-form by \eqref{eq.closealpha}.
 \end{proof}

\section{Proofs of SNAP}\label{sec:snap:proof}
 In this section, we  show that SNAP converges to an $(\epsilon_G,\epsilon_H)$-SOSP1 in a finite number of steps. In \algref{alg:p3}, it can be observed that $\bd^{(r)}$ could be chosen by projected gradient $q_{\pi}(\bx^{(r)})$ or negative curvature $\bv(\bx^{(r)})$. Using the line search algorithm ensures that the iterates stay in the feasible set. When $\alpha^{(r)}_{\max}$ is chosen by \eqref{eq.alphamax}, the objective function will not increase. When $\alpha^{(r)}_{\max}$ is not chosen by \eqref{eq.alphamax}, we will have a sufficient descent. We will give the following three lemmas that quantify the minimum decrease of the objective value by implementing one step of the algorithm, i.e., $\bx^{(r+1)}=\bx^{(r)}+\alpha^{(r)}\bd^{(r)}$. They serve as the stepping stones for the main result that follows.

The descent Lemma of PGD is given by the following.

\subsection{Descent Lemmas}
\begin{lemma}\label{le.1}
If $\bx^{(r+1)}$ is computed by projected gradient descent with step-size chosen by $1/L_1$, then $f(\bx^{(r+1)})\le f(\bx^{(r)})-\frac{\epsilon^2_G}{18L_1}$.
\end{lemma}
\begin{proof}
The proof follows the classic theory of the projected gradient descent. According to the optimality condition of the projection, we have
\begin{equation}
\left\langle \bx^{(r+1)}-(\bx^{(r)}-\alpha_\pi\nabla f(\bx^{(r)})),\bx-\bx^{(r+1)}\right\rangle\ge0\quad \bx\in\mathcal{X}.
\end{equation}
Applying this relation with $\bx=\bx^{(r)}$, we obtain
\begin{equation}
\left\langle\nabla f(\bx^{(r)}),\bx^{(r+1)}-\bx^{(r)}\right\rangle\le-\frac{1}{\alpha_\pi}\|\bx^{(r+1)}-\bx^{(r)}\|^2.
\end{equation}
	
According to $L_1$-Lipschitz continuity, we have
\begin{equation}
f(\bx^{(r+1)})-f(\bx^{(r)})\le\nabla f(\bx^{(r)})^{\T}(\bx^{(r+1)}-\bx^{(r)})+\frac{L_1}{2}\|\bx^{(r+1)}-\bx^{(r)}\|^2,
\end{equation}
where
\begin{equation}
\bx^{(r+1)}=\pi_{\mathcal{X}}(\bx^{(r)}-\alpha_\pi\nabla f(\bx^{(r)})).
\end{equation}
	
Then, we have
\begin{equation}
f(\bx^{(r+1)})\le f(\bx^{(r)})+\left(\frac{L_1}{2}-\frac{1}{\alpha_\pi}\right)\|\bx^{(r+1)}-\bx^{(r)}\|^2,
\end{equation}
where $0<\alpha_\pi\le 1/L_1$, implying
\begin{align}
\notag
f(\bx^{(r+1)})\le & f(\bx^{(r)})-\frac{L_1}{2}\|\bx^{(r+1)}-\bx^{(r)}\|^2
\\\notag
\mathop{\le}\limits^{(a)}&f(\bx^{(r)})-\frac{L_1}{2(\frac{2}{\alpha_{\pi}}+L_1)^2}\|g_{\pi}(\bx^{(r)})\|^2
\\\notag
\mathop{\le}\limits^{(b)}&f(\bx^{(r)})-\frac{\epsilon^2_G}{18L_1}
\label{eq.objpgd}
\end{align}
where in $(a)$ we use the nonexpansiveness of the projection operator,
and the details are as follows:
\begin{align}\notag
\|g_{\pi}(\bx^{(r)})\|=&\frac{1}{\alpha_{\pi}}\|\pi_{\mathcal{X}}(\bx^{(r)}-\alpha_{\pi}\nabla f(\bx^{(r)}))-\bx^{(r)}\|
\\\notag
=&\frac{1}{\alpha_{\pi}}\|\pi_{\mathcal{X}}(\bx^{(r)}-\alpha_{\pi}\nabla f(\bx^{(r)}))-\bx^{(r+1)}+\bx^{(r+1)}-\bx^{(r)}\|
\\\notag
\le&\frac{1}{\alpha_{\pi}}\|\bx^{(r+1)}-\bx^{(r)}\|+\frac{1}{\alpha_{\pi}}\|\bx^{(r+1)}-\pi_{\mathcal{X}}(\bx^{(r)}-\alpha_{\pi}\nabla f(\bx^{(r)}))\|
\\\notag
=&\frac{1}{\alpha_{\pi}}\|\bx^{(r+1)}-\bx^{(r)}\|+\frac{1}{\alpha_{\pi}}\|\pi_{\mathcal{X}}(\bx^{(r+1)}-\alpha_{\pi}\nabla f(\bx^{(r+1)}))-\pi_{\mathcal{X}}(\bx^{(r)}-\alpha_{\pi}\nabla f(\bx^{(r)}))\|
\\\notag
\le&\frac{2}{\alpha_{\pi}}\|\bx^{(r+1)}-\bx^{(r)}\|+\|\nabla f(\bx^{(r+1)})-\nabla f(\bx^{(r)})\|
\\
\le&\left(\frac{2}{\alpha_{\pi}}+L_1\right)\|\bx^{(r+1)}-\bx^{(r)}\|;
\end{align}
in $(b)$ we take $\alpha_{\pi}=1/L_1$.

From \eqref{eq.objpgd}, we have the sufficient descent of the objective value if the constant step-size is used.
\end{proof}

\begin{lemma}\label{le.2}
If $\bd^{(r)}$ is chosen as $-q_{\pi}(\bx^{(r)})$ and $\bx^{(r+1)}$ is computed by the NCD step of \algref{alg:p1}, {and if $\alpha^{(r)}_{\max}$ is not selected by the line search algorithm,} then {the line search algorithm terminates with} $\alpha\ge1/(2L_1)$ and a descent of the following can be achieved $f(\bx^{(r+1)})\le f(\bx^{(r)})-{0.18\epsilon'^3_H(\delta)/L^2_2}$.
\end{lemma}

\begin{proof}
If $\bd^{(r)}=-q_{\pi}(\bx^{(r)})$ in line 9 of \algref{alg:p1}, then the proof follows the classical gradient descent algorithm with the Armijo rule, showing that the objective values obtained by PGD achieves sufficient descent when $\alpha$ is small. \mhcomment{below how are the two cases are devided? the first case is when $\alpha$ can go very far, the second is when $\alpha_{\max}$ is chosen?}

First, according to the $L_1$-Lipschitz continuity, we have
\begin{align}
\notag
f(\bx^{(r+1)})=&f\left(\bx^{(r)}-\alpha q_{\pi}(\bx^{(r)})\right)
\\\notag
\mathop{\le}\limits^{(a)}& f(\bx^{(r)})-\alpha \nabla f(\bx^{(r)})^{\T}q_{\pi}(\bx^{(r)})+\frac{\alpha^2}{2}L_1\|q_{\pi}(\bx^{(r)})\|^2
\\\notag
\mathop{=}\limits^{(b)}& f(\bx^{(r)})-\alpha q_{\pi}(\bx^{(r)})^{\T}q_{\pi}(\bx^{(r)})+\frac{\alpha^2}{2}L_1\|q_{\pi}(\bx^{(r)})\|^2
\\\notag
=& f(\bx^{(r)})-\alpha\|q_{\pi}(\bx^{(r)})\|^2+\frac{\alpha^2}{2}L_1\|q_{\pi}(\bx^{(r)})\|^2
\\
=&f(\bx^{(r)}) -\left(\alpha-\frac{\alpha^2}{2}L_1\right)\|q_{\pi}(\bx^{(r)})\|^2\label{eq:descent}
\end{align}
where in $(a)$ we use the gradient Lipschitz continuity; $(b)$ is true because $\nabla f(\bx^{(r)})=\bP(\bx^{(r)})\nabla f(\bx^{(r)})+\bP_\perp(\bx^{(r)})\nabla f(\bx^{(r)})$ and $q_{\pi}(\bx^{(r)})=\bP(\bx^{(r)})\nabla f(\bx^{(r)})$. It can be observed that there must exist a small $\alpha$ such that the objective is decreased, so the line search algorithm will be terminated within finite number of steps. \mhcomment{specific, which condition is satisfied? about the $\rho$..}
	
Second, by the definition of $\alpha^{(r)}_{\max}$, we know that along the direction $-q_{\pi}(\bx^{(r)})$, one can go with a step of length at least $\alpha^{(r)}_{\max}$ without hitting the boundary. Then we can determine a lower bound of $\alpha^{(r)}_{\max}$ as follows. We divide this analysis into two steps. \mhcomment{not clear what you want to show here. You want to say that, if the algorithm enters line 11 of the line search, then the intiail $\alpha^r_{\max}$ must be lower bounded xxxx. }

\noindent {\bf Step (a)} Suppose that $\alpha^{(r)}_{\max}$ does meet the criteria \eqref{eq.back}, \mhedit{that is
\begin{align}
f(\bx^{(r)}+\alpha^{(r)}_{\max} \bd^{(r)}) > f(\bx^{(r)})+{\frac{1}{2}}\rho(\alpha^{(r)}_{\max}).
\end{align}}
Then we have $\alpha^{(r)}_{\max}\ge \frac{1}{L_1}$, because otherwise,
\begin{align}
 \alpha^{(r)}_{\max} \le \frac{1}{L_1} \Rightarrow \left(\alpha^{(r)}_{\max} - \frac{(\alpha^{(r)}_{\max})^2 L_1}{2}\right)\ge \frac{1}{2} \alpha^{(r)}_{\max}.
\end{align}
The above fact combined with the descent estimate \eqref{eq:descent} implies that	\eqref{eq.back}  stops to hold true, which is a contradiction.

\noindent  {\bf Step (b)} Suppose that $\alpha^{(r)}_{\max}$ does not meet the criteria \eqref{eq.back}. However this would imply that line 8 of the line search algorithm will hold, so the algorithm has already returned -- again a contradiction.
Therefore we conclude that the initial stepsize $\alpha^{(r)}_{\max}$ is lower bounded by $1/L_1$.

By the backtracking algorithm from \eqref{eq.back}, we can  find an $\alpha$ that is at least $1/(2L_1)$ such that
\begin{align}
f(\bx^{(r+1)})\le & f(\bx^{(r)})-\frac{3}{8L_1}\|q_{\pi}(\bx^{(r)})\|^2\label{eq.desofq}
\\
\mathop{\le}\limits^{(a)} & f(\bx^{(r)})- {0.18\frac{\epsilon'^3_H(\delta)}{L^2_2},}
\end{align}
where in $(a)$  we use $\|q_{\pi}(\bx^{(r)})\|^2\ge 63L_1\epsilon'^3_H(\delta)/(128L^2_2)$ since in line 8 of \algref{alg:p1} we know from the algorithm that $-q_{\pi}(\bx^{(r)})$ is chosen when { $$-q_{\pi}(\bx^{(r)})^{\T}\bv(\bx^{(r)})\frac{L_1\epsilon'_H(\delta)}{L_2}+\frac{63L_1\epsilon_H'^3(\delta)}{128L^2_2}\le\|q_{\pi}(\bx^{(r)})\|^2, \quad q_{\pi}(\bx^{(r)})^{\T}\bv(\bx^{(r)})\le 0.$$}
This completes the proof.
\end{proof}

\begin{lemma}\label{le.3}
If $\bd^{(r)}$ is chosen by $\bv(\bx^{(r)})$,  $\bx^{(r+1)}$ is computed by the NCD procedure in \algref{alg:p1} {and $\alpha^{(r)}_{\max}$ is not selected, then {the line search algorithm terminates with} $\alpha\ge 9\epsilon'_H(\delta)/(4L_2)$ and a descent of the following can be achieved: $f(\bx^{(r+1)})\le f(\bx^{(r)})-0.06\epsilon'^3_H(\delta)/L^2_2$.}
\end{lemma}
\begin{proof}
If $\bd^{(r)}=\bv(\bx^{(r)})$ in line 11 of \algref{alg:p1}, according to the $L_2$-Lipschitz continuity, we have
\begin{align}
\notag
&f(\bx^{(r)}+\alpha \bv(\bx^{(r)}))
\\\notag
\le & f(\bx^{(r)})+\alpha q_{\pi}(\bx^{(r)})^{\T}\bv(\bx^{(r)})+\frac{\alpha^2}{2}\bv(\bx^{(r)})^{\T}\nabla^2 f(\bx^{(r)})\bv(\bx^{(r)})+\frac{\alpha^3}{6}L_2 \|\bv(\bx^{(r)})\|^3
\\
= &f(\bx^{(r)})+\alpha q_{\pi}(\bx^{(r)})^{\T}\bv(\bx^{(r)})+\frac{\alpha^2}{2}\bv(\bx^{(r)})^{\T}\nabla^2 f(\bx^{(r)})\bv(\bx^{(r)})+\frac{\alpha^3}{6}L_2,\label{eq.desv}
\end{align}
where we used $\|\bv(\bx^{(r)})\|=1$.
Since $\bv(\bx^{(r)})^{\T}\nabla^2 f(\bx^{(r)})\bv(\bx^{(r)})\le-\epsilon'_H(\delta)$ and $q_{\pi}(\bx^{(r)})^{\T}\bv(\bx^{(r)})\le0$,  we know that
\begin{equation}\label{eq.vdes}
\alpha q_{\pi}(\bx^{(r)})^{\T}\bv(\bx^{(r)})+\frac{\alpha^2}{2}\bv(\bx^{(r)})^{\T}\nabla^2 f(\bx^{(r)})\bv(\bx^{(r)})\le-\frac{\alpha^2\epsilon'_H(\delta)}{2}<0.
\end{equation}

Then, combining \eqref{eq.desv} and  \eqref{eq.vdes} we obtain
\begin{align}\label{eq.derho}
f(\bx^{(r)}+\alpha \bv(\bx^{(r)}))-f(\bx^{(r)})
\mathop{\le}&-\left(1-\frac{\alpha L_2}{3\epsilon'_H(\delta)}\right)\frac{\alpha^2\epsilon'_H(\delta)}{2}.
\end{align}
{It follows that when choosing $0<\alpha<\frac{3\epsilon'_H(\delta)}{L_2}$, the objective function is decreasing. By using the similar argument as in the previous lemma (by applying criteria \eqref{eq.back}), we can conclude that
\begin{align}
\alpha^{(r)}_{\max}>\frac{9\epsilon'_H(\delta)}{4 L_2}.
\end{align}}
{Finally, it is easy to see that by using the backtracking line search where each time the step-size is shrank by $1/2$, the algorithm will stop at $\alpha\ge 3\epsilon'_H(\delta)/(8L_2)$, therefore we will have at least the following amount of descent:} 
\begin{align}\notag
f(\bx^{(r+1)})\le & f(\bx^{(r)})- \frac{7}{8}\frac{(\frac{3}{8})^2\epsilon'^3_H(\delta)}{2L^2_2}\le  f(\bx^{(r)})-0.06\frac{\epsilon'^3_H(\delta)}{L^2_2},
\end{align}
which completes the proof.
\end{proof}

\begin{lemma}\label{le.4}
Consider Algorithm 1. The algorithm will stop if for $\min\{d,m\}$ consecutive iterations, its line search procedure only returns with stepsize $\alpha^{(r)}_{\max}$ chosen as in \eqref{eq.alphamax}. 
\end{lemma}
\begin{proof}
First, we show that $\dim(\mathcal{\bx}^{(r)})$ is not increasing if $\bx^{(r)}$ is updated by NCD successviely. Since at the $r$th iteration, the  equality $\bA'(\bx^{(r)})\bx^{(r)}=\bb'(\bx^{(r)})$ holds (due to the definition of active set), which implies that
\begin{equation}
\bA'(\bx^{(r+1)})=\bA'(\bx^{(r)}+\alpha^{(r)}\bd^{(r)})=\bb'(\bx^{(r)})+\alpha^{(r)}\bA'\bd^{(r)}\mathop{=}\limits^{(a)}\bb'(\bx^{(r)})
\end{equation}
where $(a)$ is true because $\bd^{(r)}\in\textsf{Null}(\bA'(\bx^{(r)}))$, so $\dim(\mathcal{F}(\bx^{(r+1)}))$ is no more than $\dim(\mathcal{F}(\bx^{(r)}))$. Second, we show that if $\alpha^{(r)}_{\max}$ is chosen, $\dim(\mathcal{F}(\bx^{(r)}))$ is decreased at least by 1. Since at the $r+1$th iteration the algorithm still choose $\alpha^{(r+1)}_{\max}$, meaning that iterate $\bx^{(r+1)}$ at least touches a new boundary, i.e., $\dim(\mathcal{F}(\bx^{(r)}))\ge\dim(\mathcal{F}(\bx^{(r+1)}))+1$. In other words, when step-size $\alpha^{(r)}_{\max}$ is chosen and updated by \eqref{eq.alphamax}, the dimension of the free space is reduced at least by 1. Therefore, if step-size $\alpha^{(r)}_{\max}$ is chosen consecutively and updated by \eqref{eq.alphamax}, $\dim(\mathcal{F}(\bx^{(r)}))$ is monotonically decreasing.
Since the dimension of the subspace is at most $d$ and the total number of constraints is at most $m$, the algorithm consecutively performs NCD at most $\min\{d,m\}$ times.
\end{proof}

\subsection{Simplified SNAP}\label{sec:ssnp}
Before proving \thref{th.1}, we give a simplified version of SNAP shown in \algref{alg:p5} and show the convergence of this algorithm, which will be helpful of understanding the key steps in the proof of SNAP. The reason is that some techniques, which are considered in SNAP to reduce the computational complexity, involve multiple branches that SNAP may use. {A combinatorial choice of these subroutines makes the convergence analysis complicated, so it will be more intuitive to see the proof for the simplified algorithm, which essentially has the same rate as SNAP.} Here, we give a concise proof for \algref{alg:p5} in the following.

\begin{algorithm}[ht]
\caption{A simplified Negative-curvature grAdient Projection algorithm}
\label{alg:p6}\footnotesize
\begin{algorithmic}[1]\small
\State {\bfseries Input:} $\bx^{(1)},\epsilon_G,\epsilon_H,L_1,L_2,\alpha_\pi=1/L_1,\delta,\bA,\bb, \textsf{flag}=\Diamond$
\For {$r=1,\ldots$}
\If {$\|g_\pi(\bx^{(r)})\|\le {\epsilon_G}$}
\State $[\textsf{flag}, \bv(\bx^{(r)}), -{\epsilon'_H(\delta)}]= \textsf{ \it Negative-Eigen-Pair}(\bx^{(r)},f,\delta)$
\If {$\textsf{flag}=\Diamond$}
\State Compute $q_{\pi}(\bx^{(r)})$ by \eqref{eq.compq}
\State Choose $\bv(\bx^{(r)})$ such that $q_{\pi}(\bx^{(r)})^{\T}\bv(\bx^{(r)})\le0$
\State $\bd^{(r)}=\bv(\bx^{(r)})$  \Comment{Choose negative curvature direction}
\State Update $\bx^{(r+1)}$ by  \algref{alg:p3} \Comment{Perform line search}
\Else \State Output $\bx^{(r)}$
\EndIf
\Else \State Update $\bx^{(r+1)}$ by \eqref{eq.pgd}\Comment{{Perform PGD}}
\EndIf
\EndFor
\end{algorithmic}
\end{algorithm}

\begin{proof}
We will show that after the number of iteration given in \eqref{alg:p5}, the algorithm will converge to an $(\epsilon_G, \epsilon_H)$-SOSP1 defined in \eqref{eq.cond1}.

Let us suppose that at a given point $\bx^{(r)}$, the condition \eqref{eq.cond1} does not hold.

First suppose that the first-order condition is not satisfied, that is
 $\|g_{\pi}(\bx^{(r)})\|\ge\epsilon_G$. Then the algorithm will perform the PGD step \eqref{eq.pgd}. By \leref{le.1}, the descent of the objective value is given by
 $\frac{\epsilon^2_G}{18L_1}.$

Second, when the size of the gradient is small, but the second-order condition in \eqref{eq.cond12} is not satisfied (i.e., when $\textsf{flag}=\Diamond$). Then in this case,  NCD will be performed, and  there are two choices for selecting the step-size:\\
\noindent{\bf Case 1) ($\textsf{flag}_{\alpha}=\emptyset$)}: The algorithm implements $\bx^{(r+1)}=\bx^{(r)}+\alpha^{(r)}\bd^{(r)}$ without using $\alpha^{(r)}_{\max}$ computed by \eqref{eq.alphamax}.\\
\noindent{\bf Case 2) ($\textsf{flag}_{\alpha}=\Diamond$)}: $\alpha^{(r)}_{\max}$ is computed by \eqref{eq.alphamax} to update $\bx^{(r+1)}$.

In the first case, we know that if $\alpha^{(r)}_{\max}$ is not chosen by \eqref{eq.alphamax}, then some sufficient descent will be achieved. From \leref{le.2}, we know that after one step update the objective value decreases as
\begin{equation}
f(\bx^{(r+1)})\le f(\bx^{(r)})-\Delta, \;\mbox{where}\; \Delta=\frac{0.06\epsilon'^3_H(\delta)}{L^2_2}.
\end{equation}

In the second case, the descent for each step may not be quantified. However, it is important to see that, by \leref{le.4}, the algorithm can repeat this case (i.e., choosing $\alpha^{(r)}_{\max}$  by \eqref{eq.alphamax}) for at most $\min\{d,m\}$  consecutive times.

By using the above fact, let us look at the second case in more detail and see how we can quantify the descent achieved by some $k\le\min\{d,m\}$ consecutive times that {\bf Case 2)} happens.
Since {\bf Case 2)} can happen at most $\min\{d,m\}$ consecutively  times, our strategy is to trace back the steps of the algorithm from the current iteration $\bx^{(r)}$ and see what happens.
To this end, let us suppose that at iteration $r$ {\bf Case 2)} happens.

First of all, if the sequence has never been updated by either {\bf Case 1} or PGD,  the algorithm must stop by at most $d$ iterations. If the algorithm stops, it is clear that  an $(\epsilon_G, \epsilon_H)$-SOSP1 solution is obtained. {This is because the inactive set becomes empty and \eqref{eq.cond12} is satisfied automatically.}

Second, consider iteration from $r-\min\{d,m\}$ until $r$. The sequence must be updated by either {\bf Case 1)} or the PGD step, otherwise the algorithm will stop and output an $(\epsilon_G, \epsilon_H)$-SOSP1 solution. Then we must have
\begin{align}\label{eq.descdm}
f(\bx^{(r)}) - f(\bx^{(r-\min\{d,m\})}) <  - \min \left\{\epsilon^2_G/(18L_1), 0.06\epsilon'^3_H(\delta)/(L^2_2)\right\},\forall r>\min\{d,m\}.
\end{align}

Summarizing the argument so far, we have that, after every consecutive $\min\{d,m\}$ iterations of the algorithm, either the algorithms stops, or \eqref{eq.descdm} holds true.

After applying the telescope sum on \eqref{eq.descdm}, we have
\begin{equation}\label{eq.de}
f^{\star}-f(\bx^{(1)})\le f(\bx^{(r)})-f(\bx^{(1)})\le -r \frac{\min \left\{\epsilon^2_G/(18L_1), 0.06\epsilon'^3_H(\delta)/(L^2_2)\right\}}{\min\{d,m\}}.
\end{equation}
where $f^{\star}$ denotes the minimum objective value achieved by the global optimal solution. By defining
\begin{equation}
    \Delta'\bydef\min\left\{0.06\frac{\epsilon'^3_H(\delta)}{L^2_2},\frac{\epsilon^2_G}{18L_1}\right\}\frac{1}{\min\{d,m\}},
\end{equation}
we obtain
\begin{equation}
r\le\frac{f(\bx^{(1)})-f^{\star}}{\Delta'}.
\end{equation}

Since the probability that eigen-pair fails to extract the negative curvature is $\delta$, applying the union bound, we only need to set $\delta'=\delta (f(\bx^{(1)}-f^{\star})/\Delta'$ so that we can have the claim that SNAP will output approximate SOSP1s with probability $1-\delta'$. Note that $\gamma \epsilon'_H(\delta)>\epsilon_H$. We can obtain the convergence rate of \algref{alg:p6} by
\begin{equation}
\widetilde{\mathcal{O}}\left(\frac{\min\{d,m\}(f(\bx^{(1)})-f^{\star})}{\min\left\{\frac{\epsilon^2_G}{L_1},\frac{\epsilon^3_H}{L^2_2}\right\}}\right),
\end{equation}
which completes the proof.
\end{proof}

\subsection{Proof of \thref{th.1}}

Compared with the simplified SNAP, SNAP has two main differences: 1) $\bd^{(r)}$ can be chosen by either $-q_{\pi}(\bx^{(r)})$ or $\bv(\bx^{(r)})$ in the NCD  step based on the minimum amount of the objective reduction; 2) when $\textsf{flag}^{(r)}_{\alpha}=\emptyset$ there is a minimum number of iterations (denoted by $r_{\textsf{th}}$) that SNAP calls subroutine $\textsf{\it Negative-Eigen-Pair}$ twice.
\begin{proof}
We show that after the number of iteration given in \eqref{alg:p1}, SNAP will converge to an $(\epsilon_G, \epsilon_H)$-SOSP1 defined in \eqref{eq.cond1} with high probability.

Let us suppose that at a given point $\bx^{(r)}$, the condition \eqref{eq.cond1} does not hold.

If the first-order condition is not satisfied, (i.e., $\|g_{\pi}(\bx^{(r)})\|\ge\epsilon_G$). By \leref{le.1}, the descent of the objective value by performing the PGD step \eqref{eq.pgd} is at least
 $\frac{\epsilon^2_G}{18L_1}$, i.e.,
 \begin{equation}\label{eq.desgdl}
     f(\bx^{(r+1)})\le f(\bx^{(r)})-\frac{\epsilon^2_G}{18L_1}.
 \end{equation}

Second, when the size of the gradient is small, but the second-order condition in \eqref{eq.cond12} is not satisfied (i.e., when $\textsf{flag}=\Diamond$). Then in this case,  the NCD will be performed, and  there are two choices for selecting the step-size:\\
\noindent{\bf Case 1) ($\textsf{flag}_{\alpha}=\emptyset$)}: The algorithm implements $\bx^{(r+1)}=\bx^{(r)}+\alpha^{(r)}\bd^{(r)}$ without using $\alpha^{(r)}_{\max}$ computed by \eqref{eq.alphamax}.\\
\noindent{\bf Case 2) ($\textsf{flag}_{\alpha}=\Diamond$)}: $\alpha^{(r)}_{\max}$ is computed by \eqref{eq.alphamax} to update $\bx^{(r+1)}$.

In the {\bf first} case, we know that if $\alpha^{(r)}_{\max}$ is not chosen by \eqref{eq.alphamax}, then some sufficient descent will be achieved. In particular, from \leref{le.2}--\leref{le.3}, no matter which direction (i.e., either $-q_{\pi}(\bx^{(r)})$ or $\bv(\bx^{(r)})$) is chosen, after one update the objective value decreases as 
\begin{equation}\label{eq.sufd}
f(\bx^{(r+1)})\le f(\bx^{(r)})-\Delta, \;\mbox{where}\; \Delta=\min\left\{\frac{0.18\epsilon'^3_H(\delta)}{L^2_2},\frac{0.06\epsilon'^3_H(\delta)}{L^2_2}\right\} = \frac{0.06\epsilon'^3_H(\delta)}{L^2_2}.
\end{equation}
After performing one step, $\textsf{flag}_\alpha$ becomes $\Diamond$. From the algorithm we know that  $r_{th}$ number of PGD will be performed. However, the amount of descent cannot be quantified  (becuase we are in NCD so $\|g_{\pi}(\mathbf{x}^{(r)})\|\le \epsilon_G$). Thus, we have
\begin{equation}\label{eq.sufd:3}
f(\bx^{(r+r_{th})})\le f(\bx^{(r)})-\Delta.
\end{equation}

In the {\bf second} case, the descent for each step may not be quantified. However, it is important to see that, by \leref{le.4}, the algorithm can repeat this case (i.e., choosing $\alpha^{(r)}_{\max}$  by \eqref{eq.alphamax}) for at most $\min\{d,m\}$  consecutive times.

By using the above fact, let us look at the second case in more detail and see how we can quantify the descent achieved by some $k\le\min\{d,m\}$ consecutive times that {\bf Case 2)} happens.
Since {\bf Case 2)} can happen at most $\min\{d,m\}$ consecutively  times, our strategy is to trace back the steps of the algorithm from the current iteration $\bx^{(r)}$ and see what happens.
To this end, let us suppose that at iteration $r$ {\bf Case 2)} happens.

First of all, if the sequence has never been updated by either {\bf Case 1} or PGD,  the algorithm must stop by at most $d$ iterations. If the algorithm stops, it is clear that  an $(\epsilon_G, \epsilon_H)$-SOSP1 solution is obtained. {This is because the inactive set becomes empty and \eqref{eq.cond12} is satisfied automatically.}

Second, consider iteration from $(r-\min\{d,m\})$ until $r$. The sequence must be updated by either {\bf Case 1)} or the PGD step, otherwise the algorithm will stop and output an $(\epsilon_G, \epsilon_H)$-SOSP1 solution.
Then we must have
\begin{align}\label{eq.descdm2}
f(\bx^{(r)}) - f(\bx^{(r-\min\{d,m\})}) \le  - \min \left\{\epsilon^2_G/(18L_1), 0.06\epsilon'^3_H(\delta)/L^2_2\right\},\forall r>\min\{d,m\}.
\end{align}

Summarizing the argument so far in the second case, we have that, after every consecutive $\min\{d,m\}$ iterations of the algorithm, either the algorithms stops, \eqref{eq.descdm2} holds true.

Note that {\bf Case 1} and {\bf Case 2} are mutually exclusive.
Take $T'\bydef\min\{d,m\}\cdot r_{\textsf{th}}$. From \eqref{eq.sufd}, we know that
\begin{equation}\label{eq.tele1}
    f(\bx^{(r+T')})-f({\bx^{(r)}})\le-\min\{d,m\}\frac{0.06\epsilon'^3_H(\delta)}{L^2_2}.
\end{equation}
From \eqref{eq.descdm2}, we have
\begin{equation}\label{eq.tele2}
    f(\bx^{(r)})-f(\bx^{(r-T')})\le-r_{\textsf{th}}\min\left\{\frac{\epsilon^2_G}{18L_1},\frac{0.06\epsilon'^3_H(\delta)}{L^2_2}\right\}.
\end{equation}
Adding \eqref{eq.tele1} and \eqref{eq.tele2} together, we have
\begin{equation}\label{eq.tele3}
    f(\bx^{(r+T')})-f(\bx^{(r-T')})\le -\min\left\{\min\{d,m\}\frac{0.06\epsilon'^3_H(\delta)}{L^2_2},r_{\textsf{th}}\min\left\{\frac{\epsilon^2_G}{18L_1},\frac{0.06\epsilon'^3_H(\delta)}{L^2_2}\right\}\right\}.
\end{equation}
Let $n$ be the number of  $2T'$ blocks contained in $[1,r]$. After applying the telescope sum on \eqref{eq.desgdl}, \eqref{eq.tele3}, we have
\begin{align}\notag
&f^{\star}-f(\bx^{(1)})\le f(\bx^{(r)})-f(\bx^{(1)})\le f(\bx^{(2nT'+1)})-f(\bx^{(1)})
\\
\le&-n \min\left\{\min\{d,m\}\frac{0.06\epsilon'^3_H(\delta)}{L^2_2},r_{\textsf{th}}\min\left\{\frac{\epsilon^2_G}{18L_1},\frac{0.06\epsilon'^3_H(\delta)}{L^2_2}\right\},T'\frac{\epsilon^2_G}{18L_1}\right\}\label{eq.de2}
\end{align}
where $f^{\star}$ denotes the minimum objective value achieved by the global optimal solution, and $n\ge(r-1)/(2T')$. By defining
\begin{equation}
    \Delta'\bydef \min\left\{\min\{d,m\}\frac{0.06\epsilon'^3_H(\delta)}{L^2_2},r_{\textsf{th}}\min\left\{\frac{\epsilon^2_G}{18L_1},\frac{0.06\epsilon'^3_H(\delta)}{L^2_2}\right\},T'\frac{\epsilon^2_G}{18L_1}\right\},
\end{equation}
we obtain
\begin{equation}\label{eq.rupbd}
n\le\frac{f(\bx^{(1)})-f^{\star}}{\Delta'},
\end{equation}
and
\begin{align}\label{eq.ruppbd}
\notag
   r\le 2nT'+1\le&\frac{f(\bx^{(1)})-f^{\star}}{\Delta'}
    \\
    \le&(f(\bx^{(1)})-f^{\star})\max\left\{\frac{0.06\max\{r_{\textsf{th}},\min\{d,m\}\} L^2_2}{\epsilon'^3_H(\delta)},\frac{18\min\{d,m\}L_1}{\epsilon^2_G}\right\}.
\end{align}
Since the probability that eigen-pair fails to extract the negative curvature is $\delta$, applying the union bound, we only need to set $\delta'=\delta (f(\bx^{(1)}-f^{\star})/\Delta'$ so that we can have the claim that SNAP will output approximate SOSP1s with probability $1-\delta'$.

Note that $\gamma \epsilon'_H(\delta)>\epsilon_H$. We can obtain the convergence rate of \algref{alg:p1} by
\begin{equation}
\widetilde{\mathcal{O}}\left(\frac{(f(\bx^{(1)})-f^{\star})}{\min\left\{\frac{\epsilon^2_G}{\min\{d,m\}L_1},\frac{\epsilon^3_H}{\max\{r_{\textsf{th}},\min\{d,m\}\}L^2_2}\right\}}\right),
\end{equation}
This completes the proof.

If $r_{\textsf{th}}$ is a constant, the convergence rate is $\widetilde{\mathcal{O}}\left((\min\{d,m\}(f(\bx^{(1)})-f^{\star}))/(\min\left\{\frac{\epsilon^2_G}{L_1},\frac{\epsilon^3_H}{L^2_2}\right\})\right).$

From the above proof, we can see that essentially we only need to quantify the average descent of objective value per-iteration since we have descent either by one step or over a certain number of iterations, e.g., $\min\{d,m\}$ or $r_{\textsf{th}}$. Therefore, we give the following summary that quantify the average descent of the objetive value per-iteration if SNAP does not meet the stopping criteria.
\begin{enumerate}
	\item $\|g_{\pi}(\bx^{(r)})\|\ge\epsilon_G$: descent per-iteration is $\frac{\epsilon^2_G}{18L_1}$ by \leref{le.1}.
	
	\item $\|g_{\pi}(\bx^{(r)})\|\le\epsilon_G$
	\begin{itemize}
		\item   $\textsf{flag}_{\alpha}=\emptyset$ and $r-r_{\textsf{last}}<r_{\textsf{th}}$: descent per-iteration: $0.06\epsilon'^3_H(\delta)/(r_{\textsf{th}}L^2_2)$ by \eqref{eq.sufd:3}.
		\item  $\textsf{flag}_{\alpha}=\Diamond$ or   $r-r_{\textsf{last}}\ge r_{\textsf{th}}$: 
		\begin{itemize}
			\item $\textsf{flag}=\emptyset$: \mhedit{This case means that $\bx^{(r)}$ is an $(\epsilon_G,\epsilon_H)$-SOSP1.} We output $\bx^{(r)}$.
			\item $\textsf{flag}=\Diamond$ (\mhedit{there is a negative curvature}):
			\begin{itemize}
				\item $\textsf{flag}_{\alpha}=\emptyset$ (\mhedit{boundary not touched}): descent per-iteration is at least $0.06\frac{\epsilon'^3_H(\delta)}{L^2_2}$ by \eqref{eq.sufd}. 
				\item $\textsf{flag}_{\alpha}=\Diamond$ (\mhedit{boundary touched}): descent per-iteration  $\min\{\frac{\epsilon^2_G}{18L_1},0.06\frac{\epsilon'^3_H(\delta)}{L^2_2}\}/\min\{d,m\}$ by \eqref{eq.descdm2}. 
			\end{itemize}
		\end{itemize}
	\end{itemize}
\end{enumerate}
\end{proof}
Finally, we comment that it is of interest to find an $(\epsilon,\sqrt{\epsilon})$ (i.e, $\epsilon_G=\epsilon$ and $\epsilon_H=\sqrt{L_2\epsilon}$) in practice. If we choose $r_{\textsf{th}}\sim L_1\sqrt{L_2\epsilon}$,  then $\Delta'\sim\mathcal{O}(\epsilon^2/(\min\{d,m\}L_1))$, resulting in the convergence rate of SNAP as $\widetilde{\mathcal{O}}(\min\{d,m\}L_1/\epsilon^2)$.

\section{Proofs of SNAP$^+$}\label{sec:snap+}

Before proceeds, we first have the following definitions and corresponding properties of the SP-GD iterates which will be helpful in the proof of the convergence rate. Throughout the proof we will assume that \asref{as1} is satisfied.

\paragraph{Strict Saddle Point:}
\begin{condition}\label{cond:saddle}
A strict saddle point $\bx$ satisfies the following condition:
\begin{equation}
\lambda_{\min}(\bH_{\bP}(\bx))\le-\epsilon_H.
\end{equation}
Let $\vec{\be}$ denotes the eigenvector of $\bH_{\bP}({\bx})$ corresponds to the smallest  eigenvalue of $\bH_{\bP}(\bx)$.
\end{condition}

\paragraph{Approximate Objective Function:}
We define an approximate objective function as
\begin{equation}
\widehat{f}_{\bx}(\bu)\bydef f(\bx+\bu)-f(\bx)-\nabla_{\bx} f(\bx)^{\T}\bu.
\end{equation}
Then, we can have
\begin{equation}
\nabla_{\bu}\widehat{f}_{\bx}(\bu)=\nabla_{\bx} f(\bx+\bu)-\nabla_{\bx}  f(\bx).
\end{equation}
It is easy to see that $\nabla_{\bu}\widehat{f}_{\bx}(\bu)$ is also $L_1$-Lipschitz continuous.

In the rest of the paper, we just use $\widehat{f}$ as abbreviated $\widehat{f}_{\bx}$. Let $\widehat{q}_{\pi}(\bu)\bydef\bP \nabla_{\bu}\widehat{f}(\bu)$ and $\bP$ is the projection matrix defined in \eqref{eq:P}.

We can have
\begin{equation}
\widehat{q}_{\pi}(\bu)= q_{\pi}(\bx+\bu)-q_{\pi}(\bx).
\end{equation}

From the update rule of SP-GD \eqref{eq.spgd}, we know that
\begin{equation}
\bu^{(r+1)}=\bu^{(r)}-\beta \widehat{q}_{\pi}(\bu^{(r)}).
\end{equation}

In the following, \leref{le.l2lip} is a preliminary lemma which will be used in \leref{le.layer31} and \thref{th:spgd1}. Further combining \leref{le.layer31} and \leref{le.layer32} leads to \thref{th:spgd1}.

\begin{lemma}\label{le.l2lip}
If function $f(\cdot)$ is $L_2$-Hessian Lipschitz, we have
\begin{equation}
\left\|\int^1_0\bP^{\T}\nabla^2 f(\theta\bx)d\theta-\bP^{\T}\nabla^2 f(\bx')\bP\right\|\le L_2\|\bx\|+\|\bx'\|, \forall \bx,\bx'\in\mathcal{X}.
\end{equation}
where $\bP$ denotes the projection matrix and $\theta\in[0,1]$.
\end{lemma}
\begin{proof} We have the following relations:
\begin{align}
\notag
&\left\|\int^1_0\bP^{\T}\left(\nabla^2 f(\theta\bx)-\nabla^2 f(\bx')\bP\right) d\theta \right\|
\\\notag
=&\left\|\int^1_0\bP^{\T}\left(\nabla^2 f(\theta\bx)-\nabla^2 f(\bx')(\bI-\bP_{\perp})\right) d\theta \right\|
\\\notag
\mathop{\le}\limits^{(a)}&\int^1_0\|\nabla^2 f(\theta\bx)-\nabla^2 f(\bx')\|d\theta+\|\bP^{\T}\nabla^2 f(\bx')\bP_{\perp}\|
\\\notag
\mathop{\le}\limits^{(b)}& L_2\int^1_0\|\theta\bx-\bx'\|d\theta\le L_2\int^1_0\theta\|\bx\|d\theta+L_2\|\bx'\|\le L_2(\|\bx\|+\|\bx'\|)
\end{align}
where in $(a)$ $\bP_{\perp}=(\bA'(\bz))^{\T}\left(\bA'(\bz)(\bA'(\bz))^{\T}\right)^{-1}\bA'(\bz)$ and we use $\|\bP\|=1$, the symmetry of matrix $\bP$ and the fact that $\bP\nabla^2 f(\bx')\bP_{\perp}=0$ since $\bP$ projects all the column of $\nabla^2 f(\bx')$ into the range space of $\bP$ so that it is in the null space of $\bP_{\perp}$; in $(b)$ we use the $L_2$-Hessian Lipschitz continuity.
\end{proof}

We also need to introduce some constants defined as follows,
\begin{subequations}
\begin{align}
\sF\bydef& \frac{\epsilon^{3}_H}{L^2_2\widehat{c}^5}\log^{-3}\left(\frac{d\kappa}{\delta}\right),\label{eq.defff}
\\
\sS\bydef& \frac{\epsilon_H}{L_2\widehat{c}^2}\log^{-1}\left(\frac{d\kappa}{\delta}\right),\label{eq.defs}
\\
\sT\bydef&\frac{\log\left(\frac{d\kappa}{\delta}\right)}{\beta\epsilon_H}.\label{eq.deft}
\end{align}
\end{subequations}
These quantities
refer to different units of the algorithm. Specifically, $\sF$
accounts for the objective value,
$\sS$ for the norm of the difference between iterates, and $\sT$ for the number
of iterations. Also, we define a condition number in terms of $\epsilon_H$ as
$$\kappa\bydef\frac{L_1}{\epsilon_H}\ge 1.$$ In the process of the proofs, we also use conditions $$\log(\frac{d\kappa}{\delta})\ge1$$ when $\delta\in(0,\frac{d\kappa}{e}]$ repeatedly to simply the expressions of the inequalities.

\begin{lemma}\label{le.layer31}
Under assumption, consider ${\bx}$ that satisfies \conref{cond:saddle} and a  sequence $\bu^{(r)}$ generated by SP-GD.
Let us define a constant $\beta\le1/L_1$, and the following quantities:
\begin{equation}
\sR\bydef\frac{\sS}{\widehat{c}^2\kappa\log\left(\frac{d\kappa}{e}\right)},\quad\textrm{and}\quad T\bydef\min\left\{\min_{r\ge 1}\{r|\widehat{f}(\bu^{(r)})-\widehat{f}(\bu^{(1)})\le-2\sF\},\widehat{c}\cdot\sT\right\}.\label{eq.defoft}
\end{equation}
Then for any constant $\widehat{c}\ge1$, $\delta\in(0,\frac{d\kappa}{e}]$, when initial point $\bu^{(1)}$ satisfies
\begin{equation}
\|\bu^{(1)}-\bx\|\le 2\sR,
\end{equation}
 the iterates generated by SP-GD satisfy $\|\bu^{(r)}-{\bx}\|\le3\sS,\forall r<T$.
\end{lemma}
\begin{proof}
Without loss of generality, let $\bu^{(1)}$ be the origin, i.e., $\bu^{(1)}=0$. According to the update rule of SP-GD, we have
\begin{equation}\label{eq.iteofu}
\bu^{(r+1)}=\bu^{(r)}-\beta \widehat{q}_{\pi}(\bu^{(r)}).
\end{equation}

Similar as the derivation in \eqref{eq:descent}, according to the $L_1$-gradient Lipschitz continuity, we have
\begin{align}
\notag
\hf(\bu^{(r+1)})=&\hf\left(\bu^{(r)}-\beta \hq_{\pi}(\bx^{(r)})\right)
\\\notag
\le& \hf(\bu^{(r)})-\beta \nabla \hf(\bu^{(r)})^{\T}\hq_{\pi}(\bu^{(r)})+\frac{\beta^2}{2}L_1\|\hq_{\pi}(\bu^{(r)})\|^2
\\\notag
\le& \hf(\bu^{(r)})-\beta \hq_{\pi}(\bu^{(r)})^{\T}\hq_{\pi}(\bu^{(r)})+\frac{\beta^2}{2}L_1\|\hq_{\pi}(\bu^{(r)})\|^2
\\\notag
=& \hf(\bu^{(r)})-\beta\|\hq_{\pi}(\bu^{(r)})\|^2+\frac{\beta^2}{2}L_1\|\hq_{\pi}(\bu^{(r)})\|^2
\\
=&\hf(\bu^{(r)}) -\left(\beta-\frac{\beta^2}{2}L_1\right)\|\hq_{\pi}(\bu^{(r)})\|^2.\label{eq.desbeta}
\end{align}

From \eqref{eq.desbeta}, we also know that
\begin{align}
\notag
\hf(\bu^{(r+1)})\mathop{\le}\limits^{(a)}& \hf(\bu^{(r)})-\frac{\beta}{2}\|\hq_{\pi}(\bu^{(r)})\|^2
\\
\mathop{=}\limits^{\eqref{eq.iteofu}}&\hf(\bu^{(r)})-\frac{1}{2\beta}\|\bu^{(r+1)}-\bu^{(r)}\|^2\label{eq.desofu}
\end{align}
where in $(a)$ we choose $\beta\le 1/L_1$.

By applying telescoping sum of \eqref{eq.desofu}, we have
\begin{equation}
\hf(\bu^{(r+1)})\le \hf(\bu^{(1)})-\frac{1}{2\beta}\sum^r_{\tau=1}\|\bu^{(\tau+1)}-\bu^{(\tau)}\|^2,\quad\forall r< T.\label{eq.desbd}
\end{equation}

According to the definition of $T$, we know that
\begin{equation}
\hf(\bu^{(1)})-\hf(\bu^{(r)})<2\sF,\quad\forall r< T.\label{eq.desbydef}
\end{equation}

Combining \eqref{eq.desbd} and \eqref{eq.desbydef} , we know that
\begin{equation}
\sum^{r-1}_{\tau=1}\|\bu^{(\tau+1)}-\bu^{(\tau)}\|^2 < 4\beta\sF.\label{eq.decubd}
\end{equation}

Next, we will get the upper bound of $\|\bu^{(r)}-\bu^{(1)}\|,\forall r<T$ as the following.
First, by the triangle inequality, we know
\begin{equation}
\|\bu^{(r)}-\bu^{(1)}\|\le\sum^{r-1}_{\tau=1}\|\bu^{(\tau+1)}-\bu^{(\tau)}\|,
\end{equation}
so we have
\begin{align}
\|\bu^{(r)}-\bu^{(1)}\|^2\le & (r-1)\sum^{r-1}_{\tau=1}\|\bu^{(\tau+1)}-\bu^{(\tau)}\|^2
\\
\le & (T-1)\sum^{r-1}_{\tau=1}\|\bu^{(\tau+1)}-\bu^{(\tau)}\|^2
\\
\mathop{\le}\limits^{\eqref{eq.decubd}}& T4\beta\sF\mathop{\le}\limits^{\eqref{eq.defoft}} 4\widehat{c}\beta\sF \sT \mathop{\le}\limits^{(a)}4\sS^2,\label{eq.normofu}
\end{align}
where in $(a)$ we use the relation ${\widehat{c}\beta\sF \sT}=\sS^2$ by applying \eqref{eq.defff}\eqref{eq.defs}\eqref{eq.deft}.

Due to the following fact
\begin{equation}
\|\bu^{(r)}-\bx\|=\|\bu^{(r)}-\bu^{(1)}+\bu^{(1)}-\bx\|\le\underbrace{\|\bu^{(r)}-\bu^{(1)}\|}_{\le 2\sS}+\underbrace{\|\bu^{(1)}-\bx\|}_{\le\frac{\sS}{\widehat{c}^2\log(\frac{d\kappa}{\delta})}}\le 3\sS
\end{equation}
where the last inequality is true when $\widehat{c}\ge 1$, and $d\kappa/\delta>e$. Therefore, we know that $\|\bu^{(r)}-\bx\|\le 3 \sS,\forall r<T$ where $\beta\le 1/L_1$, which completes the proof.
\end{proof}
\begin{lemma}\label{le.layer32}
Consider $\bx$ that satisfies \conref{cond:saddle}.  Suppose that  there exist
two iterates $\{\bu^{(r)}\}$ and $\{\bw^{(r)}\}$, generated
by SP-GD with two different initial points $\{\bu^{(1)}, \bw^{(1)}\}$, where these initial points satisfy
\begin{equation}
\|\bu^{(1)}-\bx\|\le \sR,\;\bw^{(1)}=\bu^{(1)}+\upsilon \sR \vec{\be},\;\upsilon\in[\delta/(2\sqrt{d}),1],\label{eq.inicond}
\end{equation}
where $\sR $ is defined in \eqref{eq.defoft}. Let us also define
\begin{equation}
T\bydef\min\left\{\min_{r\ge 1}\{r|{\hf}(\bw^{(r)})-\hf(\bw^{(1)})\le-2\sF\},\widehat{c}\cdot\sT\right\}\label{eq.defoftl}.
\end{equation}
Suppose  $\widehat{c}\ge 51$, $\delta\in(0,\frac{d\kappa}{e}]$, $\beta\le 1/L_1$, $\|\bu^{(r)}-\bx\|\le 3\sS, \forall r<T$, then we will have $T<\widehat{c}\cdot\sT$, that is, we must have
\begin{align}
{\hf}(\bw^{(r)})-\hf(\bw^{(1)})\le-2\sF.
\end{align}
\end{lemma}
\begin{proof}
Let $\mathcal{H}\bydef\bH_{\bP}({\bx})$, where $\bx$ satisfies \conref{cond:saddle}. Let $\bu^{(1)}=0$ and define $\bv^{(r)}\bydef\bw^{(r)}-\bu^{(r)}$. According to the assumption of \leref{le.layer32}, we know
\begin{equation}
\bv^{(1)} = \bw^{(1)}=\upsilon \sR \vec{\be}=\upsilon\frac{\sS}{\widehat{c}^2\kappa\log\left(\frac{d\kappa}{e}\right)}\vec{\be}
\end{equation}
where $\upsilon\in[\delta/(2\sqrt{d}),1]$. Clearly we have $\|\bv^{(1)}\|\le \sR$.

Then sequence $\bw^{(r+1)}$ can be expressed by {\small
\begin{align}
\notag
\bu^{(r+1)}+\bv^{(r+1)}& =\bw^{(r+1)}
\\\notag
&=\bw^{(r)}-\beta \left(q_{\pi}(\bw^{(r)}+\bx)-q_{\pi}(\bx)\right)
\\\notag
&=\bu^{(r)}+\bv^{(r)}-\beta \left(q_{\pi}(\bu^{(r)}+\bv^{(r)}+\bx)-q_{\pi}(\bx)\right)
\\\notag
&=\bu^{(r)}+\bv^{(r)}-\beta \left(q_{\pi}(\bu^{(r)}+\bv^{(r)}+\bx)-q_{\pi}(\bu^{(r)}+\bx)+q_{\pi}(\bu^{(r)}+\bx)-q_{\pi}(\bx)\right)
\\\notag
&\stackrel{(a)}=\bu^{(r)}-\beta\left(q_{\pi}(\bu^{(r)}+\bx)-q_{\pi}(\bx)\right)+\bv^{(r)}-\beta\left[\int^1_0\bP^{\T}\nabla^2 f(\bu^{(r)}+\bx+\theta\bv^{(r)})d\theta\right]\bv^{(r)}
\\\notag
&=\bu^{(r)}-\beta\left(q_{\pi}(\bu^{(r)}+\bx)-q_{\pi}(\bx)\right)+\bv^{(r)}-\beta(\mathcal{H}+\Delta^{(r)})\bv^{(r)}
\\
&=\bu^{(r)}-\beta\left(q_{\pi}(\bu^{(r)}+\bx)-q_{\pi}(\bx)\right)+(\bI-\beta\mathcal{H}-\beta\Delta^{(r)})\bv^{(r)}\nonumber\\
&=\bu^{(r+1)}+(\bI-\beta\mathcal{H}-\beta\Delta^{(r)})\bv^{(r)}\label{eq.dyv}
\end{align}}
where $(a)$ uses the Mean Value Theorem; $\Delta^{(r)}=\int^1_0\bP^{\T}\nabla^2 f(\bu^{(r)}+\bx+\theta\bv^{(r)})d\theta-\mathcal{H}$. Therefore, we have
\begin{equation}\label{eq:dynamic}
\bv^{(r+1)}=(\bI-\beta\mathcal{H}-\beta\Delta^{(r)})\bv^{(r)}.
\end{equation}

By applying \leref{le.l2lip} and $L_2$-Lipschitz continuity of the objective function, we have
\begin{align}\label{eq:delta}
{\|\Delta^{(r)}\|\le L_2(\|\bu^{(r)}\|+\|\bv^{(r)}\|+2\|{\bx}\|).}
\end{align}
Note that $\|\bw^{(1)}-\bx\|\le\|\bu^{(1)}-\bx\|+\|\bv^{(1)}\|\le2\sR$. This means that as a sequence generated by SP-GD,
$\{\bw^{(r)}\}$ satisfies the assumption given in \leref{le.layer31}. Also note that we have assumed that $\widehat{c}\ge 51$, then by the same lemma, it follows that
$$\|\bw^{(r)}-\bx\|\le 3\sS,\quad \forall r<T.$$
Similarly, we can apply \leref{le.layer31} again to obtain $\|\bu^{(r)}-\bx\|\le3\sS,\forall r<T$ since we have assumed $\|\bu^{(1)}-\bx\|\le\sR$. Combining these two results, we have
\begin{equation}
\|\bv^{(r)}\|=\|\bw^{(r)}-\bu^{(r)}\|\le\|\bw^{(r)}-\bx\|+\|\bu^{(r)}-\bx\|\le6\sS.\label{eq.sizeofv}
\end{equation}

Next let us prove that the following hold:
$$\|\bx\|\le\sR\le\sS$$
where the first inequality is because the assumption that $\bu^{(1)}=0$ and $\|\bu^{(1)} -\bx\|\le \sR$;
the second inequality is due to the following choices of the constants $\widehat{c}\ge 1$, $\kappa\ge1$ and $\log(d\kappa/\delta)\ge1$. Further, from \eqref{eq.normofu} and the assumption that $\bu^{(1)}=0$, we have $\|\bu^{(r)}\|\le2\sS$. Combining the above relations with \eqref{eq:delta}, we conclude
\begin{align}\label{eq:delta:norm}
\|\Delta^{(r)}\|\le 10L_2\sS\quad\textrm{and} \quad \beta\|\Delta^{(r)}\|\le 10\beta L_2\sS.
\end{align}
By \conref{cond:saddle} we know that $\mathbf{I}-\beta\mathcal{H}$ has maximum eigenvalue at least $1+\epsilon_H\beta$.
Let $\phi^{(r)}$ denote the norm of $\bv^{(r)}$ projected on the space spanned by $\vec{\be}$ , and let
$\psi^{(r)}$ denote the norm of $\bv^{(r)}$ projected onto the remaining space. From \eqref{eq:dynamic}, we have
\begin{subequations}
\begin{align}
\phi^{(r+1)}\ge&(1+\epsilon_H\beta)\phi^{(r)}-\mu\sqrt{(\phi^{(r)})^2+(\psi^{(r)})^2},\label{eq.rephi}
\\
\psi^{(r+1)}\le&(1+\epsilon_H\beta)\psi^{(r)}+\mu\sqrt{(\phi^{(r)})^2+(\psi^{(r)})^2},\label{eq.repsi}
\end{align}
\end{subequations}
where we have defined
\begin{equation}\label{eq.defofmu}
\mu=10\beta L_2\sS
\end{equation}
and the inequalities are true due to the use of triangular inequality and the bound in \eqref{eq:delta:norm}.

Then, we will use mathematical induction to prove
\begin{equation}\label{eq.induc}
\psi^{(r)}\le 4\mu r\phi^{(r)}, \; \forall~r<T.
\end{equation}
Intuitively, the above result says that,  the projection of $\bv^{(r)}$ in the negative curvature direction should be relatively large,  and this fact will finally lead to a fast descent in the objective.

Let us prove \eqref{eq.induc}. Clearly this equation is true when $r=1$ since by definition, we have
\begin{align}
\bv^{(1)} = \bw^{(1)} - \bu^{(1)} = \upsilon\sR\vec{\be},
\end{align}
which implies that $\|\psi^{(1)}\|=0$.

Next, let us assume that \eqref{eq.induc} is true at the $r$th iteration, we need to prove
\begin{equation}\label{eq.reduc}
\psi^{(r+1)}\le 4\mu(r+1)\phi^{(r+1)}, \; \forall~r<T-1.
\end{equation}
To show this result, we utilize \eqref{eq.rephi} and \eqref{eq.repsi} to lower and upper bound  $4\mu(r+1)\phi^{(r+1)}$ and $\psi^{(r+1)}$, respectively.
Substituting \eqref{eq.repsi} into LHS of \eqref{eq.reduc}, we have the upper bound of $\psi^{(r+1)}$, i.e.,
\begin{equation}
\psi^{(r+1)}\le\left(1+\epsilon_H\beta\right)4\mu r\phi^{(r)}+\mu\sqrt{(\phi^{(r)})^2+(\psi^{(r)})^2}.\label{eq.ubdofphi}
\end{equation}
Applying \eqref{eq.rephi} into RHS of \eqref{eq.reduc}, we have the lower bound of $4\mu(r+1)\phi^{(r+1)}$ as the following:
\begin{equation}
4\mu(r+1)\phi^{(r+1)}\ge4\mu(r+1)\left((1+\epsilon_H\beta)\phi^{(r)}-\mu\sqrt{(\phi^{(r)})^2+(\psi^{(r)})^2}\right).\label{eq.lbdofmu}
\end{equation}
Next, we will show that the following holds,
\begin{equation}
\left(1+4\mu(r+1)\right)\left(\sqrt{(\phi^{(r)})^2+(\psi^{(r)})^2}\right)\le4\phi^{(r)}.\label{eq.suffineq}
\end{equation}
If this is true, then after manipulation, we can show that the RHS of \eqref{eq.lbdofmu} is greater than the RHS of \eqref{eq.ubdofphi}, which will eventually imply \eqref{eq.reduc}.

In the following, we will show that the above relation \eqref{eq.suffineq} is true, i.e.,  RHS of \eqref{eq.rephi} is greater
than RHS of \eqref{eq.repsi}.
\paragraph{First step:}
We know that
\begin{equation}\label{eq.bdofmu}
4\mu(r+1)\le4\mu T\mathop{\le}\limits^{\eqref{eq.defofmu}}40\beta L_2\widehat{c}\sS\cdot \sT\mathop{\le}\limits^{(a)}\frac{40}{\widehat{c}}\mathop{\le}\limits^{(b)}1
\end{equation}
where the first inequality is true because $r<T-1$; in $(a)$ we use the relation $\beta L_2\sS\widehat{c}\cdot \sT=\frac{1}{\widehat{c}}$ by applying \eqref{eq.defs}\eqref{eq.deft}; $(b)$ is true when $\widehat{c}\ge 40$.

\paragraph{Second step:}
By using the induction assumption and the previous step, we have
\begin{equation}
4\phi^{(r)}\ge2\sqrt{2(\phi^{(r)})^2}\mathop{\ge}\limits^{\eqref{eq.induc},\eqref{eq.bdofmu}}(1+4\mu(r+1))\sqrt{(\phi^{(r)})^2+(\psi^{(r)})^2},
\end{equation}
which gives \eqref{eq.suffineq}. Therefore, we can conclude that $\psi^{(r+1)}\le 4\mu(r+1)\phi^{(r+1)}$ is true, which completes the induction.

\paragraph{Recursion of $\phi^{(r)}$:} Next we will show that the projection of $\bv^{r}$ on the negative curvature direction $\vec{\be}$ will be exponentially increasing.
$\\$
Using \eqref{eq.induc}, we have
\begin{equation}\label{eq.psilesph}
    \psi^{(r)}\mathop{\le}\limits^{\eqref{eq.induc}} 4\mu r\phi^{(r)}\mathop{\le}\limits^{\eqref{eq.bdofmu}} \phi^{(r)}.
\end{equation}
Then, we can  get the recursion of $\phi^{(r+1)}$ by the following steps.
\begin{align}
\notag
\phi^{(r+1)}\mathop{\ge}\limits^{\eqref{eq.rephi}}&(1+\epsilon_H\beta)\phi^{(r)}-\mu\sqrt{(\phi^{(r)})^2+(\psi^{(r)})^2}
\\\notag
\mathop{\ge}\limits^{(a)}&(1+\epsilon_H\beta)\phi^{(r)}-\mu\sqrt{2}\phi^{(r)}
\\\notag
\mathop{=}\limits^{\eqref{eq.defofmu}}&(1+\epsilon_H\beta)\phi^{(r)}-10\beta L_2\sS\sqrt{2}\phi^{(r)}
\\\notag
\mathop{=}\limits^{\eqref{eq.defs}}&(1+\epsilon_H\beta)\phi^{(r)}-\frac{10\sqrt{2}\epsilon_H\beta}{\widehat{c}^2\log(\frac{d\kappa}{\delta})}\phi^{(r)}
\\
\mathop{\ge}\limits^{(b)}&(1+\frac{\epsilon_H\beta}{2})\phi^{(r)}\label{eq.relation}
\end{align}
where $(a)$ is true because \eqref{eq.psilesph}; $(b)$ is true when $\widehat{c}\ge 2\sqrt{5\sqrt{2}}$.

\paragraph{Quantifying Escaping Time:} Next we estimate how many iterations does it require for $\bw^{(r)}$ to reduce the objective value sufficiently.
$\\$
From  \eqref{eq.sizeofv} and the definition of $\phi^{(r)}$, we have
\begin{align}
\notag
6\sS\ge&\|\bv^{(r)}\|\ge\phi^{(r)}
\\\notag
\mathop{\ge}\limits^{\eqref{eq.relation}}&(1+\frac{\beta\epsilon_H}{2})^{r}\phi^{(1)}
\\
\notag
\stackrel{(a)}=&(1+\frac{\beta\epsilon_H}{2})^{r}\|\bw^{(1)}-\bu^{(1)}\|
\\
\mathop{\ge}\limits^{(b)}&(1+\frac{\beta\epsilon_H}{2})^{r}\frac{\delta}{2\sqrt{d}}\frac{\sS}{\widehat{c}^2\kappa}\log^{-1}(\frac{d\kappa}{\delta})\quad\forall r<T
\label{eq.upoft}
\end{align}
where in $(a)$ we used \eqref{eq.inicond}; in $(b)$ we use condition $\upsilon\in[\delta/(2\sqrt{d}),1]$.

Since \eqref{eq.upoft} is true $\forall r<T$, then it must hold for $r=T-1$. Taking log on both sides of \eqref{eq.upoft}, letting $r=T-1$,  we can have
\begin{align}
\notag
T\le&\frac{\log(12\widehat{c}^2(\frac{\kappa\sqrt{d}}{\delta})\log(\frac{d\kappa}{\delta}))}{\log(1+\frac{\beta\epsilon_H}{2})}+1
\mathop{<}\limits^{(a)}\frac{4\log(12\widehat{c}^2(\frac{\sqrt{d}\kappa}{\delta})\log(\frac{d\kappa}{\delta}))}{\beta\epsilon_H}+1
\\
\mathop{<}\limits^{(b)}&\frac{4\log(12\widehat{c}^2(\frac{d\kappa }{\delta})^2)}{\beta\epsilon_H}+1
\mathop{\le}\limits^{(c),\eqref{eq.deft}}4(2+\log(12\widehat{c}^2))\sT+1\mathop{\le}\limits^{(d),\eqref{eq.deft}}4(2\frac{1}{4}+\log(12\widehat{c}^2))\sT\label{eq.bdt}
\end{align}
where $(a)$ comes from inequality $\log(1+x)>x/2$ when $x<1$,  in $(b)$ we used relation $\log(x)<x, x>0$, and $(c)$ is true because $\delta\in(0,\frac{d\kappa}{e}]$ and $\log(d\kappa/\delta)>1$ so that $\log(12\widehat{c}^2)+2\log(\frac{d\kappa}{\delta})\le(\log(12\widehat{c}^2)+2)\log(\frac{d\kappa}{\delta})$; $(d)$ is true due to the fact that $\beta L_1\le 1$, $\kappa\ge1$, and $\log(d\kappa/\delta)\ge1$ so we have $\sT\ge1$ .

From \eqref{eq.bdt}, we know that when
\begin{equation}
4\left(2\frac{1}{4}+\log(12\widehat{c}^2)\right)<\widehat{c},\label{eq.chat}
\end{equation}
we will have $T<\widehat{c}\sT$.

It can be observed that LHS of \eqref{eq.chat} is a logarithmic with respect to $\widehat{c}$ and RHS of \eqref{eq.chat} is a linear function in terms of $\widehat{c}$, implying that when $\widehat{c}$ is large enough inequality \eqref{eq.chat} holds. It is can be numerically checked that when $\widehat{c}\ge 51$ inequality \eqref{eq.chat} holds. The proof is complete.

\end{proof}

\subsection{Proof of \thref{th:spgd1}}
The proof of \thref{th:spgd1} is similar as the one of proving convergence of PGD shown in \cite[Lemma 14,15]{jin2017jordan} and NEON in \cite[Theorem 2]{xu2017first}. Considering the completeness of the whole proof in this paper, here we give the following proof of this lemma in details.

\begin{proof}
Let $\bz^{(1)}$ be a vector that follows uniform distribution within the ball $\mathbb{B}^{(d')}_{\bx}(\sR)$, where $\mathbb{B}^{(d')}_{\bx}$ denotes the $d'$-dimensional ball centered at $\bx$ with radius $\sR$ and  $d'=|\mathcal{F}(\bx)|$.

{\bf Step 1:} we will quantify the decrease of the objective value after $T$ number of iterations. Suppose that \asref{as1} is satisfied. Let $\bx$ denote a saddle point which satisfies \conref{cond:saddle}. Consider two sequences generated by SP-GD, i.e., $\{\bu^{(r)}\}$ and  $\{\bw^{(r)}\}$, where the initial points of these two sequences satisfy the conditions \eqref{eq.inicond} as shown in \leref{le.layer32}.

Again, without loss of generality, we assume $\bu^{(1)}=0$ and let $T^*\bydef\widehat{c}\sT$ and $T'\bydef\inf_{r\ge 1}\left\{r|{\hf}(\bu^{(r)})-\hf(\bu^{(1)})\le-2\sF\right\}$. Then, we have the following two cases to analyze the decrease of the objective value.
\begin{enumerate}
\item Case $T'\le T^*$: Applying \leref{le.layer31}, we know that
\begin{align}
\notag
&f(\bx+\bu^{(T')})-f(\bx)-\nabla f(\bx)^{\T}\bu^{(T')}
\le f(\bx+\bu^{(1)})-f(\bx)-\nabla f(\bx)^{\T}\bu^{(1)}-2\sF
\\
\mathop{\le}\limits^{(a)} &\frac{L_1}{2}\|\bu^{(1)}\|^2-2\sF\mathop{\le}\limits^{(b)}-2\sF\label{eq.ftf}
\end{align}
where $(a)$ is true because of the $L_1$-gradient Lipschitz continuity; $(b)$ is true because $\bu^{(1)}=0$.

From \eqref{eq.desbeta}, we know that SP-GD is always decreasing the approximate objective function $\hf$. When $\widehat{c}\ge 1$ for any $T>\widehat{c}\sT=T^*\ge T'$, we have
\begin{align}
\notag
&f(\bx+\bu^{(T)})-f(\bx)-\nabla f(\bx)^{\T}\bu^{(T)}
\le f(\bx+\bu^{(T^*)})-f(\bx)-\nabla f(\bx)^{\T}\bu^{(T^*)}
\\
\le& f(\bx+\bu^{(T')})-f(\bx)-\nabla f(\bx)^{\T}\bu^{(T')}\le -2\sF.\label{eq.ftf1}
\end{align}
Also, since $\bu^{(T')}=\bu^{(T'-1)}-\beta(\nabla f(\bx+\bu^{(T'-1)})-\nabla f(\bx))$, we have $\|\bu^{(T')}\|\le\|\bu^{(T'-1)}\|+\beta L_1\|\bu^{(T'-1)}\|\le4\sS$ by $L_1$-gradient Lipschitz continuity and  applying \eqref{eq.normofu}.

\item Case $T'> T^*$: Applying \leref{le.layer31}, we know that $\|\bu^{(r)}-\bu^{(1)}\|\le 3\sS$ for $r< T^*$.
Define $T''=\inf_{r\ge 1}\left\{r|{f}(\bw^{(r)})-f(\bw^{(1)})\le -2\mathcal{F}\right\}$. Then, after applying \leref{le.layer32}, we know $T''< T^*$. Using the same argument as the above case, for $T\ge \widehat{c}\sT= T^*> T''$, we also have
\begin{align}
\notag
&f(\bx+\bw^{(T)})-f(\bx)-\nabla f(\bx)^{\T}\bw^{(T)}
\le  f(\bw^{(T^*)})-f(\bx)-\nabla f(\bx)^{\T}\bw^{(T^*)}
\\\notag
\le & f(\bw^{(T'')})-f(\bx)-\nabla f(\bx)^{\T}\bw^{(T'')})
\le  f(\bw^{(1)})-f(\bx)-\nabla f(\bx)^{\T}\bw^{(1)}-2\sF
\\
\mathop{\le}\limits^{(a)} & \frac{L_1}{2}\|\bw^{(1)}\|^2-2\sF \mathop{\le}\limits^{(a)}-1.5\sF\label{eq.ftf2}
\end{align}
where in $(a)$ we use the initialization conditions of the iterates shown in \eqref{eq.inicond} in \leref{le.layer32}  so that we have $L_1\sR^2/2\mathop{=}\limits^{\eqref{eq.defoft}}L_1\epsilon^2_H/(2\widehat{c}^8\kappa^2L^2_2\log^4(d\kappa/\delta))\mathop{\le}\limits^{\eqref{eq.defs}}\epsilon^3_H/(2\widehat{c}^8\kappa L^2_2\log^3(d\kappa/\delta))\le\epsilon^3_H/(2L^2_2\widehat{c}^5\log^3(d\kappa/\delta))\mathop{=}\limits^{\eqref{eq.defff}} 0.5\sF$. Also, similar as the previous case, we have $\|\bw^{(T'')}\|\le 4\sS$  since $\|\bw^{(1)}-\bx\|\le2\sR$.

\end{enumerate}

{\bf Step 2:} we show that at least one sequence, i.e., either $\bu^{(r)}$ or $\bw^{(r)}$, will give the sufficient descent of the approximate objetive value after $T$ iterations. Combining \eqref{eq.ftf1} and \eqref{eq.ftf2}, we have
\begin{equation}
\min\left\{f(\bx+\bu^{(T)})-f(\bx)-\nabla f(\bx)^{\T}\bu^{(T)},f(\bx+\bw^{(T)})-f(\bx)-\nabla f(\bx)^{\T}\bw^{(T)}\right\}\le -1.5\sF,
\;\forall T\ge \widehat{c}\sT,\label{eq.suffdes}
\end{equation}
meaning that at least one of the sequences can give a sufficient decrease of the objective function if the initial points of the two sequences are separated apart with each other far enough along the negative curvature direction $\vec{\be}$.

Let $\mathcal{X}_{\textsf{stuck}}$ denote the set where a generic sequence $\bu^{(r)}$ is initialized such that the sequence cannot escape from the strict saddle point after $T$ iterations, i.e., $f(\bx+\bu^{(T)})-f(\bx)-\nabla f(\bx)^{\T}\bu^{(T)}>-1.5\sF$. According to \eqref{eq.suffdes} and \leref{le.layer32}, we can conclude that if $\bu^{(1)}\in\mathcal{X}_{\textsf{stuck}}$, then initialization  $(\bu^{(1)}\pm\upsilon \sR\vec{\be})\notin\mathcal{X}_{\textsf{stuck}}$ where $\upsilon\in[\frac{\delta}{2\sqrt{d}},1]$.

{\bf Step 3:} next, we give the upper
bound of the volume of $\mathcal{X}_{\textsf{stuck}}$,
\begin{align}
\notag
\textsf{Vol}(\mathcal{X}_{\textsf{stuck}})=&\int_{\mathbb{B}^{(d')}_{\bx}}d\bu I_{\mathcal{X}_{\textsf{stuck}}}(\bu)=\int_{\mathbb{B}^{(d'-1)}_{\bx}}d\bu_{-1}\int^{\bx_1+\sqrt{\sR^2-\|\bx_{-1}-\bu_{-1}\|^2}}_{\bx_1-\sqrt{\sR^2-\|\bx_{-1}-\bu_{-1}\|^2}}d\bu_1 I_{\mathcal{X}_{\textsf{stuck}}}(\bu)
\\\notag
\le&\int_{\mathbb{B}^{(d'-1)}_{\bx}}d\bu_{-1}\left(2\frac{\delta}{2\sqrt{d'}\sR}\right)=\textsf{Vol}\left(\mathbb{B}^{(d'-1)}_{\bx}(\sR)\right)\frac{\sR\delta}{\sqrt{d'}}
\end{align}
where $I_{\mathcal{X}_{\textsf{stuck}}}(\bu)$ is an indicator function showing that $\bu$ belongs to set $\mathcal{X}_{\textsf{stuck}}$, and $\bu_1$ represents the component of vector $\bu$ along $\vec{\be}$ direction, and $\bu_{-1}$ is the remaining $d'-1$ dimensional vector.

Then, the ratio of $\textsf{Vol}(\mathcal{X}_{\textsf{stuck}})$ over the whole volume of the initialization/perturbation ball can be upper bounded by
\begin{equation}\notag
\frac{\textsf{Vol}(\mathcal{X}_{\textsf{stuck}})}{\textsf{Vol}(\mathbb{B}^{(d')}_{\bx}(\sR))}\le\frac{\frac{\sR\delta}{\sqrt{d'}}
\textsf{Vol}(\mathbb{B}^{(d'-1)}_{\bx}(\sR))}{\textsf{Vol}(\mathbb{B}^{(d')}_{\bx}(\sR))}
=\frac{\delta}{\sqrt{d\pi}}\frac{\Gamma(\frac{d'}{2}+1)}{\Gamma(\frac{d'}{2}+1)}\le\frac{\delta}{\sqrt{d'\pi}}\sqrt{\frac{d'}{2}+\frac{1}{2}}\le\delta
\end{equation}
where $\Gamma(\cdot)$ denotes the Gamma function, and inequality is true due to the fact that $\Gamma(x+1)/\Gamma(x+1/2)<\sqrt{x+1/2}$ when $x\ge0$.

{\bf Step 4:} finally, we show that the  output of SP-GD can give an approximate eigenvector whose smallest eigenvalue is less than $-\epsilon_H$ with high probability. Combining \eqref{eq.suffdes} and the results of the last step, we can show that
\begin{equation}\label{eq.sufde}
f(\bx+\bz^{(T)})-f(\bx)-\nabla f(\bx)^{\T}\bz^{(T)}\le-1.5\sF
\end{equation}
with at least probability $1-\delta$. By the $L_2$-Lipschitz continuity, we have
\begin{equation}\label{eq.lipsc}
\left|f(\bx+\bz^{(T)})-f(\bx)-\nabla f(\bx)^{\T}\bz^{(T)}-\frac{1}{2}(\bu^{(T)})^{\T}\nabla^2f(\bx)\bu^{(T)}\right|\le\frac{L_2}{6}\|\bz^{(T)}\|^3.
\end{equation}
and $\|\bz^{(T)}\|\le4\sS$.
Applying \eqref{eq.sufde} into \eqref{eq.lipsc}, we have
\begin{align}
\notag
\frac{1}{2}(\bu^{(T)})^{\T}\nabla^2f(\bx)\bu^{(T)}\le &f(\bx+\bz^{(T)})-f(\bx)-\nabla f(\bx)^{\T}\bz^{(T)}+\frac{L_2}{6}\|\bz^{(T)}\|^3
\\
\mathop{\le}\limits^{(a)} & -1.5\sF+0.5\sF\le -\sF
\end{align}
where in $(a)$ we use \eqref{eq.defff}\eqref{eq.defs} so that we have $\widehat{c}L_2\sS^3=\sF$ where $\widehat{c}\ge51$.
Therefore, we have
\begin{equation}\label{eq.negav}
\frac{(\bu^{(T)})^{\T}\nabla^2f(\bx)\bu^{(T)}}{\|\bz^{(T)}\|^2}\le\frac{-2\sF}{(4\sS)^2}\mathop{\le}\limits^{\eqref{eq.defff},\eqref{eq.defs}}-\frac{\epsilon_H}{8\widehat{c}\log(d\kappa/\delta)}
\end{equation}
so that we can claim that if SP-GD returns $\Diamond$ then with probability $1-\delta$ the output $\bz^{(T)}$ holds for \eqref{eq.negav}, otherwise SP-GD returns $\emptyset$ which indicates that $\lambda_{\min}(\bH_{\bP}(\bx))\ge -\epsilon_H$ with probability $1-\delta$.
\end{proof}

\subsection{Proof of \coref{co:th2}}\label{sec:th2}
\begin{proof}
In this section, we give the proof of finding $(\epsilon,\sqrt{\epsilon})$-SOSP1 by SOSP$^+$. The main difference between SNAP and SNAP$^+$ is that we replace the oracle $\textsf{\it Negative-Eigen-Pair}$ by SP-GD. Other steps are the same as the proof of \thref{th.1}. Here, we only focus on the difference of the objetive reduction between SNAP and SNAP$^+$. First, let the number of iterations run by SP-GD for extracting the negative curvature once be \begin{equation}
    T_{\textsf{SP-GD}}=\frac{\widehat{c}\log(\frac{dL_1}{\epsilon_H\delta})}{\beta\epsilon_H}+1\sim\mathcal{O}\left(\frac{1}{\epsilon_H}\right).
\end{equation}
In the following, we show the objective reduction in the NCD step, where the number of iterations required in the inner loop is taken into account.\\
\noindent{\bf Case 1) ($\textsf{flag}_{\alpha}=\emptyset$)}: The algorithm implements $\bx^{(r+1)}=\bx^{(r)}+\alpha^{(r)}\bd^{(r)}$ without using $\alpha^{(r)}_{\max}$ computed by \eqref{eq.alphamax}.
By \eqref{eq.sufd:3}, we have the descent of the objective value by
$$f(\bx^{(r+r_{\textsf{th}})})\le f(\bx^{(r)})-\frac{\Delta}{T_{\textsf{SP-GD}}}.$$\\
\noindent{\bf Case 2) ($\textsf{flag}_{\alpha}=\Diamond$)}: $\alpha^{(r)}_{\max}$ is computed by \eqref{eq.alphamax} to update $\bx^{(r+1)}$;
By \eqref{eq.descdm2}, we have
\begin{align}\notag
f(\bx^{(r)}) - f(\bx^{(r-\min\{d,m\})}) <  - \min \left\{\epsilon^2_G/(18L_1), 0.06\epsilon'^3_H(\delta)/(L^2_2T_{\textsf{SP-GD}})\right\},\forall r>\min\{d,m\}.
\end{align}
Applying the same argument from \eqref{eq.tele1} to \eqref{eq.ruppbd}, we know the upper bound of the number of iterations by $ \frac{(f(\bx^{(1)})-f^{\star})T_{\textsf{SP-GD}}}{\Delta'}$. From \thref{th:spgd1}, we know that $ \epsilon'_H(\delta)=\frac{\epsilon_H}{8\widehat{c}\log(d\kappa/\delta)}$, i.e., $\gamma=8\widehat{c}\log(d\kappa/\delta)$. Applying $\epsilon_G=\epsilon$, $\epsilon_H=\sqrt{L_2\epsilon}$ and $\beta\le1/L_1$ (note that $r_{\textsf{th}}$ could be either a constant or chosen in the order of $\mathcal{O}(L_1/\sqrt{L_2\epsilon})$), we can obtain the convergence rate of SNAP$^+$ by
\begin{equation}
{\widehat{\mathcal{O}}}\left(\frac{\min\{d,m\}(f(\bx^{(1)})-f^{\star})}{\min\left\{\frac{\epsilon^{2.5}L_2^{1/2}}{L^2_1},\frac{\epsilon^2}{L_1}\right\}}\right),
\end{equation}
which completes the proof.
\end{proof}
\remark SP-GD in SNAP$^+$ is not needed for every step. Also, the accelerated version of SP-GD and PGD (e.g., by incorporating Nesterov acceleration technique) can be used such that we can have a faster convergence rate of SNAP$^+$.


\end{document}